\documentclass[12pt,reqno]{amsart}
\usepackage{mathrsfs}
\usepackage{cases}
\usepackage{epic}
\usepackage{amsfonts}
\usepackage{graphicx}
\usepackage{amsmath}
\usepackage{amssymb, upgreek}
\usepackage{bm}
\usepackage{latexsym,todonotes}
\usepackage{pdflscape}
\usepackage[all]{xypic}
\usepackage[all]{xy}
\usepackage{color}
\usepackage{colordvi}
\usepackage{multicol}
\usepackage[normalem]{ulem}
\usepackage[mathscr]{euscript}
\input diagxy
\usepackage[linktocpage=true]{hyperref}
\hypersetup{colorlinks,linkcolor=blue,urlcolor=cyan,citecolor=blue}

\usetikzlibrary{
	cd,
	shapes,
	arrows,
	positioning,
	decorations.markings,
	decorations.pathmorphing,
	circuits.logic.US,
	circuits.logic.IEC,
	fit,
	calc,
	plotmarks,
	matrix
}
\begin{document}
	\input xy
	\xyoption{all}

	\newtheorem{innercustomthm}{{\bf Theorem}}
	\newenvironment{customthm}[1]
	{\renewcommand\theinnercustomthm{#1}\innercustomthm}
	{\endinnercustomthm}
	
	\newtheorem{innercustomcor}{{\bf Corollary}}
	\newenvironment{customcor}[1]
	{\renewcommand\theinnercustomcor{#1}\innercustomcor}
	{\endinnercustomthm}
	
	\newtheorem{innercustomprop}{{\bf Proposition}}
	\newenvironment{customprop}[1]
	{\renewcommand\theinnercustomprop{#1}\innercustomprop}
	{\endinnercustomthm}

	\newtheorem{theorem}{Theorem}[section]
	\newtheorem{acknowledgement}[theorem]{Acknowledgement}
	\newtheorem{algorithm}[theorem]{Algorithm}
	\newtheorem{axiom}[theorem]{Axiom}
	\newtheorem{case}[theorem]{Case}
	\newtheorem{claim}[theorem]{Claim}
	\newtheorem{conclusion}[theorem]{Conclusion}
	\newtheorem{condition}[theorem]{Condition}
	\newtheorem{conjecture}[theorem]{Conjecture}
	\newtheorem{construction}[theorem]{Construction}
	\newtheorem{corollary}[theorem]{Corollary}
	\newtheorem{criterion}[theorem]{Criterion}
	\newtheorem{definition}[theorem]{Definition}
	\newtheorem{example}[theorem]{Example}
	\newtheorem{exercise}[theorem]{Exercise}
	\newtheorem{lemma}[theorem]{Lemma}
	\newtheorem{notation}[theorem]{Notation}
	\newtheorem{problem}[theorem]{Problem}
	\newtheorem{proposition}[theorem]{Proposition}
	\newtheorem{solution}[theorem]{Solution}
	\newtheorem{summary}[theorem]{Summary}
	\numberwithin{equation}{section}
	
	\theoremstyle{remark}
	\newtheorem{remark}[theorem]{Remark}
	
	\makeatletter
	\newcommand{\rmnum}[1]{\romannumeral #1}
	\newcommand{\Rmnum}[1]{\expandafter\@slowromancap\romannumeral #1@}
	\def \g{\mathfrak{g}}
	\def \bH{{\mathbf H}}
	\def \bB{{\mathbf B }}
	\def \BKH{{\acute{H}}}
	\def \bBKH{{\acute{\mathbf H}}}
	\def \bK{\mathbb{K}}
	\def \bC{\mathbb{\bK_\de}}
	\def \bQ{\mathbb{Q}}
	\def \Z{\mathbb{Z}}
	\def \I{\mathbb{I}}
	\def \ch{{\mathcal H}}
	\def \cm{{\mathcal M}}
	\def \ct{{\mathcal T}}
	\def \hW{\widehat{W}}
	\def \bZ{\mathbb{Z}}
	\def \TH{\Theta}
	\def \tMH{{\cm\widetilde{\ch}(\Lambda^\imath)}}
	\def \tMHl{{\cm\widetilde{\ch}(\bs_\ell\Lambda^\imath)}}
	\def \bTH { \boldsymbol{\Theta}}
	\def \bDel{ \boldsymbol{\Delta}}
	\def \tY{\widetilde{Y}}
	\def \tH{{\widetilde{H}}}
	\def \tTH{\widetilde{\Theta}}
	\def \btau{\widehat{\tau}}
	\def \s{\varsigma}
	\def \bvs{\boldsymbol{\varsigma}}
	\def \bs{\mathbf r}
	\def \cR{\mathcal{R}}
	\def \fg{\mathfrak{g}}
	\def \Aut{\operatorname{Aut}\nolimits}
	
	\newcommand{\hgt}{\text{ht}}
	\newcommand{\dc}{\dot{\texttt c}}
	\def \ad{\text{ad}\,}
	\def \gr{\text{gr}\,}
	\def \hg{\widehat{\g}}
	\def \mf{\mathfrak}
	\def \N{\mathbb N}
	\def \tk{\widetilde{k}}
	\def \brW{{\rm Br}(W_{\tau})}
	\def \bome{\bm{\omega}}
	\def \bth{\bm{\theta}}
	\def\Br{{\rm Br}}
	\renewcommand{\t}{\boldsymbol{\omega}}
	
	\def \tfX{\widetilde{ \Upsilon}}
	\def \tT{\widetilde{T}}
	\def \tTT{ \widetilde{\mathbf{T}}}
	
	\def \tcT{\widetilde{\mathscr T}}
	\def \E{\widetilde{E}}
	
	\def \SS{\mathbb{S}}
	\def \sll{\mathfrak{sl}}
	\def \tUD{{}^{Dr}\tU}

	\newcommand{\UU}{{\mathbf U}\otimes {\mathbf U}}
	\newcommand{\UUi}{(\UU)^\imath}
	\newcommand{\tUU}{{\tU}\otimes {\tU}}
	\newcommand{\tUUi}{(\tUU)^\imath}
	\newcommand{\tK}{\widetilde{K}}
	\newcommand{\tU}{\widetilde{{\mathbf U}} }
	\newcommand{\tUi}{\widetilde{{\mathbf U}}^\imath}
	\newcommand{\tUiii}{\tUi(\widehat\sll_3,\tau)}
	\newcommand{\tUiD}{{}^{\text{Dr}}\widetilde{{\mathbf U}}^\imath}
	\newcommand{\tUigr}{{\text{gr}}\tUi}
	\newcommand{\tUigrp}{{\text{gr}}\widetilde{{\mathbf U}}^{\imath,+}}
	\newcommand{\tUiDgr}{{\text{gr}}\tUiD}

	\newcommand{\sqq}{{\bf v}}
	\newcommand{\sqvs}{\sqrt{\vs}}
	\newcommand{\dbl}{\operatorname{dbl}\nolimits}
	\newcommand{\swa}{\operatorname{swap}\nolimits}
	\newcommand{\Gp}{\operatorname{Gp}\nolimits}
	\newcommand{\Sym}{\operatorname{Sym}\nolimits}
	\newcommand{\qbinom}[2]{\begin{bmatrix} #1\\#2 \end{bmatrix} }
	
	\def \ov{\overline}
	\def \balpha{{\boldsymbol \alpha}} 
	\def \K{\mathbb K}
	
	\newcommand{\nc}{\newcommand}
	\nc{\browntext}[1]{\textcolor{brown}{#1}}
	\nc{\greentext}[1]{\textcolor{green}{#1}}
	\nc{\redtext}[1]{\textcolor{red}{#1}}
	\nc{\bluetext}[1]{\textcolor{blue}{#1}}
	\nc{\brown}[1]{\browntext{ #1}}
	\nc{\green}[1]{\greentext{ #1}}
	\nc{\red}[1]{\redtext{ #1}}
	\nc{\blue}[1]{\bluetext{ #1}}

	\def \Q {\mathbb Q}
	\def \C{\mathbb C}
	\def \TT{\mathbf T}
	\def \tt{\mathbf t}
	\newcommand{\wt}{\text{wt}}
	\def \de{\delta}
	\def \bvs{{\boldsymbol{\varsigma}}}
	\def \vs{\varsigma}
	\def \U{\mathbf U}
	\def \Ui{\mathbf{U}^\imath}
	\def \dm{\diamond}
	\def \bS{\mathbb S}
	
	\def \bR{\mathbb R}
	\def \bP{\mathbb P}
	\def \bF{\mathbb F}
	\def \II{\mathbb{I}_0}
	\allowdisplaybreaks
	
	\newcommand{\lutodo}{\todo[inline,color=violet!20, caption={}]}
	\newcommand{\wtodo}{\todo[inline,color=cyan!20, caption={}]}
	\newcommand{\ztodo}{\todo[inline,color=green!20, caption={}]}
	
	\title[Quasi-split affine $\imath$quantum groups]
	{Braid group action and quasi-split affine $\imath$quantum groups I}
	
	\author[Ming Lu]{Ming Lu}
	\address{Department of Mathematics, Sichuan University, Chengdu 610064, P.R.China}
	\email{luming@scu.edu.cn}

	\author[Weiqiang Wang]{Weiqiang Wang}
	\address{Department of Mathematics, University of Virginia, Charlottesville, VA 22904, USA}
	\email{ww9c@virginia.edu}

	\author[Weinan Zhang]{Weinan Zhang}
	\address{Department of Mathematics, University of Virginia, Charlottesville, VA 22904, USA}
	\email{wz3nz@virginia.edu}


	\begin{abstract}
		This is the first of our papers on quasi-split affine quantum symmetric pairs $\big(\widetilde{\mathbf U}(\widehat{\mathfrak g}), \widetilde{{\mathbf U}}^\imath \big)$,  focusing on the real rank one case, i.e., $\mathfrak g = \mathfrak{sl}_3$ equipped with a diagram involution. We construct explicitly a relative braid group action of type $A_2^{(2)}$ on the affine $\imath$quantum group $\widetilde{{\mathbf U}}^\imath$. Real and imaginary root vectors for $\widetilde{{\mathbf U}}^\imath$ are constructed, and a Drinfeld type presentation of $\widetilde{{\mathbf U}}^\imath$ is then established. This provides a new basic ingredient for the Drinfeld type presentation of higher rank quasi-split affine $\imath$quantum groups in the sequels.
	\end{abstract}
	
	\maketitle
	\setcounter{tocdepth}{1}
	\tableofcontents

	\section{Introduction}
	\subsection{Background}
	
	Besides the Drinfeld-Jimbo presentation, affine quantum groups admit a second presentation due to Drinfeld \cite{Dr87, Dr88, Da93, Be94, Da15}, which is a remarkable quantum analog of the current presentation of affine Lie algebras. The Drinfeld (current) presentation has played a fundamental role in representation theory and mathematical physics; cf., e.g.,  \cite{CP91, FR99, FM01}, and see the survey \cite{CH10} for extensive references.
	
	The $\imath$quantum groups $\Ui$ arising from quantum symmetric pairs $(\U, \Ui)$ associated to Satake diagrams \cite{Let99} (see \cite{Ko14}) can be viewed as a vast generalization of Drinfeld-Jimbo quantum groups associated to Dynkin diagrams; see the survey \cite{W22} and references therein. In this paper, we shall work with {\em universal} $\imath$quantum groups $(\tU, \tUi)$ following \cite{LW22, WZ22}, where $\tU$ is the Drinfeld double quantum group, as this allows us to formulate the relative braid group action conceptually. Just as the quantum group $\U$ is obtained by $\tU$ by a central reduction,
	a central reduction from universal $\imath$quantum groups recovers $\imath$quantum groups with parameters \cite{Let02, Ko14}.
	
	In \cite{LW21a}, two authors of this paper obtained a Drinfeld type presentation for affine $\imath$quantum groups of split ADE type; here ``split" means that the underlying Satake diagram contains only white nodes and a trivial diagram involution, and this class of algebras appeared in \cite{BB10} in connection to boundary affine Toda field theories. Subsequently, the third author streamlined some of the main arguments in the split ADE type and succeeded in generalizing to the split BCFG type \cite{Z21}.
	
	The affine $\imath$quantum group of split rank one is also known as the q-Onsager algebra in the mathematics and physics literature; cf. \cite{T18, BB10} and references therein. The real and imaginary root vectors for q-Onsager algebra were first constructed and their commutator relations in a somewhat tedious form were also obtained in \cite{BK20}. These relations were transformed and upgraded in  \cite{LW21a} into a Drinfeld type relation for the (universal) q-Onsager algebra. These constructions on q-Onsager algebra were instrumental for the Drinfeld presentation of split affine $\imath$quantum groups of higher rank.

	\subsection{Goal}
	
	In this paper and its sequels, we shall construct the Drinfeld type presentations of {\em quasi-split} affine $\imath$quantum groups, where the Satake diagrams are of affine types $\widehat A_n, \widehat D_n, \widehat E_6$ with nontrivial diagram involutions fixing the affine simple root.
	This class of $\imath$quantum groups contains 3 distinct affine $\imath$quantum groups of real rank one:
	\begin{enumerate}
		\item[(i)] Drinfeld-Jimbo quantum group $\tU(\widehat\sll_2)$;
		\item[(ii)] q-Onsager algebra $\tUi(\widehat\sll_2)$;
		\item[(iii)]  the $\imath$quantum group $\tUiii$ associated to the Satake diagram $(\I, \tau)$ in \eqref{eq:satakerank1}, where $\tau$ is the involution on $\I =\{0,1,2\}$ of affine type $\widehat A_2$
		\[
		\tau: \I \longrightarrow \I,
		\qquad\qquad
		\tau(0) =0, \qquad 1 \stackrel{\tau}{\longleftrightarrow} 2.
		\]
	\end{enumerate}
	
	The goal of this paper is to give a Drinfeld type presentation for $\tUiii$, the case (iii) in the above list of 3 affine rank one types. Toward this goal, we shall also establish the relative braid group action of twisted affine type $A_2^{(2)}$ on $\tUiii$, which is another main result of this paper. The constructions in this paper (as well as the known Drinfeld presentations in Cases (i)--(ii)) will play a basic role in the sequels, in which a Drinfeld type presentation for arbitrary quasi-split affine $\imath$quantum groups will be established.
	
	The Drinfeld type presentations of affine $\imath$quantum groups are expected to play a foundational role in their representation theory, to which we shall return elsewhere. They may have additional applications to quantum integrable systems (such as XXZ spin chain, Sine-Gordon and Liouville field theories), cf. \cite{BK05}.

	\subsection{Features of new affine rank one}
	
	There are several reasons why the algebra $\tUiii$ is substantially more involved than the other 2 affine rank one types and deserves a separate new investigation.
	
	While $\tUiii$ is viewed as a new case (iii) of affine real rank one (a building block which cannot be further reduced), it is a subalgebra of the quantum group $\tU(\widehat\sll_3)$ and behaves with complexity of affine rank two, in contrast to Cases (i)--(ii). As we shall see, the Drinfeld presentation of $\tUiii$ requires \red{2} infinite series of real (respectively, imaginary) roots and Serre relations, and moreover, $\tUiii$ admits a braid group action of twisted affine type $A_2^{(2)}$.
	
	To construct Drinfeld presentations of all affine quantum groups including twisted types, besides the usual type $\widehat  A_1$, one needs to treat the twisted affine rank one type $A_{2}^{(2)}$ separately; see \cite{Da00, Da15}. In the study of geometry of twisted affine Grassmannians, there are two special parahoric groups for $A_{2r}^{(2)}$, one of which (known as {\em absolutely special} or {\em hyperspecial}) requires a separate treatment. The quasi-split affine $\imath$quantum groups of type  $\widehat A_{n}$ with nontrivial diagram involution (which is $\tUiii$ if $n=2$) has been realized geometrically in \cite{FLLLW}, and its $\imath$canonical basis admits favorable positivity properties (compare \cite{Lus93}).
	
	The $\imath$Hall algebra constructions based on $\imath$quivers or $\imath$weighted projective lines (cf. \cite{LW22, LR21}) can be used to realize all affine quasi-split $\imath$quantum groups except $\widehat A_{2r}$ with nontrivial diagram involution. The simplest affine case not covered by the current $\imath$Hall algebra approach is exactly $\tUiii$. So unlike the split ADE types \cite{LW21b, LRW23, LR21}, we do not have access to a Hall algebra construction to gain insights into the Drinfeld type presentation for $\tUiii$.

	
	
	We shall denote $\tUiii$ by $\tUi$ in the remainder of this paper.

	\subsection{Relative braid group symmetries}
	
	The algebra $\tUi$ is generated by $B_i, \K_i$, for $i\in\I =\{0,1,2\}$, subject to Serre type  relations~\eqref{kB}--\eqref{Bi00}.
	The relative root system for $(\tU, \tUi)$, which contains 2 simple roots $\balpha_0, \balpha_1$, is of twisted affine type $A_2^{(2)}$. Very recently, two of the authors in \cite{WZ22} constructed relative braid group symmetries on $\tUi$ of arbitrary finite type, confirming a longstanding conjecture of Kolb-Pellegrini \cite{KP11} and generalizing braid group symmetries on quantum groups \cite{Lus93}.
	It is natural to hope for automorphisms $\TT_0, \TT_1$ on $\tUi$ which generate a relative braid group action of type $A_2^{(2)}$.
	
	The automorphisms $\TT_0$ and $\TT_0^{-1}$ with respect to the simple root $\balpha_0$ are not difficult to construct; see Proposition~\ref{prop:T0}. Actually, the general constructions of relative braid group symmetries in \cite{WZ22} can be applied to cover this case (even though our $\tUi$ is not of finite type), and the formulas for the actions on generators of $\tUi$ essentially arise from finite type considerations (compare \cite{KP11, LW21a}).
	
	In contrast, the automorphisms $\TT_1$ and $\TT_1^{-1}$ with respect to $\balpha_1$ are  difficult to establish.
	
	\begin{customthm} {\bf A}
		[Theorems~\ref{thm:T-1}--\ref{thm:T1}]
		\label{thm:A}
		There exists a $\Q(v)$-algebra automorphism $\TT_1$ on $\tUi$ such that
		\begin{align*}
			\TT_1(\K_1)&= v^{-1} \K_2^{-1},\quad \TT_1(\K_2)= v^{-1} \K_1^{-1},\quad \TT_1(\K_0)= v^{2} \K_0 \K_1^2\K_2^2, \\
			\TT_1(B_1)&=-v^{-2}B_1\K_2^{-1} ,\qquad \TT_1(B_2)=-v^{-2}B_2\K_1^{-1} ,\\
			\TT_1(B_0)&=
			v\Big[\big[ [B_0,B_1]_v,B_2\big] ,[B_2,B_1]_v\Big]- \big[B_0, [B_2,B_1]_{v^3} \big]\K_1+ v B_0 \K_1\K_2.
		\end{align*}
		The inverse automorphism $\TT_1^{-1}$ is given by
		\begin{align*}
			\TT_1^{-1}(\K_1)&= v^{-1} \K_2^{-1},\quad \TT_1^{-1}(\K_2)= v^{-1} \K_1^{-1},\quad \TT_1^{-1}(\K_0)= v^{2} \K_0 \K_1^2\K_2^2,
			\\
			\TT_1^{-1}(B_1)&=-vB_1 \K_1^{-1},\qquad \TT_1^{-1}(B_2)=-vB_2 \K_2^{-1},
			\\
			\TT_1^{-1}(B_0)&=
			v\Big[[B_1,B_2]_v,\big[ B_2,[B_1,B_0]_v\big] \Big]- \big[[B_1,B_2]_{v^3},B_0 \big]\K_2+ v B_0 \K_1\K_2.
		\end{align*}
		Moreover, $\TT_1$ and $\TT_1^{-1}$ are related by the $\sigma_{\tau}$-conjugation:
		$\TT_1^{-1}= \sigma_{\tau} \circ \TT_1 \circ \sigma_{\tau}.$
		(For the anti-involution $\sigma_{\tau}$,
		see Lemma~\ref{lem:sigma}.)
	\end{customthm}
	
	Let us discuss the obstacles about $\TT_1$. One has to guess the explicit formulas for the action of $\TT_1$ on generators of $\tUi$, and then to show it is indeed an automorphism. The action of $\TT_1$ on most generators of $\tUi$ can be formulated without much trouble, except $\TT_1(B_0)$. Earlier on we guessed a formula for $\TT_1(B_0)$ (as a polynomial in $B_i$ of degree 5), based on the internal consistency and desired properties of root vectors. (Recall \cite{Lus93} a particular braid group operator formula of type $A_2^{(2)}$ is of degree $5$ when the relevant Cartan integer is $-4$.) However, in contrast to the q-Onsager algebra case as done in \cite{BK20} (see also \cite{T18}), it turned out to be too difficult for a (super) computer to verify that $\TT_1$ is an automorphism of $\tUi$ (e.g., that the Serre relations for $\tUi$ is preserved by $\TT_1$).
	
	The conceptual and general approach toward relative braid group action on $\imath$quantum groups developed in \cite{WZ22} is conjectured to be valid for Kac-Moody type. One advantage of this approach is a built-in mechanism for proving $\TT_1$ is an automorphism of $\tUi$. We follow the strategy {\em loc. cit.} to define $\TT_1$ via a certain rescaled braid operator on the Drinfeld double $\tU$ and the quasi $K$-matrix of type AIII$_2$ \cite{BW18} (with an explicit formula given in \cite{DK19}). Accordingly, we obtain formulas for $\TT_1$ on $B_1, B_2$ and $\K_i$, for $i\in \I$. It takes substantial computations however to make this approach work to produce a neat closed formula for $\TT_1(B_0)$ in Theorem~\ref{thm:A} (which in particular asserts that $\TT_1(B_0) \in\tUi$). This in particular verifies \cite[Conjecture~ 5.13]{WZ22} (formulated for $\tUi$ of Kac-Moody type) in the first new case beyond finite types.

	\subsection{Drinfeld type presentation of $\tUi$}
	
	Using the ``translation" braid group operator
	\[
	\TT_{\bome} :=\TT_0 \TT_1,
	\]
	we  define the {\em real} root vectors $B_{i,k}$ in \eqref{eq:Bik}, for $i\in \{1,2\}$ and $k\in \Z$ (cf. \cite{Da93, BK20}). On the other hand, we define inductively in \eqref{TH}--\eqref{THn} the {\em imaginary $v$-root vectors} $\TH_{i,m}$ via commutators between real root vectors, for $i\in \{1,2\}$ and $m\ge 1$. The definition of $\TH_{i,m}$ is by no means obvious; see Remark~\ref{rem:imagine}. Sometimes, it is more convenient to work with a new set of imaginary root vectors $H_{i,m}$; see \eqref{exp h} for its relation to $\TH_{i,m}$.
	
	The Drinfeld type presentation of $\tUi$ is built on the real and imaginary root vectors; compare \cite{LW21b, Z21} in split affine types for similarities and differences. These earlier works (and also \cite{BK20}) help us to formulate the relations in Theorem~\ref{thm:B} below. However, the proof of this theorem remains challenging due to the affine rank two complexity of $\tUi$. 
We shall need some notations in order to formulate Theorem \ref{thm:B}. We shall denote $[A,B]_{v^a} =AB -v^aBA$. The shorthand notion $\SS(k_1,k_2\mid l;i)$ is defined in \eqref{eq:SS} while the definition of the symmetrization $\Sym_{k_1,k_2}$ can be found in \S\ref{subsec:sym}. 
	
	\begin{customthm} {\bf B}
		[Definition~\ref{def:iDR}, Theorem~\ref{thm:Dr1}]
		\label{thm:B}
		The $\Q(v)$-algebra $\tUi$ has a presentation with generators $B_{i,l}$, $H_{i,m}$, $\bK_i^{\pm1}$, $C^{\pm1}$, where $i\in \{1,2\}$, $l\in\Z$ and $m \in \Z_{\ge 1}$, subject to the following relations: for $m, n \ge 1, k_1, k_2, k, l \in \Z$, and $i, j \in \{1,2\}$,
		\begin{align}
			C \text{ is central,} \quad &
			\K_i\K_j=\K_j\K_i, \quad
			\K_i H_{j,m}=H_{j,m}\K_i,\quad
			\bK_iB_{j,l}=v^{c_{\tau i,j}-c_{ij}} B_{j,l} \bK_i,
			\label{R1} \\
			[H_{i,m},H_{j,n}] &=0,
			\label{R2} \\
			[H_{i,m},B_{j,l}] &=\frac{[mc_{ij}]}{m} B_{j,l+m}-\frac{[mc_{\tau i,j}]}{m} B_{j,l-m}C^m,
			\label{R3} \\
			[B_{i,k},B_{i,l+1}]_{v^{-2}} & -v^{-2}[B_{i,k+1},B_{i,l}]_{v^{2} }=0,
			\\
			[B_{i,k},B_{\tau i,l+1}]_v & -v[B_{i,k+1},B_{\tau i,l}]_{v^{-1}} = -\Theta_{{\tau i},l-k+1}C^k \bK_{i} +v \Theta_{ {\tau i},l-k-1}C^{k+1}\bK_{i}
			\notag \\
			& \qquad\qquad\qquad\qquad\quad\;\,  -\Theta_{i,k-l+1}C^l\bK_{{\tau i}} +v \Theta_{i,k-l-1}C^{l+1}\bK_{\tau i},
			\\  \mathbb{S}(k_1,k_2\mid l;i )
			= [2]&\Sym_{k_1,k_2}\sum_{p\geq 0}v^{2p}
            \big[\TH_{\tau i,l-k_2-p}\K_i-v\TH_{\tau i,l-k_2-p-2}C\K_i, B_{i,k_1-p} \big]_{v^{-4p-1}}C^{k_2+p}
            \notag 
            \\
             +v[2]&\Sym_{k_1,k_2}\sum_{p\geq 0}v^{2p} \big[ B_{i,k_1+p+1},\TH_{i,k_2-l-p+1}\K_{\tau i}-v\TH_{i,k_2-l-p-1}C \K_{\tau i}\big]_{v^{-4p-3}} C^{l-1}.
			\label{R6}
		\end{align}
	\end{customthm}
	If we set all the summands involving $C$ to zero in the above relations, the above presentation is essentially reduced to the Drinfeld presentation for half of $\tU(\sll_3)$.
	Theorem~\ref{thm:B} admits a generating function reformulation in terms of $\bB_{i}(z)$, $\bTH_{i}(z)$, $\bH_i(z)$ $(i=1,2)$ and $\bDel(z)$ from \eqref{eq:Genfun}; see Theorem~\ref{thm:DrqsA1}.
	
	Denoting by $\tUiD$ the $\Q(v)$-algebra with generators and relations given in Theorem~\ref{thm:B}, we are reduced to establish an algebra isomorphism $\Phi: \tUiD \longrightarrow\tUi$, which matches generators in the same notations (for $\tUi$ they stand for the root vectors). Assume that $\Phi$ is a homomorphism for now. One shows that $\Phi$ is surjective by checking all generators of $\tUi$ lie in the image of $\Phi$. The $\imath$quantum group $\tUi$ is a filtered algebra with its associated graded algebra isomorphic to $ \U^- \otimes \Q(v)[\K_i^{\pm 1} \mid i\in \I]$; see Proposition \ref{prop:graded}. 
 The injectivity of $\Phi$ is then reduced by some filtration arguments to the corresponding isomorphism for the Drinfeld presentation of $\widehat\sll_3$.
	
	It remains to show that $\Phi$ is a homomorphism, which is the most involved part of the proof, i.e., to verify all the relations stated in Theorem~\ref{thm:B} are satisfied by the root vectors in $\tUi$. As explained in \S\ref{subsec:strategy}, the overall strategy of the verification of the relations is an inductive argument which goes like a big spiral. The earlier approaches in \cite{Da93, BK20} for affine quantum group $\tU(\widehat\sll_2)$ and q-Onsager algebra $\tUi(\widehat\sll_2)$ have provided us a helpful roadmap but we have to deal with additional complexity of affine rank two.
	It is worth noting that some crucial proofs here follow more closely the approach in \cite{Z21} (instead of \cite{LW21b}), especially in establishing the relation \eqref{R3} and the Serre relation \eqref{R6}. Along the way, we establish the $\TT_{\bome}$-invariance of $\TH_{i,n}$ below, which is intimately related to the commutativity among $\TH_{i,n}$, for all $i=1,2$ and $n\ge 1$.
	
	\begin{customthm} {\bf C}
		[Theorem~\ref{thm:fix1}]
		\label{thm:C}
		We have
		\[
		\TT_{\bome} (\TH_{i,n}) =\TH_{i,n},
		\quad \text{ for all } n \ge 1 \text{ and } i=1,2.
		\]
	\end{customthm}

	\subsection{The organization}	
	
	The paper is organized as follows. In Section~\ref{sec:braid},
	we formulate the algebra $\tUi$ and some basic properties. We then formulate the relative braid group symmetries $\TT_0^{\pm 1}$ and $\TT_1^{\pm 1}$ on $\tUi$. The detailed proof of Theorem~\ref{thm:A} regarding $\TT_1^{\pm 1}$ is given in Section~\ref{sec:proofT1}. Some additional technical computations for an identity used in the proof of Theorem~\ref{thm:A} can be found in Appendix~\ref{App:A}.
	
	In Section~\ref{sec:rtvector}, we construct the real and imaginary $v$-root vectors for $\tUi$. We further identify the classical limits as $v\rightsquigarrow 1$ of these $v$-root vectors and their Drinfeld type relations.
	
	In Section~\ref{sec:Dr1}, we formulate the Drinfeld type presentation for $\tUi$, and prove Theorem~\ref{thm:B} except the verification of the relations. The lengthy proofs of the relations are given in Section~\ref{sec:verify1}. Theorem~\ref{thm:C} is proved along the way.

	\vspace{2mm}
	\noindent {\bf Acknowledgement.}
	ML is partially supported by the National Natural Science Foundation of China (No. 12171333). WW is partially supported by the NSF grant DMS-2001351. WZ is supported by a GSAS fellowship at University of Virginia and WW's NSF Graduate Research Assistantship. We thank Zac Carlini (whose undergraduate research is supported by WW's NSF grant) for his helpful explorations with computer computation (SAGE, FELIX, UVA Rivanna cluster) on braid group symmetries at an early stage of this project. We thank an anonymous referee for many comments and suggestions which help to improve the exposition.

	\section{Relative braid group action}
	\label{sec:braid}

	\subsection{Quantum groups}	
	
	Let $\I=\{0, 1, 2\}$ and $(c_{ij})_{i,j\in \I}$ be the Cartan matrix of affine type $\widehat A_2$. Let $\g:=\mathfrak{sl}_3$ be the simple Lie algebra of type $A_2$ corresponding to $\II =\{1,2 \}$, and $\hg$ be the affine Lie algebra of affine type $A_2^{(1)}$.
	Let $\{\alpha_i\mid i\in\I\}$ be the simple roots of the affine Lie algebra $\hg$, and
	\[
	\de =\alpha_0 +\alpha_1+\alpha_2
	\]
	be the basic imaginary root. Let $\Z\I:=\bigoplus_{i\in\I}\Z\alpha_i$ be the root lattice with a symmetric bilinear form
	\begin{align}
		\label{BF}
		(\cdot, \cdot ): \Z\I \times \Z\I \longrightarrow \Z,
		\qquad (\alpha_i, \alpha_j) =c_{ij}.
	\end{align}
	Let $W$ be the affine Weyl group of type $\widehat A_2$. Let $P=\Z \omega_1\oplus \Z \omega_2$ be the weight lattice for $\g$ where $\omega_i,i=1,2$ are the fundamental weights. Let $ \widehat{W} := S_3 \ltimes P$ be the extended affine Weyl group, where $S_3$ denotes the symmetric group of 3 letters.
	
	Let $v$ be the quantum parameter. For $n\in \Z, r\in \N$, denote the quantum binomial coefficients by
	\[
	[n] =\frac{v^n -v^{-n}}{v-v^{-1}},\qquad [r]^{!}= [r]_v^! =\prod_{i=1}^r [i]_v,
	\qquad\qbinom{n}{r} =\frac{[n][n-1]\ldots [n-r+1]}{[r]!}.
	\]
 
	For $A, B$ in a $\Q(v)$-algebra, we shall denote $[A,B]_{v^a} =AB -v^aBA$, and $[A,B] =AB - BA$. The Drinfeld double quantum group $\tU =\tU(\widehat{\g})$ is generated by $E_i, F_i, K_i^{\pm 1}, K_i'^{\pm 1}$, for $i\in \I$, subject to the relations $K_i K_j=K_jK_i, K_i' K_j'=K_j'K_i', K_i K_j'=K_j'K_i,$ $K_i E_j =v^{c_{ij}} E_j K_i, K_i F_j =v^{-c_{ij}} F_j K_i,$ $K_i' E_j =v^{-c_{ij}} E_j K_i',$  $K_i' F_j =v^{c_{ij}} F_j K_i'$,
	and standard quantum Serre relations. Note that $K_i K_i'$ are central in $\tU$. The Drinfeld-Jimbo quantum group $\U =\U(\widehat{\g})$ is recovered from $\tU$ by a central reduction:
	\[
	\U = \tU / \langle K_i K_i' -1 \mid i\in \I \rangle.
	\]
	
	Let $\Br(W)$ is the braid group associated to $W$:
	\[
	\Br(W)=\langle s_0,s_1,s_2\mid s_i s_j s_i=s_j s_i s_j, \forall i\neq j \in \I\rangle.
	\]
	Let $T_i\, (i\in \I)$ be the braid group automorphism on $\tU$ defined by (see \cite{LW22b})
	\begin{align*}
		&T_i(K_i)=K_i^{-1},\qquad T_i(K_j)=K_i K_j \qquad T_i(E_i)=-F_i K_i,\qquad T_i(F_i)=-K'_i E_i,\\
		&T_i(K_i')=(K_i')^{-1},\quad T_i(K_j')=K_i' K'_j,\qquad T_i(E_j)=[E_i, E_j]_{v^{-1}} ,\qquad T_i(F_j)=[F_j, F_i]_{v},
	\end{align*}
	for $0\leq i\neq j\leq 2$. These formulas (by setting $K_i'=K_i^{-1}$) recover Lusztig's braid group action on $\U$ \cite{Lus93}.
	
	We also have
	\begin{equation}\label{braid3}
		T_{\omega_1}=\circledast T_2 T_1,\qquad T_{\omega_2}=\circledast^{-1} T_1 T_2
	\end{equation}
	where $\circledast$ is the diagram automorphism given by $\circledast(0)=1,\circledast(1)=2,\circledast(2)=0$.

	\subsection{The $\imath$quantum group}	
	
	Let $\tau$ be the following diagram automorphism given by swapping vertices $1$ and $2$ while fixing $0$:
	\begin{center}\setlength{\unitlength}{0.7mm}
		\vspace{-.2cm}
		\begin{equation}
			\label{eq:satakerank1}
			\begin{picture}(50,13)(0,-10)
				\put(-0.5,-2){\small $1$}
				\put(20,-2){\small $2$}
				
				\put(12.5,-19){\line(1,2){8}}
				\put(9.5,-19){\line(-1,2){8}}
				\put(9.5,-23){\small $0$}
				\put(3,-.5){\line(1,0){16}}
				\color{purple}
				\qbezier(11,4)(15,3.7)(19.5,1)
				\qbezier(11,4)(7,3.7)(2.5,1)
				\put(19,1.1){\vector(2,-1){0.5}}
				\put(2,1.1){\vector(-2,-1){0.5}}
				\put(10,-.5){\small $^{\tau}$}
			\end{picture}
		\end{equation}
		\vspace{.5cm}
	\end{center}
	
	Recall $\I=\{0,1,2\}$. Let $\tUi:=\tUi(\widehat{\g})$ be the universal quasi-split $\imath$quantum group associated to the Satake diagram \eqref{eq:satakerank1} (with all nodes white); see \cite{LW22, CLW21} (also cf.  \cite{Let99,Ko14}). By definition, $\tUi$ is generated by $B_i,\K_i$ ($i\in \I$), 
 where $\K_i$ are invertible, 
 subject to the following relations:
	\begin{align}
		\K_i \K_j= \K_j \K_i,
		\qquad \K_i B_j &= v^{c_{\tau i,j} -c_{i j}} B_j \K_i\qquad (i,j\in\I),
		\label{kB} \\
		B_{2}B_1^2 -[2]B_1B_{2}B_1 + B_1^2B_{2} &=- [2](v\K_1 B_1 +v B_1\K_{2}),
		\label{B211} \\
		B_{1}B_2^2 -[2]B_2B_{1}B_2 + B_2^2B_{1} &=- [2](v\K_2 B_2 +v B_2\K_{1}),
		\label{B122} \\
		B_{0}B_i^2 -[2]B_iB_{0}B_i + B_i^2B_{0} &=0 \qquad (i=1,2),
		\label{B0ii} \\
		B_{i}B_0^2 -[2]B_0B_{i}B_0 + B_0^2B_{i} &= -v^{-1} \K_0 B_{i} \qquad (i=1,2).
		\label{Bi00}
	\end{align}
	Note that $\K_0$  is central. It helps to make \eqref{kB}, say for $i=1$, explicit (the relations remain valid up to a swap of indices $1 \leftrightarrow 2$):
	\begin{align}
		\label{eq:KB}
		\K_1B_1 =v^{-3} B_1\K_1, \quad
		\K_1B_2 =v^{3} B_2\K_1, \quad
		\K_1B_0 = B_0\K_1.
	\end{align}
	
	There are some flexibility on scaling on the generators $\K_i$; our convention on $\K_i$ and the relations \eqref{kB}--\eqref{Bi00} gives us an embedding $\imath: \tUi \longrightarrow \tU$ (compare \cite[\S6.1]{LW22}) by letting
	\begin{align}
		\label{eq:embed}	
		B_i &\mapsto F_i+E_{\tau i}K'_i, \qquad \K_0\mapsto -v^2 K_0K'_0,\qquad
		\K_j \mapsto K_j  K'_{\tau j},\quad \forall i\in\I, j\neq0.
	\end{align}
	We often identify $\tUi$ with a subalgebra of $\tU$ through the embedding $\imath$, and then identify $B_i =F_i+E_{\tau i}K'_i$, and so on.
	
	For $\mu = \sum_{i\in \I} a_i \alpha_i  \in \Z \I$,  define $\K_\mu =\prod_{i\in \I} \K_i^{a_i}$ and
	\begin{align*}
		\K_\delta =\K_0 \K_1 \K_2.
	\end{align*}
	The algebra $\tUi$ is endowed with a filtered algebra structure
	\begin{align}  \label{eq:filt1R1}
		\widetilde{\U}^{\imath,0} \subset \widetilde{\U}^{\imath,1} \subset \cdots \subset \widetilde{\U}^{\imath,m} \subset \cdots
	\end{align}
	by setting
	\begin{align}  \label{eq:filtR1}
		\widetilde{\U}^{\imath,m} =\Q(v)\text{-span} \{ B_{i_1} B_{i_2} \ldots B_{i_s} \K_\mu \mid \mu \in \Z\I, i_1, \ldots, i_s \in \I, s \le m \}.
	\end{align}
	Note that
	\begin{align}  \label{eq:UiCartanR1}
		\widetilde{\U}^{\imath,0} =\bigoplus_{\mu \in \N\I} \Q(v) \K_\mu,
	\end{align}
	is the $\Q(v)$-subalgebra generated by $\K_i$ for $i\in \I$.
	The following is a $\tUi$-variant of a basic result of Letzter and Kolb on quantum symmetric pairs $(\U, \Ui)$.
	
	\begin{proposition}
		(cf. \cite{Let02, Ko14})
		\label{prop:graded}
		The associated graded algebra $\gr \tUi$ with respect to \eqref{eq:filt1R1}--\eqref{eq:filtR1} admits the following identification:
		\begin{align}   \label{eq:filterR1}
			\begin{split}
				\gr \tUi & \cong \U^- \otimes \Q(v)[\K_i^{\pm 1} \mid i\in \I],
				\\
				\overline{B_i} & \mapsto F_i,  \qquad
				\overline{\K}_i \mapsto \K_i \quad (i\in \I).
			\end{split}
		\end{align}
	\end{proposition}
	
	\begin{lemma}
		The $\Q(v)$-algebra $\tUi$ is $\Z \I$-graded by 
		\begin{align}
			\label{eq:deg}
			\deg (B_i) =\alpha_i, \quad \deg (\K_i) =\alpha_i+\alpha_{\tau i}, \quad \text{ for } i \in \I.
		\end{align}
	\end{lemma}
	
	For any $\gamma\in\Z\I$, we denote by $\tUi_\gamma$ the homogeneous subspace of degree $\gamma$, and then
	\begin{align}
		\label{eq:grading}
		\tUi=\bigoplus_{\gamma\in\Z\I} \tUi_\gamma.
	\end{align}
	
	Recall the bilinear form $(\cdot, \cdot)$ on $\Z\I$.
	
	\begin{lemma}
		\label{lem:KiX}
		We have
		\[
		\K_i X = v^{- (\alpha_i -\alpha_{\tau i}, \gamma)} X \K_i,
		\]
		for  $i\in \I$ and $X\in\tUi_{\gamma}$ $(\gamma \in \N\I)$. In particular, we have $\K_i X =X \K_i$, for $X\in\tUi_{k\delta}$ $(k\in\Z)$.
	\end{lemma}
	
	\begin{proof}
		We first observe from \eqref{kB} that
		$\K_i X = v^{- (\alpha_i -\alpha_{\tau i}, \deg(X))} X \K_i$, for all generators $X =\K_j, B_j$. The formula for general $X$ follows. The special case follows by $(\cdot, \delta)=0$.
	\end{proof}
	
	The diagram involution $\tau$ gives rise to an involution $\btau$ on the algebra $\tUi$:
	\begin{align} \label{Phi}
		\btau(B_i)=B_{\tau i}, \qquad \btau(\K_i)=\K_{\tau i},\qquad \forall i=0,1,2.
	\end{align}
	
	
	The following lemma follows by inspection of the defining relations of $\tUi$.
	
	\begin{lemma} 
		\label{lem:sigma}
		There exists a $\Q(v)$-algebra anti-involution $\sigma_{\tau}:\tUi\rightarrow \tUi$ such that
		\[
		\sigma_{\tau}(B_i)=B_i, \quad \sigma_{\tau}(\K_i)= \K_{\tau i},
		\quad \forall i\in \I.
		\]
	\end{lemma}

	\subsection{Relative braid group operators $\TT_0^{\pm 1}$}
	
	Set
	\[
	{\boldsymbol \alpha}_0:=\alpha_0, \qquad {\boldsymbol \alpha}_1:=(\alpha_1 +\alpha_2)/2.
	\]
	The relative root system with simple roots $\{{\boldsymbol \alpha}_0, {\boldsymbol \alpha}_1\}$ is of twisted affine type $A_2^{(2)}$:
	\begin{center}
		\begin{tikzpicture}[baseline=0, scale=5]
			\node at (-0.5,-0.02) {\large$\circ$};
			\node at (0,-0.02) {\large$\circ$};

			\draw[-](-0.45, 0.025) to (-0.07, 0.025);
			\draw[-](-0.45, 0) to (-0.06, 0);
			\draw[-](-0.45, -0.025) to (-0.06, -0.025);
			\draw[-](-0.45, -0.05) to (-0.07, -0.05);
			\node at (0, -.1) { $1$};
			\draw[-](-0.08,-0.07) to (-0.05,-0.0125) to (-0.08,0.045);
			\node at (-0.5,-.1) {$0$};
		\end{tikzpicture}
	\end{center}
	
	Let $\Br(W,\tau)=\langle \bs_0,\bs_1 \rangle$ be the free group of two generators, which can be regarded as the braid group associated to the Weyl group of twisted affine type $A_2^{(2)}$. We have a  group embedding
	\begin{align*}
		\Br(W,\tau) &\to \Br(W)\\
		\bs_0 &\mapsto s_0, \qquad
		\bs_1 \mapsto s_1 s_2 s_1.
	\end{align*}
	The diagram involution $\tau$ induces an involution on $\Br(W)$, $s_i\mapsto s_{\tau i}$ for all $i$.
	Then we can identify $\Br(W,\tau)$ with the $\tau$-fixed subgroup $\Br(W)^\tau$ via the above embedding.
	
	Following the proposal \cite{KP11} as established in \cite{WZ22} for $\imath$quantum groups of finite types, one expects a relative braid group action of $\Br(W,\tau)$ on $\tUi$. Our first main goal is to construct the braid group operators $\TT_0, \TT_1$ (and respectively, $\TT_0^{-1}, \TT_1^{-1}$) explicitly corresponding to $\bs_0, \bs_1$. We start with the easier one $\TT_0^{\pm 1}$. The formulas for the braid operator $\TT_0^{\pm 1}$ essentially coincide with the ones appearing in finite type; cf. \cite{KP11, LW21a}.

	\begin{proposition} [\cite{WZ22}]
		\label{prop:T0}
		There exists a $\Q(v)$-algebra automorphism $\TT_0:\tUi\rightarrow\tUi$ such that
		\begin{align}
			&\TT_0(\K_{\gamma})=\K_{s_0 \gamma},\quad
			\TT_0(B_0)=B_0 \K_0^{-1},\quad
			\TT_0(B_i) \mapsto [B_i,B_0]_v,
		\end{align}
		for $\gamma \in \N\I$, and $i=1,2.$
		Its inverse is given by
		\begin{align}
			\TT_0^{-1}(\K_{\gamma})= \K_{s_0 \gamma},\quad \TT_0^{-1}(B_0)= B_0 \K_0^{-1},\quad \TT^{-1}_0(B_i) = [B_0,B_i]_v.
		\end{align}
		Moreover,
		\begin{align}
			\label{eq:T1sigma0}
			\TT_0^{-1}= \sigma_{\tau} \TT_0 \sigma_{\tau}.
		\end{align}
	\end{proposition}
	
	\begin{proof}
		The formulation of $\TT_0$ and its proof are covered by the constructions in \cite{WZ22} (see Remark~5.8 therein), even though the formulation {\em loc. cit.} is focused on finite types.
	\end{proof}
	\subsection{Relative braid group operators $\TT_1^{\pm 1}$}
	
	Formulating and establishing the braid group operator $\TT_1$ on $\tUi$ turns out to require a significant amount of new work. We caution that the process of formulating $\TT_1$ in this subsection involves some extension of the field $\Q(v)$, but the final formulas for the action of $\TT_1$ on generators of $\tUi$ are valid over $\Q(v)$.
	
	Recall the Satake diagram of quasi-split affine $A_2$ type \eqref{eq:satakerank1}. Setting $\vs_{\diamond} =-v^{-1/2}$, we denote by $\Psi_{\diamond}$ the scaling automorphism on $\tU$ defined by
	\begin{align*}
		\Psi_{\diamond}: & \tU \longrightarrow \tU,\qquad
		K_i\mapsto \vs_{\diamond}^{1/2} K_i,\quad
		K_i'\mapsto \vs_{\diamond}^{1/2} K_i',\quad
		E_i\mapsto \vs_{\diamond}^{1/2} E_i,\quad F_i\mapsto F_i.
	\end{align*}
	
	Following \cite[\S 2.3]{WZ22}, we define the rescaled braid group operators on $\tU$ as follows:
	\begin{align}
		\label{eq:TTT}
		\begin{split}
			\tT_j &:= \Psi_{\diamond}^{-1} \circ T_j \circ \Psi_{\diamond} \qquad (j \in \{1,2\}),
			\\
			\tT_{\bs_1}^{-1} &:= \tT_{1}^{-1} \tT_{2}^{-1} \tT_{1}^{-1} =\tT_{2}^{-1} \tT_{1}^{-1} \tT_{2}^{-1}.
		\end{split}
	\end{align}
	(The scaling looks a little different from {\em loc. cit.}, but leads to the same $\tT_j$.)
	
	Let $\tfX_1$ be the quasi $K$-matrix for the universal quantum symmetric pair associated to the rank one Satake subdiagram $(\{1,2\},\tau)$; cf. \cite{BW18}. The following formula is due to \cite{DK19} (who works in the setting of $\imath$quantum groups with parameters).
	
	\begin{lemma} [\cite{DK19}]
		We have
		\begin{align} \label{eq:r1}
			\tfX_1&=\bigg(\sum_{m\geq 0} \frac{ v^{ -m(m-1)/2} (-1)^m}{[m]!} [E_1,E_{2}]_{v^{-1}}^m \bigg)\bigg(\sum_{m\geq 0} \frac{ v^{-m(m-1)/2}(-1)^m}{[m]!} [E_{2},E_1]_{v^{-1}}^m \bigg).
		\end{align}
	\end{lemma}

	The following theorem is an affine analogue of \cite[Theorem~B]{WZ22},	where we treat $\TT_1^{-1}$ as a symbol for now.
	\begin{theorem}\label{thm:T-1}
		There exists an automorphism $\TT_1^{-1}$ of $\tUi$ such that
		\begin{equation}
			\label{eq:conj}
			\TT_1^{-1}(x) \cdot \tfX_{1}=\tfX_{1} \cdot \tT_{\bs_1}^{-1}(x),
			\quad \text{ for all } x\in \tUi.
		\end{equation}
		More explicitly, the action of $\TT_1^{-1}$ is given by
		\begin{align}
			\label{eq:T-1K}
			\TT_1^{-1}(\K_1)&= v^{-1} \K_2^{-1},\quad \TT_1^{-1}(\K_2)= v^{-1} \K_1^{-1},\quad \TT_1^{-1}(\K_0)= v^{2} \K_0 \K_1^2\K_2^2, \\
			\label{eq:T-1B1B2}
			\TT_1^{-1}(B_1)&=-vB_1 \K_1^{-1},\qquad \TT_1^{-1}(B_2)=-vB_2 \K_2^{-1},
			\\
			\label{eq:T-1B0}
			\TT_1^{-1}(B_0)&=
			v\Big[[B_1,B_2]_v,\big[ B_2,[B_1,B_0]_v\big] \Big]- \big[[B_1,B_2]_{v^3},B_0 \big]\K_2+ v B_0 \K_1\K_2.
		\end{align}
	\end{theorem}
	The proof of Theorem~\ref{thm:T-1} is highly nontrivial and lengthy; it will occupy Section~\ref{sec:proofT1} below.

    \begin{theorem} \label{thm:T1}
		There exists a $\Q(v)$-algebra automorphism $\TT_1:\tUi\rightarrow\tUi$ such that
		\begin{align}
			\label{eq:T1K}
			\TT_1(\K_1)&= v^{-1} \K_2^{-1},\quad \TT_1(\K_2)= v^{-1} \K_1^{-1},\quad \TT_1(\K_0)= v^{2} \K_0 \K_1^2\K_2^2, \\
			\label{eq:T1B1B2}
			\TT_1(B_1)&=-v^{-2}B_1\K_2^{-1} ,\qquad \TT_1(B_2)=-v^{-2}B_2\K_1^{-1} ,\\
			\label{eq:T1B0}
			\TT_1(B_0)&=
			v\Big[\big[ [B_0,B_1]_v,B_2\big] ,[B_2,B_1]_v\Big]- \big[B_0, [B_2,B_1]_{v^3} \big]\K_1+ v B_0 \K_1\K_2.
		\end{align}
		Moreover, $\TT_1^{-1}$ and $\TT_1$ are mutual inverses, and
		\begin{align}
			\label{eq:T1sigma}
			\TT_1^{-1}= \sigma_{\tau} \TT_1 \sigma_{\tau}.
		\end{align}
	\end{theorem}
	
	\begin{proof}
		Just as the automorphism $\TT_1^{-1}$ is defined by the intertwining property \eqref{eq:conj}, the automorphism $\TT_1$ can be characterized by the following intertwining property (cf. \cite[Theorem~6.1]{WZ22}):
		\begin{equation}
			\label{eq:conjb}
			\TT_1 (x) \cdot \tT_{\bs_1}(\tfX_{1}^{-1}) =\tT_{\bs_1}(\tfX_{1}^{-1}) \cdot \tT_{\bs_1} (x),
			\quad \text{ for all } x\in \tUi.
		\end{equation}
		The formulas \eqref{eq:T1K}--\eqref{eq:T1B1B2} were already established in \cite[Proposition~6.2-6.3]{WZ22}, which are valid over Kac-Moody setting. The proof of the formula \eqref{eq:T1B0} is a version of the proof for $\TT_1^{-1}$ given in Section~\ref{sec:proofT1}, and will be omitted.
		
		The proof that $\TT_1^{-1}$ and $\TT_1$ are mutual inverses is the same as for \cite[Theorem~6.7]{WZ22}; it formally boils down to the uniqueness of the intertwining properties \eqref{eq:conj} and \eqref{eq:conjb}.
		
		The identity \eqref{eq:T1sigma} follows by comparing the formulas \eqref{eq:T-1K}--\eqref{eq:T-1B0} and  \eqref{eq:T1K}--\eqref{eq:T1B0} and using Lemma~\ref{lem:sigma}.
	\end{proof}
	
	\begin{lemma}
		We have
		\begin{align}
			\TT_1^{-1}(B_0)&=\Big[B_2,\big[ [B_1,B_2]_v,[B_1,B_0]_v\big]_v\Big]_v -v \big[ B_2, [B_1,B_0]_{v^{3}}\big]_{v^{-2}} \K_2
			\notag \\
			&- \big[ B_1,[ B_2 ,B_0 ]_v\big]_{v^{2}} \K_2-v^{-1}[2] \big[B_2,[B_1,B_0]_v\big]_{v^4}\K_1+vB_0\K_1\K_2.
			\label{eq:T1B0a}
			\\
			\TT_1(B_0)&=  \Big[\big[ [B_0,B_1]_v,[B_2,B_1]_v\big]_v,  B_2\Big]_v  -v\big[ [B_0,B_1]_{v^{3}}, B_2\big]_{v^{-2}} \K_1
			\notag \\
			&- \big[ [ B_0 ,B_2 ]_v,B_1\big]_{v^{2}} \K_1-v^{-1}[2] \big[[B_0,B_1]_v,B_2\big]_{v^4}\K_2+vB_0\K_1\K_2.
			\label{eq:T1B0b}
		\end{align}
	\end{lemma}
	
	\begin{proof}
		We shall verify the formula \eqref{eq:T1B0a}. Indeed, we have
		\begin{align*}
			&\Big[B_2,\big[ [B_1,B_2]_v,[B_1,B_0]_v\big]_v\Big]_v
			\\
			&=v \Big[ [B_1,B_2]_v,\big[ B_2,[B_1,B_0]_v\big] \Big]-v \Big[ [B_1,B_0]_v,\big[ B_2,[B_1,B_2]_v\big]_v \Big]_{v^{-1}}
			\\
			&=v \Big[ [B_1,B_2]_v,\big[ B_2,[B_1,B_0]_v\big] \Big]-v \Big[ [B_1,B_0]_v,[2] (v^{-1}B_2\K_2+v^2B_2\K_1) \Big]_{v^{-1}}
			\\
			&=v \Big[ [B_1,B_2]_v,\big[ B_2,[B_1,B_0]_v\big] \Big]+v^2[2] \big[ B_2,[B_1,B_0]_v \big]_{v^{-2}} \K_2 +v^{-1}[2] \big[B_2, [B_1,B_0]_v\big]_{v^4} \K_1.
		\end{align*}
		The formula \eqref{eq:T1B0b} follows by \eqref{eq:T1B0a} by applying the anti-involution $\sigma_{\tau}$ and  \eqref{eq:T1sigma}.
	\end{proof}
	
	\begin{remark}
		If one is willing to work over an extension field $\Q(v^{\frac12})$, it is possible to use variants of $\K_1, \K_2$ rescaled by $v^{\frac12}$ to make the formulas for the action of $\TT_1^{\pm 1}$ on $\K_j$, for $j\in \I$, look simpler (namely, removing the $v$-powers in the formulas \eqref{eq:T-1K} and \eqref{eq:T1K}). To achieve the same effect, another option is to rescale only $\K_i$ by a factor $v$, for a fixed $i\in \{1,2\}$ (however, there is a broken symmetry between recaled $\K_1, \K_2$).
	\end{remark}

	\section{Proof of Theorem~\ref{thm:T-1}}
	\label{sec:proofT1}
	
	In this section, we shall prove Theorem~\ref{thm:T-1} on the relative braid group operator $\TT_1^{-1}$. The proof follows the basic strategy developed in \cite{WZ22}, but the execution requires additional technical long computations.
	\subsection{Steps for proof of Theorem~\ref{thm:T-1}}
	\label{subsec:steps}
	
	Let us rephrase Theorem~\ref{thm:T-1} to facilitate the discussion of the strategy of its proof.
	\begin{enumerate}
		\item[(S1)]
		For any $x \in \tUi$, there exists a unique element $x'\in \tUi$ such that $x' \cdot \tfX_{1}=\tfX_{1} \cdot \tT_{\bs_1}^{-1}(x)$.
		\item[(S2)]
		Sending $x\mapsto x'$ defines an injective homomorphism of $\tUi$, denoted by $\TT_1^{-1}$.
		\item[(S3)]
		The formulas for $\TT_1^{-1}(x)$, with $x=\K_i\, (i\in \I)$, $x=B_i\, (i=1,2)$, and $x=B_0$ are given by \eqref{eq:T-1K}, \eqref{eq:T-1B1B2}, and \eqref{eq:T-1B0}, respectively. That is, the identity $\TT_1^{-1}(x) \cdot \tfX_{1}=\tfX_{1} \cdot \tT_{\bs_1}^{-1}(x)$ holds, for each generator $x \in \{\K_i, B_i|i\in \I\}$ of $\tUi$.
		\item[(S4)]
		$\TT_1^{-1}$ is surjective (and hence an automorphism of $\tUi$).
	\end{enumerate}
	
	Statement (S2) follows clearly from (S1) and the invertibility of $\tfX_1$.
	Since $\tT_{\bs_1}^{-1}$ is an automorphism, if $\tfX_{1} \cdot \tT_{\bs_1}^{-1}(x_1) \cdot \tfX_{1}^{-1} \in \tUi$ and $\tfX_{1} \cdot \tT_{\bs_1}^{-1}(x_2) \cdot \tfX_{1}^{-1} \in \tUi$, for $x_1, x_2 \in \tUi$, then $\tfX_{1} \cdot \tT_{\bs_1}^{-1}(x_1x_2) \cdot \tfX_{1}^{-1} \in \tUi$. This reduces the verification of (S1) for arbitrary $x$ to (S1) for $x$ being generators of $\tUi$, i.e., (S3). Statement (S4) actually follows when the counterparts to (S1)--(S3) above for $\TT_{1}$ in Theorem~\ref{thm:T1} are proved (cf. proof of \cite[Theorem~6.7]{WZ22}).
	
	Therefore, to complete the proof of Theorem~\ref{thm:T-1}, it remains to prove (S3) above.

	\subsection{The identity~\eqref{eq:conj} for $x=\K_i\, (i\in \I)$ or $x=B_i\, (i=1,2)$}
	\label{subsec:S3a}
	
	The statement that the ``rank one" formulas \eqref{eq:T-1K}--\eqref{eq:T-1B1B2} satisfy the intertwining property \eqref{eq:conj} were already proved in \cite[Proposition~4.11, Theorem~4.15]{WZ22}; as remarked {\em loc. cit.}, these statements are actually valid for $\tUi$ of arbitrary Kac-Moody type. This settles (S3) above, i.e., the intertwining property~\eqref{eq:conj}, for all generators of $\tUi$ except for $x=B_0$.
	
	We shall deal with this exceptional case on $\TT_1^{-1}(B_0)$ in the remainder of this section.

	\subsection{Reduction of Identity~\eqref{eq:interB0} }
	\label{subsec:S3b}
	
	The remainder of this section is devoted to the challenging proof that the intertwining relation~\eqref{eq:conj} is satisfied by $\TT_1^{-1}(B_0)$ in \eqref{eq:T-1B0}, i.e.,
	\begin{align}
		\label{eq:interB0}
		\TT_1^{-1}(B_0) \cdot \tfX_{1}=\tfX_{1} \cdot \tT_{\bs_1}^{-1}(B_0).
	\end{align}
	
	Write $B_0 = F_0 + E_0 K_0'$, and hence $\tT_{\bs_1}^{-1}(B_0) =\tT_{\bs_1}^{-1}(F_0 )+\tT_{\bs_1}^{-1}(E_0 K_0').$
	Set
	\begin{align*}
		\mathfrak{F}&=v\Big[[B_1, B_2 ]_v,\big[ B_2,[B_1, F_0]_v\big] \Big]- \big[ [B_1,B_2]_{v^3},F_0 \big]\K_2+ v F_0 \K_1\K_2,\\
		\mathfrak{E}&=v\Big[[B_1,B_2]_v,\big[ B_2,[B_1,E_0 K_0']_v\big] \Big]- \big[ [B_1,B_2]_{v^3},E_0 K_0' \big]\K_2+ v  E_0 K_0' \K_1\K_2.
	\end{align*}
	The desired intertwining relation \eqref{eq:interB0} can be reformulated as
	\[
	\TT_i^{-1}(B_0) =\mathfrak{F}  + \mathfrak{E},
	\]
	and it will follow once we verify the following 2 identities:
	\begin{align}
		\mathfrak{F}  \tfX_1 &=\tfX_1   \tT_{\bs_1}^{-1}(F_0 ),
		\label{eq:tqcI2}\\
		\mathfrak{E}  \tfX_1 &=\tfX_1   \tT_{\bs_1}^{-1}(E_0 K_0').
		\label{eq:tqcI3}
	\end{align}
	%
	%
	\subsection{Reformulation of Identity \eqref{eq:tqcI2}}
	
	The following lemma is an affine counterpart of \cite[Lemma 5.1]{WZ22} for $\tUi$ of finite type.
	
	\begin{lemma}
		\label{lem:tqcI2b}
		$\tT_{\bs_1}^{-1}(E_0 K_0')$ commutes with $\tfX_1$.
	\end{lemma}
	
	\begin{proof}
		We first observe that the following identity holds, for $i=1, 2$:
		\begin{equation}\label{eq:tqcI1}
			\tT_{\tau i}^{-1}\tT_{i}^{-1}(E_0) \cdot E_{\tau i}=v^2 E_{\tau i} \cdot \tT_{\tau i}^{-1}\tT_i^{-1}(E_0).
		\end{equation}
		Indeed, it follows by definition of the rescaled braid operators \eqref{eq:TTT}  that
		\begin{equation}
			\label{eq:TTE0}
			\tT_{\tau i}^{-1}\tT_{i}^{-1}(E_0)=\vs_{\diamond}^{-3/2}\big[[E_0,E_{\tau i}]_{v^{-1}},[E_i, E_{\tau i}]_{v^{-1}}\big]_{v^{-1}}.
		\end{equation}
		Hence, the identity \eqref{eq:tqcI1} follows from \eqref{eq:TTE0} and the following two $v$-commuting relations:
		\begin{align*}
			[E_0,E_{\tau i}]_{v^{-1}} E_{\tau i} =v E_{\tau i} [E_0,E_{\tau i}]_{v^{-1}},\qquad
			[E_i,E_{\tau i}]_{v^{-1}} E_{\tau i} =v E_{\tau i} [E_i,E_{\tau i}]_{v^{-1}}.
		\end{align*}
		These two $v$-commuting relations are simple reformulations of the Serre relations in $\tU$.
		
		The identity \eqref{eq:tqcI1} admits the following two equivalent reformulations:
		\begin{align}
			\label{eq:TE0TE}
			\tT_{\bs_1}^{-1}(E_0)\tT_{i}^{-1}(E_{\tau i})
			&= v^2 \tT_{i}^{-1}(E_{\tau i})\tT_{\bs_1}^{-1}(E_0);
			\\
			\label{eq:TTE0b}
			\big[\tT_{\bs_1}^{-1}(E_0 K_0'),\tT_{i}^{-1}(E_{\tau i})\big] &=0.
		\end{align}
		
		Note that $\tT_i^{-1}(E_{\tau i})=\vs_{\diamond}^{-1/2} [E_{\tau i},E_i]_{v^{-1}}$ and $\tT_{\tau i}^{-1}(E_{i})=\vs_{\diamond}^{-1/2} [E_i,E_{\tau i}]_{v^{-1}}$ are terms appearing in the formula \eqref{eq:r1} for $\tfX_1$. Hence, the lemma follows by \eqref{eq:TTE0b}.
	\end{proof}
	
	By \cite[Proposition 2.15, Lemma~ 5.1]{WZ22}, we have
	\begin{align*}
		B_1  \tfX_1 =\tfX_1 (F_1 + v^{-1} E_{2} K_1), \quad
		B_{2}  \tfX_1 =\tfX_1 (F_{2} + v^{-1} E_{1} K_{2}), \quad
		F_0 \tfX_1 =\tfX_1 F_0.
	\end{align*}
	Hence, we obtain
	\begin{align*}
		&\mathfrak{F} \cdot \tfX_1 =\tfX_1\cdot\Ddot{\mathfrak{F}},
	\end{align*}
	where
	\begin{align*}
		\Ddot{\mathfrak{F}} & := v \Big[ [F_1+v^{-1}E_2 K_1, F_2+v^{-1} E_1 K_2]_v, \big[ F_2+v^{-1} E_1 K_2,[F_1+v^{-1}E_2 K_1,F_0]_v  \big]  \Big]
		\\
		&\qquad - \big[ [F_1+v^{-1}E_2 K_1, F_2+v^{-1} E_1 K_2]_{v^3},F_0 \big]K_2 K_1'+ v F_0 K_1 K_2 K_1' K_2' .
	\end{align*}
	Thus, the desired identity \eqref{eq:tqcI2} is equivalent to the identity
	\begin{equation}\label{eq:tqcI2'}
		\Ddot{\mathfrak{F}} =\tT_{\bs_1}^{-1}(F_0 ).
	\end{equation}
	\subsection{Deduction of \eqref{eq:tqcI3} from \eqref{eq:tqcI2'}}
	\label{subsec:oneway}
	
	\begin{lemma}
		\label{lem:varphi}
		There exists an automorphism $\varphi$ of $\tU$ which sends, for $i=1,2$,
		\begin{align*}
			&F_i \mapsto E_{\tau i} K_i', \qquad E_i \mapsto v F_{\tau i} {K_i'}^{-1}, \qquad\; K_i \mapsto K_{\tau i}', \qquad K_i' \mapsto K_{\tau i},\\
			&F_0 \mapsto E_0 K_0',\qquad E_0 \mapsto v^{-2} F_0 {K'_0}^{-1}, \qquad K_0 \mapsto K_0',\qquad K_0'\mapsto K_0.
		\end{align*}
	\end{lemma}
	
	\begin{proof}
		By a direct computation, the images of $\varphi$ satisfies the defining relations of $\tU$ and hence $\varphi$ is an endomorphism of $\tU$. 
		One writes down an obvious candidate for the inverse homomorphism $\varphi^{-1}$ acting on generators, which is indeed a homomorphism by another direct computation. Hence $\varphi$ is an automorphism of $\tU$.
	\end{proof}
	
	\begin{lemma}\label{lem:tqcI0}
		We have
		\begin{itemize}
			\item[(1)]
			$\tT_{\bs_1}^{-1}(F_0 )=v\Big[ \big[F_1,F_2\big]_v, \big[F_2,[F_1,F_0]_v\big]\Big]$;
			\item[(2)]
			$\tT_{\bs_1}^{-1}(E_0 K_0')= v\Big[ \big[E_1,E_2\big]_{v^{-1}}, \big[E_1,[E_0,E_2]_{v^{-1}}\big]\Big]K_0' (K_1'K_2')^2$.
		\end{itemize}
	\end{lemma}
	
	\begin{proof}
		We prove (1) only, as (2) can be obtained in a similar way. We have
		\begin{align}\label{eq:tqcI0}
			&\tT_{1}^{-1}\tT_{2}^{-1}\tT_{1}^{-1}(F_0 ) = \Big[\big[[F_1,F_2]_v,-E_1 (K_1')^{-1}\big]_v, \big[[F_1,F_2]_v,[F_1,F_0]_v\big]_v\Big]_v.
		\end{align}
		Note that
		\begin{align}\notag
			\big[[F_1,F_2]_v,-E_1 (K_1')^{-1}\big]_v &=\big[[F_1,F_2]_v,-E_1\big] (K_1')^{-1} \\
			&= \frac{1}{v-v^{-1}} [{K_1-K_1'},F_2]_v (K_1')^{-1} = F_2.\label{eq:tqcI0'}
		\end{align}
		Hence, rewriting \eqref{eq:tqcI0} with the help of \eqref{eq:tqcI0'}, we have
		\begin{align*}
			\tT_{1}^{-1}\tT_{2}^{-1}\tT_{1}^{-1}(F_0 )
			&= \Big[ \big[F_2, \big[[F_1,F_2]_v,[F_1,F_0]_v\big]_v \Big]_v\\
			&=\Big[\big[\big[F_2,[F_1,F_2]_v\big]_v,[F_1,F_0]_v \Big]_v
			+v\Big[ \big[F_1,F_2\big]_v, \big[F_2,[F_1,F_0]_v\big]\Big]\\
			&=v\Big[ \big[F_1,F_2\big]_v, \big[F_2,[F_1,F_0]_v\big]\Big],
		\end{align*}
		where the last equality follows from the Serre relations.
	\end{proof}
	
	In terms of the automorphism $\varphi$ in Lemma~\ref{lem:varphi}, we can rephrase Lemma~ \ref{lem:tqcI0} as
	\begin{align}
		\label{eq:phiFE}
		\varphi\big(\tT_{\bs_1}^{-1}(F_0 )\big) =\tT_{\bs_1}^{-1}(E_0 K_0' ).
	\end{align}
	
	By Lemma \ref{lem:tqcI2b}, the identity \eqref{eq:tqcI3} is equivalent to the identity
	\begin{equation}\label{eq:tqcI3'}
		\mathfrak{E}=\tT_{\bs_1}^{-1}(E_0 K_0').
	\end{equation}
	The identity \eqref{eq:tqcI3'} follows from \eqref{eq:tqcI2'} by applying $\varphi$,  thanks to the identities \eqref{eq:phiFE} and
	\begin{align}
		\label{eq:FE}
		\varphi(\Ddot{\mathfrak{F}})=\mathfrak{E}.
	\end{align}
	The identity \eqref{eq:FE} holds since by Lemma~\ref{lem:varphi}, $\varphi(F_i+v^{-1} E_{\tau i} K_i)=B_i$ and $\varphi(K_i K'_{\tau i})=K_i K_{\tau i}'$, for $i=1,2$. Therefore, we have concluded that the identity \eqref{eq:tqcI3} follows from \eqref{eq:tqcI2'}.
	
	Let us summarize what we have achieved in this section. We reduced in \S\ref{subsec:steps}--\ref{subsec:S3a} the proof of Theorem~\ref{thm:T-1} to the identity \eqref{eq:interB0}. Furthermore, in  \S\ref{subsec:S3b}--\ref{subsec:oneway}, we reduced the proof of the identity \eqref{eq:interB0}, which is equivalent to the identities \eqref{eq:tqcI2}--\eqref{eq:tqcI3}, to the identity \eqref{eq:tqcI2'}.
	
	The proof of the identity \eqref{eq:tqcI2'} shall be given in Appendix~\ref{App:A}.

	\section{Constructions of root vectors}
	\label{sec:rtvector}
	
	In this section, we shall introduce the root vectors for the quasi-split $\imath$quantum groups $\tUi$.

	\subsection{Real and imaginary root vectors}
	
	Fix signs $o(1), o(2)\in \{\pm1\}$ such that $o(1)=- o(2)$.
	
	Denote, for $i \in \I =\{0,1,2\},$
	\begin{align}
		\label{eq:C}
		B_{\alpha_i}:=B_i,\qquad \bK_{\alpha_i}:=\bK_i ,\qquad \bK_\de:=\bK_1\bK_2\bK_0,\qquad C:=-v\bK_\de.
	\end{align}
	
	Introduce the following ``translation" symmetry on $\tUi$:
	\begin{equation} \label{eq:Tomeg}
		\TT_{\bome}:=\TT_0 \TT_1.
	\end{equation}
	Similar to \cite{Da93, Be94, BK20, LW21b}, we define the {\em real $v$-root vectors}
	\begin{align}
		\label{eq:Bik}
		B_{i,k} =B_{k\delta+\alpha_i} :=\big(o(i)\TT_{\bome}\big)^{-k}(B_i), \quad
		\text{ for } k \in \Z, i \in \{1,2\}.
	\end{align}
	Note that
	\begin{align}
		\label{B1-1}
		\begin{split}
			B_{i,-1} =o(i)\TT_{\bome} (B_i)
			&= o(i)v^{-1}[B_i,B_0]_v \bK_iC^{-1}
			\\
			&=-o(i)v^{-2}[B_{i},B_0]_v \bK_i\K_\de^{-1}.
		\end{split}
	\end{align}
	
	Denote, for $i=1,2$ and $k\in \Z$,
	\begin{equation}\label{Dn}
		D_{i,k} :=-[B_{\tau i},B_{i,k }]_{v^{-1}}-[B_{i,k+1},B_{\tau i,-1}]_{v^{-1}}.
	\end{equation}
	Set $\TH_{i,0} =\frac{1}{v-v^{-1}}$.
	Define the {\em imaginary $v$-root vectors} $\TH_{i,m}$, for $m\ge 1$, inductively:
	\begin{align}
		\label{TH}
		\TH_{i,1} & = -o(i)\big(\big[ B_i,[ B_{\tau i},B_0 ]_v\big]_{v^{2}}- v B_0 \bK_i\big),
		\\
		\label{eq:TH12}
		\TH_{i,2} &= -vD_{i,0} C \bK_{\tau i}^{-1}+v\TH_{i,0}C -\TH_{\tau i,0}C\bK_{\tau i}^{-1}\bK_i,
		\\
		\label{THn}
		\TH_{i,m} &=v\TH_{i,m-2} C -vD_{i,m-2}C \bK_{\tau i}^{-1},
		\quad \text{ for } m\ge 3.
	\end{align}
	For convenience, we set  $\TH_{i,m}=0$ for $m<0$.
	
	\begin{remark}
		\label{rem:imagine}
		The recursions \eqref{TH}--\eqref{THn} for $\Theta_{i,m}$ become more transparent when they are viewed as special cases of the Drinfeld type relation \eqref{qsiA1DR5}. A crucial property satisfied by $\Theta_{i,m}$ is the commutativity relation \eqref{qsiA1DR2reform}.
	\end{remark}

 \begin{remark}
 \label{rmk:aTH}
 For the split type, there are two versions of imaginary root vectors, following \cite{BK20} and \cite{LW21b}; these 2 versions differ by a simple central factor on the generating function level (see \cite[(2.31)]{LW21b}).  Accordingly a variant of the imaginary root vectors $\acute{\TH}_{i,m}$ for $m\geq 1$ can be defined as follows:
 \begin{align}\label{def:aTH}
 \begin{split}
		\acute{\TH}_{i,1} & = \TH_{i,1}, \qquad \acute{\TH}_{i,0} = \TH_{i,0}, 
  \\
		\acute{\TH}_{i,2} &= -vD_{i,0} C \bK_{\tau i}^{-1}+ \acute{\TH}_{i,0}C -\acute{\TH}_{\tau i,0}C\bK_{\tau i}^{-1}\bK_i,
		\\
		\acute{\TH}_{i,m} &=\acute{\TH}_{i,m-2} C -vD_{i,m-2}C \bK_{\tau i}^{-1},
		\quad \text{ for } m\ge 3.
  \end{split}
	\end{align}
The recursion \eqref{def:aTH} can be viewed as special cases of the Drinfeld type relation \eqref{qsiA1DR5BK}.
 \end{remark}
	
	By definition, we have, for $k\in \Z, m\ge 0$ and $i\in \{1,2\}$,
	\begin{equation}
		\label{eq:switchroot}
		B_{i,k}=(-1)^k\btau(B_{\tau i,k}),\qquad
		\TH_{i,m}=(-1)^m\btau(\TH_{\tau i,m}),
	\end{equation}
	where $\btau $ is the involution in \eqref{Phi}.
	
	Recall from \eqref{eq:grading} that the algebra $\tUi=\bigoplus_{\alpha\in\Z\I} \tUi_\alpha$ is $\Z \I$-graded.
	
	\begin{lemma}
		\label{lem:Tigrade}
		For $i=0,1$, we have
		\begin{align*}
			\TT_0(\tUi_\gamma)\subseteq \tUi_{s_0(\gamma)},\qquad \TT_1(\tUi_\gamma)\subseteq \tUi_{s_1s_2s_1(\gamma)},\qquad\forall \gamma\in\Z\I.
		\end{align*}
	\end{lemma}
	
	\begin{proof}
		This follows from the formulas for $\TT_i$ in Proposition~\ref{prop:T0} and Theorem~\ref{thm:T1} since $\TT_i$ are algebra automorphisms.
	\end{proof}

	\begin{lemma}
		\label{lem:degree root}
		We have $B_{i,k}\in\tUi_{k\delta+\alpha_i}$, and $\TH_{i,m}\in\tUi_{m\delta}$, for any $l\in\Z$, $m>0$. 
	\end{lemma}
	
	\begin{proof}
		The proof follows from Lemma \ref{lem:Tigrade}, the definition of $B_{i,k}$ in \eqref{eq:Bik} and the recursive definition of $\Theta_{i,m}$ in \eqref{TH}--\eqref{THn}.
	\end{proof}
	
	
	%
	%
	\subsection{The classical limit}
	\label{subsec:limit}
 
	Recall $\U$ is the quantum group of affine type $A_2^{(1)}$. The $\imath$quantum group $\Ui_{ \bvs }$ with parameter $\bvs =(\s_i)_{i=0,1,2}\in (\Q(v)^{\times})^3$ (cf. \cite{Ko14}) is a subalgebra of $\U$ generated by $B_i$, for $i\in\I$, and $k_1$, via the embedding
	\begin{align*}
		\iota: \Ui_{\bvs} &\longrightarrow \U \\
		k_1 \mapsto K_1  K^{-1}_{2}, &
		\qquad
		B_i \mapsto F_i+ \s_i E_{\tau i}K^{-1}_i \; \; (i \in \I).
	\end{align*}
	Alternatively, the $\imath$quantum group $\Ui_{ \bvs }$ can be obtained from $\tUi$ by the central reduction (cf. \cite{LW22, WZ22}): \begin{align*}
		\tUi/(\K_0+v^2\s_0,\K_1\K_2-\s_1\s_2) & \stackrel{\simeq}{\longrightarrow} \Ui_{ \bvs},
		\\
		B_i \mapsto B_i, \quad
		\K_0 \mapsto -v^2\s_0, \quad
		& \K_1 \mapsto \s_1 k_1, \quad
		\K_2 \mapsto \s_2 k_1^{-1}.
	\end{align*}
	
	The enveloping algebra $U(\hg)$ is recovered from $\U$ at the $v \rightsquigarrow 1$ limit by letting
	\begin{align*}
		E_i \rightsquigarrow e_i, \qquad F_i \rightsquigarrow f_i,\qquad K_i \rightsquigarrow 1, \qquad
		(K_i -K_i^{-1})/(v-v^{-1}) \rightsquigarrow  h_i.
	\end{align*}
	
	Let $\dot\s, \dot\s_0\in \Q^\times$.
	Assume that the parameters $(\s_i)_{i=0,1,2}\in \Q(v)^{\times,3}$ have limits  as $v\rightsquigarrow 1$ as follows: $\lim\limits_{v\rightsquigarrow 1} \s_0 =\dot\s_0$, and $\lim\limits_{v\rightsquigarrow 1} \s_1 =\lim\limits_{v\rightsquigarrow 1} \s_2 =\dot\s.$
	
	For $\g=\mathfrak{sl}_3$, we denote by   $\widehat{\g}=\g[t,t^{-1}]\oplus \Q c$ the corresponding affine Lie algebra of type $\widehat A_2$. Let $\{e_i,h_i,f_i\mid i=1,2\}$ (and respectively, $\{e_i,h_i,f_i\mid i\in \I\}$) be the standard Chevalley generators of $\g$ (and respectively, $\hg$). Set
	\[
	b_0 :=f_0 + \dot\s_0 e_0,
	\qquad b_i :=f_i + \dot\s e_{\tau i}, \quad \text{ for }i\in \{1,2\}.
	\]
	Then we obtain the limits of generators of $\Ui_\bvs$ as $v \rightsquigarrow 1$:
	\begin{align}
		\label{eq:v=1B}
		B_i \rightsquigarrow b_i \; (i\in \I),\qquad
		k_1 \rightsquigarrow 1,\qquad
		(k_1 -k_1^{-1})/(v-v^{-1}) \rightsquigarrow  (h_1-h_2).
	\end{align}
	Introduce a shorthand notation
	\begin{align}
		\label{eq:dc}
		\dc := \dot\s_0 \dot\s^2,
	\end{align}
	which will be often used in this section. Note that
	\[
	\K_\delta \rightsquigarrow -\dc,
	\qquad
	C \rightsquigarrow \dc.
	\]
	
	Denote $e_\theta =[e_1,e_2], f_\theta =[f_2, f_1]$. Write $x_k:=x \otimes t^k \in \hg$, for $x\in \g$ (and so $e_{i,k} =e_i \otimes t^k, e_{\theta,k} =e_{\theta}\otimes t^k$, and so on). We identify $e_0=f_\theta \otimes t,f_0= e_\theta \otimes t^{-1}$, and $h_0=[e_0,f_0]$. Denote by
	\[
	\omega_\tau: \hg \longrightarrow \hg
	\]
	the involution such that
	\[
	f_0 \mapsto \dot\s_0 e_0, \qquad
	h_0 \mapsto -h_0, \qquad
	f_i \mapsto \dot\s e_{\tau i}, \qquad
	h_i \mapsto -h_{\tau i},
	\]
	for $i\in \{1,2\}$. It follows that $\omega_\tau (c) =-c$.
	For $i=1,2$, we denote
	\begin{align}
		b_{i,r} &= f_{i,-r}+ \dot\s \dc^r e_{\tau i,r},
		\label{eq:bir} \\
		t_{i,r} &= -h_{i,-r}+\dc^r h_{\tau i,r},
		\label{eq:tir} \\
		b_{\theta,r} &=f_{\theta,-r}+ \dot\s^2 \dc^r e_{\theta,r}.
		\notag
	\end{align}
	Note that $t_{i,r} = - \dc^r t_{\tau i,-r}$. One checks that the Lie subalgebra of $\hg$ of ${\omega_\tau}$-fixed points, $\hg^{\omega_\tau}$,  has a basis
	\[
	\{b_{\theta,r}, b_{i,r}, t_{i,m}, h_1 -h_2 \mid r\in \Z, m\in \Z_{\ge 1}, i=1,2 \}.
	\]
	The enveloping algebra $U(\hg^{\omega_\tau})$ is recovered from $\Ui_\bvs$ at the $v \rightsquigarrow 1$ limit.
	
	\begin{proposition}
		The Lie algebra $\hg^{\omega_\tau}$ is generated by $\{b_{i,r}, t_{i,r} \mid r\in \Z,  i=1,2 \}$. Moreover, the following relations in  $\hg^{\omega_\tau}$ are satisfied: for $r, s, l, k_1, k_2 \in \Z$,
		\begin{align}
			[b_{i, s},b_{i,r}] &= 0,
			\label{eq:bb0}
			\\
			[b_{i,r}, b_{\tau i,s+1}] -[b_{i,r+1},b_{\tau i,s}] &= -\dot\s \dc^r t_{\tau i, s-r+1} + \dot\s \dc^{r+1} t_{\tau i, s-r-1},
			\label{eq:bbbb}	\\
			t_{i,r} &= - \dc^r t_{\tau i,-r},
			\label{eq:tt} \\
			[t_{i,r},b_{i,s}] &= 2b_{i,s+r}+\dc^r b_{i,s-r},
			\label{eq:kap-b1} \\
			[t_{\tau i,r},b_{i,s}]&=-b_{i,s+r}-2\dc^r b_{i,s-r},
			\label{eq:kap-b2}  \\
			\big[ b_{i,k_1}, [b_{i,k_2}, b_{\tau i,l}] \big]
			&= -2 \dot\s\dc^l b_{i,k_1+k_2-l}
			- \dot\s\dc^{k_1} b_{i,l+k_2-k_1}
			- \dot\s\dc^{k_2} b_{i,l+k_1-k_2}.
			\label{eq:kap-b3}
		\end{align}
	\end{proposition}
	
	\begin{proof}
		Let us write down some details on proving \eqref{eq:bbbb} and \eqref{eq:kap-b3} while skipping the details on the remaining easier identities.
		Using \eqref{eq:bir}--\eqref{eq:tir}, we compute
		\[
		[b_{\tau i, s},b_{i,r}] = 	\dot\s \dc^r t_{\tau i, s-r}  + [f_{\tau i}, f_{i}]_{-s-r} + \dot\s^2 \dc^{s+r} [e_i, e_{\tau i}]_{-s-r}.
		\]
		Then the identity \eqref{eq:bbbb} follows by using the above formula twice.
		
		We also compute again using \eqref{eq:bir}--\eqref{eq:tir}:
		\begin{align}
			[b_{i,k_2}, b_{\tau i,l}]
			&= -b_{\theta, k_2+l} +\dot\s\dc^l t_{i, k_2-l},
			\label{eq:btheta}\\
			[b_{i,k_1}, -b_{\theta, k_2+l}]
			&= -\dot\s\dc^{k_1} b_{i, l+k_2-k_1},
			\notag \\
			[b_{i,k_1}, \dot\s\dc^l t_{i, k_2-l}]
			&= -2\dot\s\dc^{l} b_{i, k_1+k_2-l}
			-\dot\s\dc^{k_2} b_{i, l+k_1-k_2}.
			\notag
		\end{align}
		Then the identity \eqref{eq:kap-b3} follows from  these formulas.
		
		Finally, the generating set statement follows by \eqref{eq:btheta}.
	\end{proof}
	
	\begin{remark}
		One can use a smaller generating set $\{b_{i,r}, t_{i,m}, t_{1,0} \mid r\in \Z, m\in \Z_{\ge 1}, i=1,2 \}$  for Lie algebra $\hg^{\omega_\tau}$, thanks to \eqref{eq:tt}. By arguments similar to the proof of Theorem~\ref{thm:Dr1}, one can show the above generators and relations provide a presentation for the Lie algebra $\hg^{\omega_\tau}$. This is the classical limit of the Drinfeld type presentation of $\tUi$ established in this paper.
	\end{remark}
	
	\begin{proposition}
		\label{prop:climit}
		Set $o(1)=-o(2)=1$. Then the classical limits of the Drinfeld generators $B_{i,r}, \Theta_{i,n}$ are $b_{i,r}, t_{i,n}$, respectively, for $i=1,2$, $r\in \Z$, and $n\ge 1$.
	\end{proposition}
	
	\begin{proof}
		Using the formula \eqref{B1-1} for $B_{i,-1}$,  \eqref{eq:B11} for $B_{i,1}$ and \eqref{TH} of $\TH_{i,1}$, we see that the classical limits of these root vectors are  $b_{i,-1}, b_{i,1}, t_{i,1}$, respectively:
		\begin{align*}
			B_{1,-1} &\rightsquigarrow [b_1,b_0](\dot\s_0\dot\s)^{-1} =f_{1,1}+(\dot\s\dot\s_0)^{-1}e_{2,-1}=b_{1,-1},\\
			B_{2,-1} &\rightsquigarrow -[b_2,b_0](\dot\s_0\dot\s)^{-1} =f_{2,1}+(\dot\s\dot\s_0)^{-1}e_{1,-1}=b_{2,-1},\\
			\TH_{1,1}& \rightsquigarrow -[b_1,[b_2,b_0]]+ \dot\s b_0=-h_{1,-1} +\dot\s_0\dot\s^2 h_{2,1}=t_{1,1},\\
			\TH_{2,1}& \rightsquigarrow [b_2,[b_1,b_0]]- \dot\s b_0=-h_{2,-1} +\dot\s_0\dot\s^2 h_{1,1}=t_{2,1},\\
			B_{1,1} &\rightsquigarrow \big[[b_1,b_2],[b_1,b_0]\big]-3\dot\s[b_1,b_0]=f_{1,-1}+\dot\s_0\dot\s^3 e_{2,1}=b_{1,1}.
		\end{align*}
		
		
		More generally, using $[\TH_{i,1},B_{j,k}]=[c_{ij}]B_{j,k+1}-[c_{\tau i,j}]B_{j,k-1}C$ from \eqref{qsiA1DR3} and \eqref{eq:kap-b1}--\eqref{eq:kap-b2}, one shows inductively that the classical limit of $B_{i,r}$ is $b_{i,r}$, for  $r\in \Z$. Similarly, using the recursive definition \eqref{TH}--\eqref{THn} of $\Theta_{i,n}$ and \eqref{eq:bbbb}, one shows inductively that the classical limit of $\Theta_{i,n}$ is $t_{i,n}$, for  $n\ge 1$.
	\end{proof}
	
	\begin{remark}\label{rmk:climit}
		Define a new Lie algebra $\widetilde \g^{\omega_\tau}$, a variant of $\hg^{\omega_\tau}$, in which $\dc$ is viewed as a  central element (instead of being a scalar), with formally the same relations \eqref{eq:bb0}--\eqref{eq:kap-b3}. Then $\widetilde \g^{\omega_\tau}$ is
		a $\Z$-graded Lie algebra, by assigning $\deg b_{i,r} =r, \deg t_{i,r} =r$, and $\deg \dc =2$, and noting that the relations \eqref{eq:bb0}--\eqref{eq:kap-b3} have become homogeneous. Compare Remark~\ref{rmk:qlimit}.
	\end{remark}

	%
	
	%
	

	%
	%
	\subsection{A vanishing criterion}
	
	We continue the notations from \S\ref{subsec:limit}.
	Recall $\g =\mathfrak{sl}_3$, which is also identified with $\g \otimes t^0$ in $\hg$. Denote by $\mathfrak k =\g^{\omega_\tau} \subset \hg^{\omega_\tau}$, as we note ${\omega_\tau}$ preserves the subalgebra $\g$ of $\hg$. Denote by $\mf l$ the (proper) Lie subalgebra of $\hg^{\omega_\tau}$  generated by $t_{i,0}$, $b_{i,n}$, for all $i\in \{1,2\}$ and $n\in \Z$. Denote by $\mf l_{\rm{ev}}$ the Lie subalgebra of $\hg^{\omega_\tau}$  spanned by $x\otimes t^{2n} +\omega_\tau(x)\otimes t^{-2n}$, for all $x\in \g$ and $n\in \Z$.
	
	\begin{lemma}
		\label{lem:L}
		We have $\mf l_{\rm{ev}} \subset \mf l$. 
	\end{lemma}
	
	\begin{proof}
		We observe by a direct computation on $[b_{i,m}, b_{j,k}]$, with $m +k \in 2\Z$ and $i\in \{1,2\}$, that $b_{\theta,n}, t_{i,n}\in \mf l$, for $n$ even.
	\end{proof}
	
	\begin{lemma}
		\label{lem:U}
		Let $u \in U(\hg^{\omega_\tau})$. If $\ad (\mf l_{\rm{ev}}) u =0$,  then $u=0$.
	\end{lemma}
	
	\begin{proof}
		To prove the lemma, it is equivalent to prove its counterpart in the symmetric algebra (in place of the enveloping algebra):
		\begin{equation}
			\label{eq:UtoS}
			\text{ If $\ad (\mf l_{\rm{ev}}) u =0$, for $u \in S(\hg^{\omega_\tau})$ then $u=0$. }
		\end{equation}
		
		The Lie algebra $\hg^{\omega_\tau}$ (and hence its subalgebra $\mf l_{\rm{ev}}$) admits an $\N$-filtered algebra structure, by letting $\deg (x\otimes t^{n} + {\omega_\tau}(x)\otimes t^{-n}) =n$, for all nonzero $x\in \g$ and $n\in \N$. Then the associated graded algebras can be identified as follows: $\gr \hg^{\omega_\tau} \cong \mf k \oplus t \g [t]$ and $\gr \mf l_{\rm{ev}} \cong \mf k \oplus t^2 \g [t^2]$.
		
		The proof of the claim \eqref{eq:UtoS} is reduced to the proof of the following counterpart in the associated graded algebra:
		\begin{equation}
			\label{eq:S}
			\text{ If $\ad (\gr \mf l_{\rm{ev}}) u =0$, for $u \in S(\gr \hg^{\omega_\tau})$ then $u=0$. }
		\end{equation}
		The proof of \eqref{eq:S} is entirely analogous to the ones of \cite[Lemmas 2.8.1, 1.7.4]{M07}, by analyzing the action of $\gr \mf l_{\rm{ev},2} \cong \g\otimes t^2$ on $u$, where $\mf l_{\rm{ev},2}$ is the subspace of $\mf l_{\rm{ev}}$ spanned by $x\otimes t^{2} +{\omega_\tau}(x)\otimes t^{-2}$, for all $x\in \g$. We omit the detail.
		
		(The proof of the lemma could have proceeded directly as for \cite[Lemmas 2.8.1, 1.7.4]{M07} without referring to the filtration and associated graded algebra explicitly, though it would involve somewhat messy notations if one insists on writing out details.)
	\end{proof}


	\begin{lemma}
		\label{lem:u=0}
		If $u \in U(\hg^{\omega_\tau})$ commutes with $b_{i,l}$ and $t_{i,0}$, for all $i\in \{1,2\}$ and $l\in \Z$, then $u=0$.
	\end{lemma}
	
	\begin{proof}
		By assumption of the lemma and the definition of $\mf l$, we have $\ad(\mf l) u=0$ in $S(\hg^{\omega_\tau})$, and hence $\ad(\mf l_{\rm{ev}}) u=0$ by Lemma~\ref{lem:L}. Now it follows by Lemma~\ref{lem:U} that $u=0$.
	\end{proof}
	
	The following lemma is the main point of this subsection.
	\begin{lemma}
		\label{lem:X0}
		Let $X\in \tUi$ be a noncommutative polynomial of $B_1,B_2,B_0$ with coefficients in $\Q(v) [\K_i^{\pm 1}, i \in \I]$ without constant term. If $\TT_{\bome}(X)=X$ and $[X,B_i]=[X,\K_i]=0$ for $i\in \{1,2\}$, then $X=0$.
	\end{lemma}
	
	\begin{proof}
		By assumption, $X$ commutes with $\K_i$ and $B_{i,l} =\TT_{\bome}^{-l}(B_i)$, for all $i\in \{1,2\}, l\in \Z$; in particular, $X$ commutes with the finite type part of $\tUi$ (generated by $B_i, \K_i$ for $i=1,2$).
		We prove by contradiction, by assuming that $X$ is nonzero. Write $X$ as a linear combination of a monomial basis (due to  Letzter and Kolb for $\Ui$). Then $X$ descends to a nonzero element in a central reduction $\Ui$ and then to a nonzero element $x\in U(\hg^{\omega_\tau})$, which commutes with $t_{i,0}$ and $b_{i,l}$, for all $i\in \{1,2\}, l\in \Z$. This contradicts with Lemma~\ref{lem:u=0}.
	\end{proof}

	\section{A Drinfeld type presentation}
	\label{sec:Dr1}

	In this section, we shall give a Drinfeld type presentation for the quasi-split affine $\imath$quantum group $\tUi$ of rank $1$.

	\subsection{The definition}
	\label{subsec:sym}
 
	Recall $\I=\{0,1,2\}$ and $(c_{ij})_{i,j \in \I}$ denotes the Cartan matrix of affine type $A_2^{(1)}$. We shall denote by $\Sym_{k_1,k_2}$ the symmetrization with respect to (current) indices $k_1, k_2 \in \Z$ in the sense $\Sym_{k_1,k_2} f(k_1, k_2) =f(k_1, k_2) +f(k_2, k_1)$; we will never apply the symmetrization to indices in $\I$. Introduce the shorthand notation
		\begin{align}
			\label{eq:SS}
			\bS(k_1,k_2|l;i) : = \Sym_{k_1,k_2}\Big(B_{i,k_1}B_{i,k_2}B_{\tau i,l}-[2]B_{i,k_1} B_{\tau i,l} B_{i,k_2} + B_{\tau i,l} B_{i,k_1}B_{i,k_2}\Big).
		\end{align}

	\begin{definition}
		\label{def:iDR}
		Let $\tUiD$ be the $\Q(v)$-algebra generated by the elements $B_{i,l}$, $H_{i,m}$, $\bK_i^{\pm1}$, $C^{\pm1}$, where $i\in \{1,2\}$, $l\in\Z$ and $m \in \Z_{\ge 1}$, subject to the following relations: for $m, n \ge 1, l, k, k_1, k_2 \in \Z$, and $i, j \in \{1,2\}$,
		\begin{align}
			C \text{ is central,} \quad &
			\K_i\K_j=\K_j\K_i, \quad
			\K_i H_{j,m}=H_{j,m}\K_i,\quad
			\bK_iB_{j,l}=v^{c_{\tau i,j}-c_{ij}} B_{j,l} \bK_i,
			\label{qsiA1DR1} \\
			[H_{i,m},H_{j,n}] &=0,\label{qsiA1DR2}
			\\
			[H_{i,m},B_{j,l}] &=\frac{[mc_{ij}]}{m} B_{j,l+m}-\frac{[mc_{\tau i,j}]}{m} B_{j,l-m}C^m,\label{qsiA1DR3}
			\\
			\label{qsiA1DR4}
			[B_{i,k},B_{i,l+1}]_{v^{-2}} & -v^{-2}[B_{i,k+1},B_{i,l}]_{v^{2} }=0,
			\\
			[B_{i,k},B_{\tau i,l+1}]_v & -v[B_{i,k+1},B_{\tau i,l}]_{v^{-1}} = -\Theta_{{\tau i},l-k+1}C^k \bK_{i} +v \Theta_{ {\tau i},l-k-1}C^{k+1}\bK_{i}
			\notag \\
			& \qquad\qquad\qquad\qquad\quad \;\, -\Theta_{i,k-l+1}C^l\bK_{{\tau i}} +v \Theta_{i,k-l-1}C^{l+1}\bK_{\tau i},
			\label{qsiA1DR5}  \\
			\mathbb{S}(k_1,k_2\mid l;i )
			=  [2]&\Sym_{k_1,k_2}\sum_{p\geq 0}v^{2p}
            \big[\TH_{\tau i,l-k_2-p}\K_i-v\TH_{\tau i,l-k_2-p-2}C\K_i, B_{i,k_1-p} \big]_{v^{-4p-1}}C^{k_2+p}
            \notag 
            \\
             +v[2]&\Sym_{k_1,k_2}\sum_{p\geq 0}v^{2p} \big[ B_{i,k_1+p+1},\TH_{i,k_2-l-p+1}\K_{\tau i}-v\TH_{i,k_2-l-p-1}C \K_{\tau i}\big]_{v^{-4p-3}} C^{l-1}.
			\label{qsiA1DR6}
		\end{align}
		Here $H_{i,m}$ are related to $\Theta_{i,m}$ by the following equation:
		\begin{align}
			\label{exp h}
			1+ \sum_{m\geq 1} (v-v^{-1})\Theta_{i,m} u^m  = \exp\Big( (v-v^{-1}) \sum_{m\geq 1} H_{i,m} u^m \Big).
		\end{align}
	\end{definition}
	 
	\begin{remark}
		Using Proposition~\ref{prop:climit}, one can show that the classical limits of relations \eqref{qsiA1DR1}-\eqref{qsiA1DR6} are given by \eqref{eq:bb0}-\eqref{eq:kap-b3}; see also \em{Remark}~\ref{rmk:climit}.
	\end{remark}
	
	The $\Q(v)$-algebra $\tUiD$ admits the following translation symmetry.
	
	\begin{lemma}  \label{lem:aut}
		There exists an automorphism $\t$ of the algebra $\tUiD$ given by
		\[
		\t (B_{j,k}) =B_{j,k-1}, \quad
		\t (H_{j,m}) =H_{j,m}, \quad
		\t (\K_j) = \K_j C^{-1}, \quad
		\t(C) =C,
		\]
		(and hence $\t (\Theta_{j,m}) = \Theta_{j,m}$), for all $k\in \Z, m\ge 1$, and $j \in \{1,2\}$.
	\end{lemma}
	
	\begin{proof}
		The proof follows by inspection of the defining relations for $\tUiD$ in Definition~\ref{def:iDR}.
	\end{proof}
	
	We introduce the following generating functions in a variable $z$:
	\begin{align}
		\label{eq:Genfun}
		\begin{split}
			\bB_{i}(z) & =\sum_{k\in\Z} B_{i,k}z^{k},
			\qquad\quad
			\bTH_{i}(z) =1+ \sum_{m \ge 1} (v-v^{-1})\Theta_{i,m}z^{m},
			\\
			\bH_i(z) &=\sum_{m\ge 1} H_{i,m} z^{m},
			\qquad\quad
			\bDel(z)  =\sum_{k\in\Z} C^k z^k,
		\end{split}
	\end{align}
	where $\bTH_i(z)$ and $\bH_i(z)$ are related by
	\begin{align}
		\label{eq:ThH}
		\bTH_i(z)=\exp \big( (v-v^{-1})\bH_i(z) \big).
	\end{align}
	
	The following lemma will be used later.
	
	\begin{lemma}
		\label{lem:relationsreform}
		The following equivalences hold:
		
		(1) The  identity \eqref{qsiA1DR2} is equivalent to
		\begin{align}
			\label{qsiA1DR2reform}
			[\TH_{i,m},\TH_{j,n}]=0,\quad \forall i,j=1,2, \text{ and }m,n\geq1.
		\end{align}
		(2) The identity \eqref{qsiA1DR3} is equivalent to
		\begin{align}
			\label{qsiA1DR3reform}
			&[\TH_{i,m},B_{j,k}]+v^{c_{i,j}-c_{\tau i,j}}[\TH_{i,m-2},B_{j,k}]_{v^{2(c_{\tau i,j}-c_{i,j})}}C
			\\\notag
			&-v^{c_{i,j}}[\TH_{i,m-1},B_{j,k+1}]_{v^{-2c_{ i,j}}}- v^{-c_{\tau i, j}}[\TH_{i,m-1},B_{j,k-1}]_{v^{2c_{\tau i,j}}}C
			=0,
		\end{align}
		for any $m>0$ and $k\in\Z$.
	\end{lemma}
	
	\begin{proof}
		The equivalence in (1) follows directly from \eqref{exp h}.
		
		The proof of the equivalence in (2) is very similar to \cite[Proposition 2.8]{LW21b}, via a generating function formalism \eqref{eq:Genfun}. We outline the main steps below.
		
		The identity \eqref{qsiA1DR3} can be equivalently reformulated as a generating function identity
		\begin{align*}
			(v-v^{-1}) [\bH_i(z), \bB_j(w)]
			&=  \ln
			\frac{(1 -v^{-c_{ij}}zw^{-1})(1 -v^{c_{\tau i,j}} zw C)}{(1 -v^{c_{ij}}zw^{-1}) (1 -v^{-c_{\tau i,j}}zw C)}
			\cdot \bB_j(w).
		\end{align*}
		Via integration and \eqref{eq:ThH}, this identity is equivalent to the identity
		\begin{align*}
			\bTH_i(z) \bB_j(w) \bTH_i(z)^{-1}
			& = \frac{(1 -v^{-c_{ij}}zw^{-1})(1 -v^{c_{\tau i,j}} zw C)}{(1 -v^{c_{ij}}zw^{-1}) (1 -v^{-c_{\tau i,j}}zw C)}
			\bB_j(w),
		\end{align*}
		or equivalently, $$
		\bTH_i(z) \bB_j(w)
		= \frac{(1 -v^{-c_{ij}}zw^{-1})(1 -v^{c_{\tau i,j}} zw C)}{(1 -v^{c_{ij}}zw^{-1}) (1 -v^{-c_{\tau i,j}}zw C)}
		\bB_j(w) \bTH_i(z).$$
		Comparing the coefficients of $z^m w^k$ of both sides of the last identity, for $m\ge 1, k\in \Z$, we obtain the equivalent identity \eqref{qsiA1DR3reform}.
	\end{proof}	
	\subsection{The isomorphism}

	\begin{theorem}
		\label{thm:Dr1}
		There is a $\Q(v)$-algebra isomorphism ${\Phi}: \tUiD \longrightarrow\tUi$, which sends
		\begin{align}
			\label{eq:isom}
			B_{i,l}\mapsto B_{i,l}, \quad \Theta_{i,m} \mapsto \Theta_{i,m},
			\quad
			\K_i\mapsto \K_i, \quad C\mapsto C,
			\quad \text{ for } m\ge 1,\; l\in \Z,\; i \in \{1,2\}.
		\end{align}
		The inverse ${\Phi}^{-1} : \tUi \longrightarrow \tUiD$ sends
		\begin{align*}
			\K_0\mapsto & -v^{-1} C \K_1^{-1}\K_2^{-1},
			\quad
			\K_i\mapsto  \K_i,
			\quad
			B_i\mapsto   B_{i,0},
			\quad \text{ for }i \in\{1,2\},
			\\
			B_0\mapsto
			& o(1) v^{-1} \big(\TH_{1,1}-v [B_1,B_{2,-1}]_{v^{-1}}  C\K_2^{-1}\big)\K_1^{-1}.
		\end{align*}
	\end{theorem}
	
	(We shall refer to $\tUiD$ the {\em Drinfeld type presentation} of $\tUi$. The proof that $\Phi$ is a homomorphism requires long computations, and will be carried out in Section~\ref{sec:verify1}.)

	\begin{proof}
		To show $\Phi$ is a homomorphism, we shall verify that all the defining relations in $\tUiD$ (see Definition~\ref{def:iDR}) are preserved by $\Phi$ in the next Section~\ref{sec:verify1}. More precisely, the relations \eqref{qsiA1DR1}--\eqref{qsiA1DR6}  in $\tUiD$ hold for the images  of the generators of $\tUiD$ under $\Phi: \tUiD \rightarrow\tUi$, thanks to Lemma~\ref{lem:4.1},  Propositions~\ref{prop:verfDR2}, \ref{prop:qsiA1DR3}, \ref{prop:BBcom}, \ref{prop:verfDR5} and \ref{prop:verfqsiA1DR6}, respectively.
		
		Next we show that $\Phi$ is surjective. To that end, it suffices to show that the generator $B_0$ of $\tUi$ lies in the image of $\Phi$ (as other generators clearly do). By definition of $\TH_{1,1}$ in \eqref{TH}, $\Phi$ maps $v^{-1}o(1)\big(\TH_{1,1}-v [B_1,B_{2,-1}]_{v^{-1}}  C\K_2^{-1}\big)\K_1^{-1}$ to $B_0$, and hence the surjectivity of $\Phi$ follows.
		
		The injectivity of $\Phi$ follows by an analogous argument as for the injectivity in  \cite[Theorem~3.13]{LW21b}. For the sake of completeness, we sketch below.
		
		We set $\II =\{1,2\}$ in this proof. Denote by $\tUi_>$ (respectively, $\tUiD_>$) the subalgebra of $\tUi$ (respectively, $\tUiD$) generated by $B_{i,m}$, $H_{i,m}, \K_i$, for $m\ge 1$, and $i\in \II$. Then ${\Phi}: \tUiD \longrightarrow\tUi$ restricts to a surjective homomorphism ${\Phi}: \tUiD_> \longrightarrow\tUi_>$.
		
		The translation symmetries $\t$ on $\tUiD$ (see Lemma~\ref{lem:aut}) and $\TT_{{\bome}}$ in \eqref{eq:Tomeg} are compatible under $\Phi$, i.e.,
		\[
		\Phi \circ \t =\TT_{{\bome}} \circ \Phi.
		\]
		The injectivity of ${\Phi}: \tUiD \rightarrow\tUi$ is then reduced to the injectivity of ${\Phi}: \tUiD_> \rightarrow\tUi_>$, since any element in the kernel of ${\Phi}: \tUiD \rightarrow\tUi$ gives rise to (via a translation automorphism $\t^{-N}$, for $N\gg 0$) to an element in the kernel of ${\Phi}: \tUiD_> \rightarrow\tUi_>$.
		
		It remains to prove the injectivity of ${\Phi}: \tUiD_> \longrightarrow\tUi_>$. We shall accomplish this by examining a certain filtration and its associated graded algebra.
		
		Define a filtration on $\tUiD_>$ by
		\begin{align}  \label{eq:filt1D}
			(\tUiD_>)^0 \subset (\tUiD_>)^1 \subset \cdots \subset (\tUiD_>)^m \subset \cdots
		\end{align}
		by setting
		\begin{align}  \label{eq:filtD}
			(\tUiD_>)^m &=\Q(v)\text{-span} \big\{x=B_{i_1,m_1} B_{i_2,m_2} \ldots B_{i_r,m_r} \Theta_{j_1,n_1} \Theta_{j_2,n_2} \ldots \Theta_{j_s,n_s} \K_\mu
			\\
			&\quad
			\mid \mu \in \N\I, i_1, \ldots, i_r, j_1, \ldots j_s, \in \II, m_1,\ldots, m_r, n_1, \ldots, n_s \ge 1, \text{ht}^+(x)\leq m \big\}.
			\notag
		\end{align}
		Here we have denoted
		\begin{align}  \label{eq:ht}
			\text{ht}^+(x) :=\sum_{a=1}^r \hgt(m_a\de +\alpha_{i_a}) +\sum_{b=1}^s n_b \hgt(\de),
		\end{align}
		where $\hgt(\beta)$ denotes the height of a positive root $\beta$. Recalling $\widetilde{\U}^{\imath,0}$ from \eqref{eq:UiCartanR1}, we have $(\tUiD_>)^0 = \widetilde{\U}^{\imath,0} =\Q(v) [ \K_i^{\pm 1} \mid i\in \I ].$ The filtration \eqref{eq:filt1D}--\eqref{eq:filtD} on $\tUiD_>$ defined via a height function is compatible with the filtration \eqref{eq:filt1R1}--\eqref{eq:filtR1} on $\tUi$ under $\Phi$, and thus the surjective homomorphism ${\Phi}: \tUiD_> \longrightarrow\tUi_>$ induces a surjective homomorphism
		\begin{align}  \label{eq:gradeP}
			{}^{\text{gr}} {\Phi}_>: \tUiDgr_> \longrightarrow\tUigr_>.
		\end{align}

		The Drinfeld presentation ${}^{\text{Dr}}\U$ of the affine quantum group $\U$ has generators $x_{i,k}^{\pm 1}$, $h_{i,m}$, $K_i^{\pm 1}, \texttt C^{\pm 1/2}$, for $i\in \II, k\in \Z, m\in \Z\backslash \{0\}$, cf. \cite{Dr87, Be94}; moreover, we have
		an isomorphism $\phi:  {}^{\text{Dr}}\U \rightarrow \U$. Denote by ${}^{\text{Dr}}\U^-_<$ the $\Q(v)$-subalgebra of $\U$ generated by $x^-_{i,-k}$, for $i\in \II, k>0$, and denote $\U^-_< =\phi({}^{\text{Dr}}\U^-_<)$. Then $\phi$ restricts to an isomorphism \cite{Be94, Da15}
		\begin{align}  \label{eq:phi}
			\phi:  {}^{\text{Dr}}\U^-_< \stackrel{\cong}{\longrightarrow} \U^-_<.
		\end{align}
		
		Recall the following algebra isomorphism from \eqref{eq:filterR1} with respect to the filtration on $\tUi$ \eqref{eq:filt1R1}--\eqref{eq:filtR1}:
		\begin{align*}
			\mathbb G: \U^- \otimes \widetilde{\U}^{\imath,0} \longrightarrow \tUigr,
			\qquad
			F_i \mapsto \overline{B}_i, \quad
			\K_i \mapsto \overline{\K}_i,
		\end{align*}
		where $\U^- =\langle F_i \mid i\in \I \rangle$. The homomorphism $\mathbb G$ above restricts to an isomorphism
		\begin{align}  \label{eq:grade1}
			\mathbb G: \U^-_< \otimes \widetilde{\U}^{\imath,0} \stackrel{\cong}{\longrightarrow} \tUigr_>.
		\end{align}
		
		Finally, by definition \eqref{eq:filtD} of the filtration on $\tUiD_>$, its associated graded algebra is compatible with setting $C=0$ in the defining relations \eqref{qsiA1DR1}--\eqref{qsiA1DR6} of $\tUiD$, which reproduces the defining relations of half the affine quantum group in its Drinfeld presentation. Thus,
		we have a surjective homomorphism
		\begin{align}  \label{eq:Xi}
			\Xi: {}^{\text{Dr}}\U^-_< \otimes \widetilde{\U}^{\imath,0}
			\longrightarrow
			\tUiDgr_>,
		\end{align}
		which sends $x_{i,-k}^- \mapsto \ov{B}_{i,k}$, for $k>0$ (note the opposite signs in indices).
		
		Combining \eqref{eq:gradeP}--\eqref{eq:Xi}, we have obtained the following commutative diagram
		\begin{align*}  
			\xymatrix{
				{}^{\text{Dr}}\U^-_< \otimes \widetilde{\U}^{\imath,0}
				\ar[rr]^{\Xi}
				\ar[d]^{\phi,\cong}
				&& \tUiDgr_>
				\ar[d]^{{}^{\text{gr}}\Phi_>}
				\\
				\U^-_< \otimes \widetilde{\U}^{\imath,0}
				\ar[rr]^{\mathbb G, \cong}
				&& \tUigr_>   }
		\end{align*}
		Since $\Xi$ and ${}^{\text{gr}}\Phi_>$ are surjective while $\phi$ and $\mathbb G$ are isomorphisms, we conclude that ${}^{\text{gr}}\Phi_>: \tUiDgr_> \longrightarrow\tUigr_>$ is injective (and indeed an isomorphism), and so is $\Xi$.
		
		The proof of Theorem~\ref{thm:Dr1} is completed.
	\end{proof}
	
	\begin{remark}\label{rmk:qlimit}
		By definition, the algebra $\tUiD$ (and hence $\tUi$ by Theorem~\ref{thm:Dr1}) is $\Z$-graded by letting
		\[
		\deg C =2, \qquad
		\deg \K_i =0, \qquad
		\deg  B_{i,k} =k, \qquad
		\deg \Theta_{i,n} =n \quad (i=1,2).
		\]
		Compare Remark~\ref{rmk:climit}.
	\end{remark}

	\subsection{Presentation via generating functions}
	
	Recall from \eqref{eq:Genfun} the generating functions $\bB_{i}(z)$, $\bTH_{i}(z)$, $\bH_i(z)$ and $\bDel(z)$ in a variable $z$.
	For $i=1,2$ and variables $w_1, w_2$, we also denote
	\begin{align}
		&\mathbb{S}(w_1,w_2\mid z;i)\\\notag
		:&=\Sym_{w_1,w_2}\Big(\bB_i(w_1)\bB_i(w_2)\bB_{\tau i}(z)-[2]\bB_i(w_1)\bB_{\tau i}(z)\bB_i(w_2)+\bB_{\tau i}(z)\bB_i(w_1)\bB_i(w_2)\Big).
	\end{align}
	Theorem~\ref{thm:Dr1} admits the following reformulation via generating functions.
	\begin{theorem}
		\label{thm:DrqsA1}
		$\tUi$ is isomorphic to the $\Q(v)$-algebra generated by the elements $B_{il}$, $H_{ik}$, $\bK_i^{\pm1}$, $C^{\pm1}$ where $i=1,2$, $l\in\Z$ and $k>0$, subject to the following relations, for $i,j \in\{1, 2\}$:
		\begin{align}
			\label{qsiA1DRG1}
			\K_i\K_j =\K_j\K_i, \quad
			\K_i\bH_j(z) &=\bH_j(z)\K_i,\quad
			\K_i\bB_{j}(z) =v^{c_{\tau i,j}-c_{ij}} \bB_{j}(z) \bK_i, \quad C \text{ is central,}
			\\
			\label{qsiA1DRG2}
			[\bH_{i}(z), \bH_{j}(w)] &=0,
			\\
			\label{qsiA1DRG3}
			\bTH_i(z) \bB_j(w)
			&= \frac{(1 -v^{-c_{ij}}zw^{-1})(1 -v^{c_{\tau i,j}} zw C)}{(1 -v^{c_{ij}}zw^{-1}) (1 -v^{-c_{\tau i,j}}zw C)}
			\bB_j(w)  \bTH_i(z),
			\\
			\label{qsiA1DRG4}
			(v^2z -w) \bB_i(z) \bB_i(w) & +(v^2w-z) \bB_i(w) \bB_i(z)=0,
			\\
			(v^{-1}z-w) \bB_i(z) \bB_{\tau i} (w) & +(v^{-1}w-z) \bB_{\tau i}(w) \bB_i(z)
			\notag \\
			&  = \frac{\bDel(zw) }{1-v^2}
			\big((z -vw) \K_{{i}} \bTH_{\tau i} (w) + (w -vz) \K_{\tau i} \bTH_i(z) \big),
			\label{qsiA1DRG5}
		\end{align}
		\begin{align}
			\label{qsiA1DRG6}
			\mathbb{S}(w_1,w_2\mid z;i)
			&= -\frac{v^{-1}[2]}{{v-v^{-1}}}\Sym_{w_1,w_2}\bDel(w_2 z) \frac{1-v w_2^{-1} z}{1-v^{-2} w_1 w_2^{-1}}\bB_i(w_1)\bTH_{\tau i}(z)\bK_i
            \notag \\\notag
            &\quad +\frac{[2]}{{v-v^{-1}}}\Sym_{w_1,w_2}\bDel(w_2 z) \frac{1-vw_2^{-1}z}{1-v^2 w_1 w_2^{-1}}\bTH_{\tau i}(z)\bK_i\bB_i(w_1)
            \\\notag
            &+\frac{v[2]}{{v-v^{-1}}}\Sym_{w_1,w_2}\bDel(w_2 z)\frac{w_1^{-1}z-v w_1^{-1} w_2}{1-v^2 w_1^{-1} w_2}\bB_i(w_1)\bTH_{i}(w_2)\bK_{\tau i}
            \\
            &+\frac{v^{-2}[2]}{{v-v^{-1}}}\Sym_{w_1,w_2}\bDel(w_2 z) \frac{v w_1^{-1}w_2 - w_1^{-1}z}{1-v^{-2}w_1^{-1}w_2} \bTH_{i}(w_2)\bK_{\tau i}\bB_i(w_1).
		\end{align}
	\end{theorem}
	 
	\begin{proof}
		We simply rewrite the relations \eqref{qsiA1DR1}--\eqref{qsiA1DR6} in Definition \ref{def:iDR} by using the generating functions \eqref{eq:Genfun}. For example, the relation \eqref{qsiA1DRG6} is obtained by multiplying both sides of the relation \eqref{qsiA1DR6} by $w_1^{k_1}w_2^{k_2}z^{l}$ and summing over $k_1,k_2,l\in\Z$.
	\end{proof}
(The factors $\frac1{v-v^{-1}}$ on the RHS of \eqref{qsiA1DRG6} were missing in the published version. Thanks to Li Luo and Zheming Xu for alerting us about this.)

 \subsection{Drinfeld type presentation via different root vectors}
 The alternative imaginary root vectors $\acute{\TH}_{i,m}$ defined in Remark~\ref{rmk:aTH} lead to the following presentation of $\tUi$, which is a variant of the one given in Definition~\ref{def:iDR}. The (new) $H_{i,m}$ used in the theorem below is defined through the old formula \eqref{exp h} (with $\TH_{im}$ therein replaced by $\acute{\TH}_{i,m}$). Recall the notation $\mathbb{S}(k_1,k_2\mid l;i )$ from \eqref{eq:SS}.
 
 \begin{theorem}
 \label{thm:Drvariant}
 $\tUi$ is isomorphic to the $\Q(v)$-algebra generated by the elements $B_{i,l}$, $H_{i,m}$, $\bK_i^{\pm1}$, $C^{\pm1}$, where $i\in \{1,2\}$, $l\in\Z$ and $m \in \Z_{\ge 1}$, subject to the relations \eqref{qsiA1DR1}-\eqref{qsiA1DR4} and the following two relations \eqref{qsiA1DR5BK}--\eqref{qsiA1DR6BK} (in place of \eqref{qsiA1DR5}-\eqref{qsiA1DR6}):
		\begin{align}
		[B_{i,k},B_{\tau i,l+1}]_v & -v[B_{i,k+1},B_{\tau i,l}]_{v^{-1}} = -\acute{\TH}_{{\tau i},l-k+1}C^k \bK_{i} + \acute{\TH}_{ {\tau i},l-k-1}C^{k+1}\bK_{i}
			\notag \\
			& \qquad\qquad\qquad\qquad\quad \;\, -\acute{\TH}_{i,k-l+1}C^l\bK_{{\tau i}} + \acute{\TH}_{i,k-l-1}C^{l+1}\bK_{\tau i},
			\label{qsiA1DR5BK}  \\
		\mathbb{S}(k_1,k_2\mid l;i )
			=[2]&\Sym_{k_1,k_2}\sum_{p\geq 0}v^{2p}
            \big[\acute{\TH}_{\tau i,l-k_2-p}\K_i-\acute{\TH}_{\tau i,l-k_2-p-2}C\K_i, B_{i,k_1-p} \big]_{v^{-4p-1}}C^{k_2+p}
            \notag 
            \\
             +v[2]&\Sym_{k_1,k_2}\sum_{p\geq 0}v^{2p} \big[ B_{i,k_1+p+1},\acute{\TH}_{i,k_2-l-p+1}\K_{\tau i}-\acute{\TH}_{i,k_2-l-p-1}C \K_{\tau i}\big]_{v^{-4p-3}} C^{l-1},
			\label{qsiA1DR6BK}
		\end{align}
  for $m, n \ge 1, k_1, k_2, k, l \in \Z$, and $i, j \in \{1,2\}$.
 \end{theorem}

	\section{Verification of the current relations}
	\label{sec:verify1}
	
	In this section, we verify that all the defining relations in $\tUiD$ are preserved by the homomorphism $\Phi$ in \eqref{eq:isom}. This completes the proof of Theorem~\ref{thm:Dr1}.

	\subsection{Strategy of proofs}
	\label{subsec:strategy}
	
	Let us explain the general strategy of the proofs before getting to the technical details. The approach of \cite{BK20} for $q$-Onsager algebra (see also \cite{Da93} for quantum affine $\sll_2$) provides a helpful though rough guideline for the overall inductive arguments. However, our quasi-split affine rank one setting behaves with complexity of affine rank two, as there are 2 infinite series of real (and respectively, imaginary) root vectors as Drinfeld type generators; in contrast, there is only one infinite series of real (and respectively, imaginary) root vectors for $q$-Onsager algebra. Accordingly, the approach developed in \cite{LW21b} and especially in \cite{Z21} dealing with Serre relations in higher ranks also helps to inspire new ways to get around various technical difficulties.
	
	A central relation \eqref{qsiA1DR2} concerns about the commutativity among imaginary root vectors $\TH_{i,n}$, for $i=1,2$ and $n\ge 1$; a closely related property is the $\TT_{\bome}$-invariance of $\TH_{i,n}$. Assuming these properties, several additional relations among root vectors for $\tUi$ can be proved.
	However, $\TH_{i,n}$ is defined in terms of (a linear combination of) $v$-commutators between real root vectors. To establish the commutativity among $\TH_{i,n}$, one has to first understand to some extent commutators of $\TH_{i,n}$ with real root vectors, and so we would run in a vicious circle if we were not very careful.
	
	By direct computations, we shall establish $\TT_{\bome} (\TH_{j,1}) =\TH_{j,1}$, and some formula for the commutator $[\TH_{j,1}, B_1]$. This suffices to derive fully the first nontrivial Relation~\eqref{qsiA1DR4}.
	
	The actual inductive proofs of Relations~\eqref{qsiA1DR2}, \eqref{qsiA1DR3}, and \eqref{qsiA1DR5} run like a spiral.
	Under the assumption on the partial commutativity between $\TH_{j,1}$ and $\TH_{i,m}$, for $m<n$ (for some fixed positive integer $n$), we show that $\TT_{\bome} (\TH_{i,n}) =\TH_{i,n}$, and then establish part of (bounded by $n$) the relations for commutators between real root vectors; the Serre relations  \eqref{qsiA1DR6} are also partially proved along the way, which imply partial relations for commutators between imaginary and real root vectors. We then use these to establish fully the commutativity among $\TH_{i,m}$, for $m\le n$, and then all these relations hold unconditionally. The proof of  Relation~\eqref{qsiA1DR3} by induction independent of (the proof of) the Serre relations \eqref{qsiA1DR6} follows closely the approach developed in \cite{Z21} (instead of \cite{LW21b}). Finally, we fully prove the Serre relations \eqref{qsiA1DR6} for $\tUi$, imitating  \cite{Z21} again.

	\subsection{Relation \eqref{qsiA1DR1} }
	
	\begin{lemma}
		\label{lem:4.1}
		Relation \eqref{qsiA1DR1} holds in $\tUi$.
	\end{lemma}
	
	\begin{proof}
		Recall that $C=-vK_\de=-v\K_0\K_1\K_2$. Then $C$ is central and $\K_i\K_j =\K_j\K_i$ by \eqref{kB}.
		
		We have $H_{j,m} \in \tUi_{m\delta}$ by using \eqref{exp h} and noting that $\Theta_{j,m} \in \tUi_{m\delta}$ (see Lemma~\ref{lem:degree root}). Hence, by Lemma~\ref{lem:KiX}, we have $\K_i H_{j,m}=H_{j,m}\K_i$.
		
		The last relation $\bK_iB_{j,l}=v^{c_{\tau i,j}-c_{ij}} B_{j,l} \bK_i$ in \eqref{qsiA1DR1} follows by Lemma~ \ref{lem:KiX} and Lemma~\ref{lem:degree root}.
	\end{proof}

	\subsection{Properties of $B_{i,1}$ and $\TH_{i,1}$}
	
	\begin{lemma}
		\label{lem:Bi1}
		For $i=1,2$, we have
		\begin{align}
			\label{eq:B11}
			B_{i,1}= o(i)\Big( \big[[B_i,B_{\tau i}]_v,[B_{i},B_0]_v\big]_v-v[B_i,B_0]_{v^{3}}\bK_{\tau i} -v^{-1} [2] [B_i,B_0]_v \bK_i\Big).
		\end{align}
	\end{lemma}
	
	The proof of Lemma \ref{lem:Bi1} is long and can be found in the Appendix~\ref{App:B}.
	
	\begin{lemma}
		\label{lem:THi1fixed}
		We have $\TT_{\bome}(\TH_{i,1})=\TH_{i,1}$, for $i=1,2$.
	\end{lemma}
	
	\begin{proof}
		Recall from \eqref{TH} that
		$\TH_{i,1} = -o(i)\big(\big[ B_i,[ B_{\tau i},B_0 ]_v\big]_{v^{2}}- v B_0 \bK_i\big)$.
		It suffices to check that
		\begin{align}
			\TT_1^{-1} \TT_0^{-1} \big(\big[ B_1,[ B_2 ,B_0 ]_v\big]_{v^{2}}- v B_0 \bK_1\big)= \big[ B_1,[ B_2 ,B_0 ]_v\big]_{v^{2}}- v B_0 \bK_1.
		\end{align}
		Since $B_{1,1}=o(1)\TT_1^{-1}\TT_0^{-1}(B_1)$ by definition, we have
		by Lemma~\ref{lem:Bi1} that
		\begin{align*}
			\TT_1^{-1}\TT_0^{-1}(B_1)=\Big( \big[[B_1,B_2]_v,[B_1,B_0]_v\big]_v-v[B_1,B_0]_{v^{3}}\bK_2 -v^{-1} [2] [B_1,B_0]_v \bK_1\Big).
		\end{align*}
		Then we compute
		\begin{align*}
			&\TT_1^{-1} \TT_0^{-1}  \Big(\big[ B_1,[ B_2 ,B_0 ]_v\big]_{v^{2}}\Big)
			\\
			&= \TT_1^{-1} \Big( \big[ \TT_0^{-1}(B_1),[[B_0,B_2]_v, B_0\K_0^{-1} ]_v \big]_{v^2} \Big)
			\\
			&=\TT_1^{-1} \Big( \big[ \TT_0^{-1}(B_1),[[B_0,B_2]_v, B_0 ]_v \big]_{v^2} \Big) \K_0^{-1}
			\\
			&=\TT_1^{-1} \Big( \big[ \TT_0^{-1}(B_1), B_2\K_0 \big]_{v^2} \Big) \K_0^{-1}
			\\
			&=\TT_1^{-1} \Big( \big[ \TT_0^{-1}(B_1), B_2 \big]_{v^2} \Big)
			\\
			&=\Big[\big[[B_1,B_2]_v,[B_1,B_0]_v\big]_v-v[B_1,B_0]_{v^{3}}\bK_2 -v^{-1} [2] [B_1,B_0]_v \bK_1,-vB_2\K_2^{-1}\Big]_{v^2}
			\\
			&=-v\Big[\big[[B_1,B_2]_v,[B_1,B_0]_v\big]_v- v[B_1,B_0]_{v^{3}}\bK_2 -v^{-1} [2] [B_1,B_0]_v \bK_1,B_2\Big]_{v^{-1}}\K_2^{-1}
			\\
			&=\Big[B_2,\big[[B_1,B_2]_v,[B_1,B_0]_v\big]_v-v[B_1,B_0]_{v^{3}}\bK_2 -v^{-1} [2] [B_1,B_0]_v \bK_1\Big]_{v}\K_2^{-1}
			\\
			&=\Big[B_2,\big[[B_1,B_2]_v,[B_1,B_0]_v\big]_v\Big]_v \K_2^{-1}
            \\
            &\quad -  v\big[ B_2, [B_1,B_0]_{v^{3}}\bK_2\big]_v \K_2^{-1}- v^{-1}[2] \big[B_2,[B_1,B_0]_v \bK_1\big]_{v}\K_2^{-1}
			\\
			&=\Big[B_2,\big[[B_1,B_2]_v,[B_1,B_0]_v\big]_v\Big]_v \K_2^{-1}- v\big[ B_2, [B_1,B_0]_{v^{3}}\big]_{v^{-2}} - v^{-1}[2] \big[B_2,[B_1,B_0]_v\big]_{v^4}\bK_1\K_2^{-1}.
		\end{align*}
		
		Also, by the formula \eqref{eq:T-1B0} for $\TT_1^{-1}(B_0)$, we have
		\begin{align*}
			-v	\TT_1^{-1} \TT_0^{-1}(B_0\K_1)&= -\TT_1^{-1}(B_0) \K_2^{-1}
			\\
			&=-v\Big[[B_1,B_2]_v,\big[ B_2,[B_1,B_0]_v\big] \Big] \bK_2^{-1}+ \big[ [B_1,B_2]_{v^3},B_0 \big]-vB_0\bK_1.
		\end{align*}
		Summing the above 2 computations and using the Serre relation \eqref{B122}, we have
		\begin{align*}
			&\TT_1^{-1} \TT_0^{-1} \big(\big[ B_1,[ B_2 ,B_0 ]_v\big]_{v^{2}}- v B_0 \bK_1\big)
			\\
			& =\Big[B_2,\big[[B_1,B_2]_v,[B_1,B_0]_v\big]_v\Big]_v \K_2^{-1}- v\big[ B_2, [B_1,B_0]_{v^{3}}\big]_{v^{-2}} - v^{-1}[2] \big[B_2,[B_1,B_0]_v\big]_{v^4}\bK_1\K_2^{-1}
			\\
			& \quad -v\Big[[B_1,B_2]_v,\big[ B_2,[B_1,B_0]_v\big] \Big] \bK_2^{-1}+ \big[ [B_1,B_2]_{v^3},B_0 \big]-vB_0\bK_1
			\\
			&=-[2]\Big[ [B_1,B_0]_v, B_2\bK_2+v^3B_2\bK_1 \Big]_{v^{-1}}\bK_2^{-1} - v\big[ B_2, [B_1,B_0]_{v^{3}}\big]_{v^{-2}}
			\\
			& \quad -v^{-1}[2] \big[B_2,[B_1,B_0]_v\big]_{v^4} \bK_1\K_2^{-1}
			+ \big[ [B_1,B_2]_{v^3},B_0 \big]-vB_0\bK_1
			\\
			&=-[2] \big[ [B_1,B_0]_v,B_2\big]_{v^2} -v\big[ B_2, [B_1,B_0]_{v^3}\big]_{v^{-2}} + \big[ [B_1,B_2]_{v^3},B_0 \big]-vB_0\bK_1
			\\
			&=\big[ B_1,[ B_2 ,B_0 ]_v\big]_{v^{2}}-v B_0 \bK_1.
		\end{align*}
		The lemma is proved.
	\end{proof}
	
	We prove a very special $m=1$ case of \eqref{qsiA1DR3reform} (equivalent to Relation \eqref{qsiA1DR3}) below.
	
	\begin{lemma}
		\label{lem:THi1Bj}
		For $i=1,2$, and $l\in\Z$, we have
		\begin{align}
			\label{eq:TH11B1}
			[\TH_{i,1},B_{j,l}]=[c_{ij}]B_{j,l+1}-[c_{\tau i,j}]B_{j,l-1}C.
		\end{align}
	\end{lemma}
	
	\begin{proof}
		We only need to consider $i=1$ thanks to the symmetry $\btau$.
		
		First, assume $j=1$. The identity \eqref{eq:TH11B1} for $l=0$ reads as
		\begin{equation}
			\label{eq:TH11Bl=1}
			[\TH_{1,1},B_{1}]=[2]B_{1,1}+B_{1,-1}C.
		\end{equation}
		Then \eqref{eq:TH11B1} for a general $l$ follows by applying the $(-l)$th power of $\TT_{\bome}$ to \eqref{eq:TH11Bl=1} since $\TH_{i,1}$ is fixed by $\TT_{\bome}$ (see Lemma \ref{lem:THi1fixed}). Hence, it remains to prove the identity \eqref{eq:TH11Bl=1}.
		
		To that end, we compute
		\begin{align*}
			[\TH_{1,1},B_1]= &- o(1)\Big[\big[ B_1,[ B_2 ,B_0 ]_v\big]_{v^{2}}-v B_0 \bK_1,B_1\Big]
			\\
			&=-o(1)\Big[\big[ B_1,[ B_2 ,B_0 ]_v\big]_{v^{2}},B_1\Big]+ vo(1) [B_0 \bK_1,B_1]
			\\
			&=-o(1)\Big[\big[ B_1,[ B_2 ,B_0 ]_v\big]_{v^{2}},B_1\Big]+ v^{-2}o(1) [B_0,B_1]_{v^3}\bK_1.
		\end{align*}
		On the other hand, we compute
		\begin{align*}
			&\Big[\big[ B_1,[ B_2 ,B_0 ]_v\big]_{v^{2}},B_1\Big]
			\\
			&=-v^2\Big[ B_2,\big[ B_1,[B_1,B_0]_v\big]_{v^{-1}} \Big]_{v^{-1}} + \big[[B_1,B_0]_v,[B_2,B_1]_v \big]_v
			\\
			& \quad +v^3\Big[ B_0,\big[ B_1,[B_1,B_2]_v\big]_{v^{-1}} \Big]_{v^{-3}} -v^{-1} \big[[B_1,B_2]_v,[B_1,B_0]_v \big]_{v^{3}}
			\\
			&=-[2]\big[[B_1,B_2]_v,[B_1,B_0]_v \big]_{v}-v[2][B_0,B_1]_{v^{-3} }\bK_1 -v^4[2] [B_0,B_1]_{v^{-3}} \bK_2.
		\end{align*}
		Combining the above 2 computations, we obtain
		\begin{align*}
			[\TH_{1,1},B_1]
			&=[2] o(1)\big[[B_1,B_2]_v,[B_1,B_0]_v \big]_{v} +[2]vo(1) [B_0,B_1]_{v^{-3} }\bK_1\\
			&\quad +[2]v^{4}o(1) [B_0,B_1]_{v^{-3}} \bK_2+ v^{-2}o(1) [B_0,B_1]_{v^3}\bK_1.
		\end{align*}
		This formula can then be converted to the desired identity \eqref{eq:TH11Bl=1} by using \eqref{B1-1} and \eqref{eq:B11}.

		Now assume $j=2$. As above, the proof of \eqref{eq:TH11B1} for a general $l$ follows from the case with $l=0$; that is, it remains to prove
		\begin{equation}
			\label{eq:TH11B2l=1}
			[\TH_{1,1},B_{2}]=-B_{2,1}-[2]B_{2,-1}C=-B_{2,1}+v[2]B_{2,-1}\K_\de.
		\end{equation}
		Indeed, we have
		\begin{align*}
			[\TH_{1,1},B_2]= &-o(1)\Big[\big[ B_1,[ B_2 ,B_0 ]_v\big]_{v^{2}}-v B_0 \bK_1,B_2\Big]
			\\
			&=-o(1)\Big[\big[ B_1,[ B_2 ,B_0 ]_v\big]_{v^{2}},B_2\Big] +vo(1) [B_0 \bK_1,B_2]
			\\
			&=-o(1)\Big[\big[ B_1,[ B_2 ,B_0 ]_v\big]_{v^{2}},B_2\Big]+v^4o(1) [B_0,B_2]_{v^{-3}}\bK_1
			\\
			&=o(1)\Big[[ B_2,B_1]_v,[ B_2 ,B_0 ]_v\Big]_{v}+ v^{4}o(1) [B_0,B_2]_{v^{-3}}\bK_1.
		\end{align*}
		This formula can be converted to \eqref{eq:TH11B2l=1} by using \eqref{B1-1}, \eqref{eq:B11} and $o(1)=-o(2)$.
	\end{proof}

	\subsection{Relation \eqref{qsiA1DR4}}
	
	\begin{lemma}
		\label{BB1}
		For $i\in \{1,2\}$ and $k\in\Z$, we have
		$[B_{i,k-1},B_{i,k}]_{v^{-2}}=0.$
	\end{lemma}
	
	\begin{proof}
		We note that the formula in the lemma for different $k$ are equivalent by applying suitable powers of $\TT_{\bome}$.
		
		It remains to prove the desired formula for $k=1$.
		By the formula \eqref{B1-1} for $B_{i,-1}$, we have
		\begin{align*}
			-v^{2}o(1)\cdot [B_{i,-1},B_i]_{v^{-2}}&= [B_i,B_0]_v  \bK_i B_i\bK_\delta^{-1}-v^{-2} B_i[B_i,B_0]_v \bK_i\bK_\delta^{-1}
			\\
			&=v^{-3} \big[[B_i,B_0]_v,B_i\big]_v\bK_i\bK_\delta^{-1}
			=0,
		\end{align*}
		where the last equality follows by the Serre relation \eqref{B0ii}.
	\end{proof}

	\begin{proposition}
		\label{prop:BBcom}
		Relation \eqref{qsiA1DR4} holds in $\tUi$, i.e.,
		\begin{equation*}
			[B_{i,k},B_{i,l+1}]_{v^{-2}}+[B_{i,l},B_{i,k+1}]_{v^{-2}}=0,
		\end{equation*}
		for $i\in \{1,2\}$, and $k,l\in\Z$.
	\end{proposition}
	
	\begin{proof}
		Without loss of generality, we assume that $i=1$.
		It follows by \eqref{eq:TH11B1} that
		\begin{equation}
			\label{eq:ThThB}
			\big[[2]\TH_{1,1}+\TH_{2,1},B_{1,k}\big]=[3]B_{1,k+1}.
		\end{equation}
		Denote $\bF_{k}=[B_{1,k},B_{1, 1}]_{v^{-2}}+[B_{1},B_{1,k+1}]_{v^{-2}}$, for $k\in\Z$. Then we have
		\begin{equation}\label{BB2}
			\big[[2]\TH_{1,1}+\TH_{2,1}, \bF_{k}\big]=[3](\bF_{k+1}+\TT_{\bome}^{-1}\bF_{k-1}).
		\end{equation}
		In particular, we have
		$\big[[2]\TH_{1,1}+\TH_{2,1}, \bF_{0}\big]= [3] (\bF_{1} +\TT_{\bome}^{-1}\bF_{-1}).$
		Since $\bF_0=0$ by Lemma \ref{BB1} and $\bF_1=\TT_{\bome}^{-1}\bF_{ -1}$, we obtain $\bF_1=\bF_{-1}=0$. It then follows by \eqref{BB2} and an induction on $k$ that $\bF_k=0 $, for all $k\in \Z$.
		
		The desired relation now follows by applying a suitable power of $\TT_{\bome}$ to $\bF_{k-l} =0$.
	\end{proof}

	\subsection{Partial commutativity and $\TT_{\omega}$-invariance  of $\TH_{i,n}$}
	
	Recall $D_{i,n}$ and $\TH_{i,n}$ defined in \eqref{Dn} and \eqref{THn}, respectively.
	
	\begin{lemma}
		\label{propTH1}
		For $i=1,2$, we have
		\begin{align}
			D_{i,-1}=-[B_{\tau i},B_{i,-1}]_{v^{-1}}-[B_i, B_{\tau i,-1}]_{v^{-1}}&=-v^{-1}(\TH_{\tau i,1}\bK_i+ \TH_{i,1} \bK_{\tau i})C^{-1}.
			\label{root0}
		\end{align}
	\end{lemma}
	
	\begin{proof}
		We assume $i=1$. By \eqref{B1-1} and \eqref{TH}, we have
		\begin{align}
			\label{pf1}
			-[B_1,B_{2,-1}]_{v^{-1}} &=-v^{-1}\TH_{1,1}\bK_2C^{-1}-o(1)v^{-1}B_0 \bK_0^{-1},\\
			-[B_2,B_{1,-1}]_{v^{-1}} &=-v^{-1}\TH_{2,1}\bK_1C^{-1}-o(2)v^{-1}B_0 \bK_0^{-1}.
			\notag
		\end{align}
		Hence, \eqref{root0} follows.
	\end{proof}
	
	We prove a very special $m=2$ case of \eqref{qsiA1DR3reform} (equivalent to Relation \eqref{qsiA1DR3}) below.
	\begin{lemma}
		\label{propTH2}
		For $i=1,2$, we have
		\begin{align*}
			[\TH_{i,2},B_i]+v^3[\TH_{i,0},B_i]_{v^{-6}}C=v^2[\TH_{i,1},B_{i,1}]_{v^{-4}}+v [\TH_{i,1},B_{i,-1}]_{v^{-2}}C.
		\end{align*}
	\end{lemma}
	
	\begin{proof}
		We assume $i=1$. We shall compute $[\TH_{1,2},B_1]$.
		Using $C= -v\bK_2\bK_1\bK_0$ from \eqref{eq:C}, recall $\TH_{1,2}$ from \eqref{eq:TH12}:
		\begin{align}
			\label{eq:Th12b}
			\TH_{1,2}&= -v^2 \Big([B_2,B_1]_{v^{-1}} + [B_{1,1},B_{2,-1}]_{v^{-1}}\Big) \bK_1\bK_0 +v\TH_{1,0}C -\TH_{2,0}C\bK_2^{-1}\bK_1.
		\end{align}
		The main part in $[\TH_{1,2},B_1]$ is
		\begin{align}\notag
			-v^2 \big[[B_{1,1},B_{2,-1}]_{v^{-1}}&\bK_1\bK_0,B_1\big]
			=v^2 \big[B_1,[B_{1,1},B_{2,-1}]_{v^{-1}}\big]_{v^{-3}}\bK_1\bK_0\\\notag
			&=v^2 \big[[B_1,B_{1,1}]_{v^{-2}},B_{2,-1}\big]_{v^{-2}}\bK_1\bK_0
			+ \big[B_{1,1},[B_1,B_{2,-1}]_{v^{-1}}\big]_v\bK_1\bK_0\\
			&=v^2 \big[\TH_{1,1},B_{1,1}\big]_{v^{-4}}+ v^{-1}o(1)[B_{1,1},B_0]_v\bK_1, \label{root1}
		\end{align}
		where we used $[B_1,B_{1,1}]_{v^{-2}}=0$ (by Lemma \ref{BB1}) and \eqref{pf1} in the last step.
		
		Next, we compute the term $[B_{1,1},B_0]_v$ in \eqref{root1}. Recall the formula for $B_{i,1}$ from \eqref{eq:B11}:
		\begin{align}  \label{Bi1}
			B_{i,1}= o(i)\Big( \big[[B_i,B_{\tau i}]_v,[B_{i},B_0]_v\big]_v-v[B_i,B_0]_{v^{3}}\bK_{\tau i}-v^{-1} [2] [B_i,B_0]_v \bK_i\Big).
		\end{align}
		Then the main part in $[B_{1,1},B_0]_v$ is
		\begin{align*}
			&o(1)\Big[\big[[B_1,B_2]_v,[B_1,B_0]_v\big]_v ,B_0\Big]_v \\
			&=o(1)\Big[[B_1,B_2]_v,\big[[B_1,B_0]_v,B_0\big]_{v^{-1}}\Big]_{v^3}-vo(1)\Big[[B_1,B_0]_v,\big[[B_1,B_2]_v,B_0\big]_{v^2}\Big]_{v^{-2}}\\
			&=o(1)\Big[[B_1,B_2]_v,\big[[B_1,B_0]_v,B_0\big]_{v^{-1}}\Big]_{v^3}\\
			&\quad -vo(1)\Big[[B_1,B_0]_v,\big[B_1,[B_2,B_0]_v\big]_{v^2}\Big]_{v^{-2}}+v^2o(1)\Big[[B_1,B_0]_v,\big[B_2,[B_1,B_0]_v\big] \Big]_{v^{-2}}\\
			&=-o(1)v^{-1}\big[[B_1,B_2]_v,B_1\big]_{v^3}\bK_0 \\
			&\quad + v^{-1}\big[\TH_{1,1}+vo(1)B_0\bK_1,[B_1,B_0]_v\big]_{v^2}+\big[\TH_{2,1}-vo(1)B_0\bK_2,[B_1,B_0]_v\big]\\
			&=-o(1)\Big(v^{-1}\big[[B_1,B_2]_v,B_1\big]_{v^3}\bK_0-\big[B_0\bK_1,[B_1,B_0]_v\big]_{v^2}+v\big[B_0\bK_2,[B_1,B_0]_v\big]\Big)\\
			&\quad +o(1)v^2\Big(v^{-1}\big[\TH_{1,1}, B_{1,-1}\big]_{v^2}+\big[\TH_{2,1}, B_{1,-1}\big]\Big)\bK_2 \bK_0.
		\end{align*}
		Plugging this new formula into the computation of $[B_{1,1},B_0]_v$ via the formula \eqref{Bi1}, we obtain
  \begin{align*}
      &-v^{-1}{o(1)} [B_{1,1},B_0]_v\bK_1 + v[\TH_{1,1},B_{1,-1}]_{v^{-2}}C
			\\
   &=\Big(v^{-2}\big[[B_1,B_2]_v,B_1\big]_{v^3}\bK_0\bK_1 -v^{-1}\big[B_0\bK_1,[B_1,B_0]_v\big]_{v^2} \bK_1+\big[B_0\bK_2,[B_1,B_0]_v\big]\bK_1\Big)\\
	&\quad +\Big(v^{-1}\big[\TH_{1,1}, B_{1,-1}\big]_{v^2}+\big[\TH_{2,1}, B_{1,-1}\big]\Big)C
 \\
 &\quad + \big[[B_1,B_0]_{v^{3}}\bK_{2},B_0\big]_v\bK_1 +v^{-2} [2] \big[[B_1,B_0]_v \bK_1,B_0\big]_v\bK_1
 + v[\TH_{1,1},B_{1,-1}]_{v^{-2}}C.
  \end{align*}
This can be simplified by collecting the like terms together:
		\begin{align}
  \notag
			-v^{-1} & {o(1)}  [B_{1,1},B_0]_v\bK_1 + v[\TH_{1,1},B_{1,-1}]_{v^{-2}}C\\
   \notag	&=\big[[2]\TH_{1,1}+\TH_{2,1},B_{1,-1}\big]C +v^{- 2}\big[[B_1,B_2]_v,B_1\big]_{v^3}\bK_1\bK_0
			\\\notag
			&\quad +v^{-1}[3](B_1B_0^2-[2]B_0 B_1 B_0 + B_0^2 B_1)\bK_1^2
			\\
			&=[3]B_1  C +v^{- 2}\big[[B_1,B_2]_v,B_1\big]_{v^3}\bK_1\bK_0 -v^{-2}[3] B_1 \bK_1^2 \bK_0,
			\label{root2}
		\end{align}
		where we have used \eqref{eq:ThThB} and the Serre relation \eqref{Bi00} in the last equality.
		Arranging the pieces \eqref{eq:Th12b}, \eqref{root1} and \eqref{root2} together, we have
		\begin{align*}
			[\TH_{1,2},B_1] &+v^3[\TH_{1,0},B_1]_{v^{-6}}C-v^2[\TH_{1,1},B_{1,1}]_{v^{-4}}-v[\TH_{1,1},B_{1,-1}]_{v^{-2}}C
			\\
			&=-v^2\big[[B_2,B_1]_{v^{-1}}\bK_1\bK_0,B_1\big]+\frac{v}{v-v^{-1}} [\bK_1^2\bK_0, B_1]
			\\
			& \quad  -v^{-2}\big[[B_1,B_2]_v,B_1\big]_{v^3}\bK_1\bK_0 +v^{-2}[3] B_1 \bK_1^2 \bK_0 \\
			& =0.
		\end{align*}
		The lemma is proved.
	\end{proof}
	
	\begin{lemma}
		\label{lem:TH1THkcom1}
		Let $n\geq 1$, and $i,j\in \{1,2\}$. Assume that
		\[
		\begin{cases}
			\TT_{\bome}(\TH_{i,m})=\TH_{i,m} , \qquad \text{ for } m\leq n , \\
			[\TH_{j,1},\TH_{i,m}] =0 ,\qquad \text{ for } m< n .
		\end{cases}
		\]
		Then $[\TH_{j,1},\TH_{i,n}]=[c_{ij}](1- \TT_{\bome})\TH_{i,n+1}$.
	\end{lemma}
	
	\begin{proof}
		By definition of $D_{i,k}$ and using Lemma \ref{lem:THi1Bj}, we have
		\begin{align*}
			&[\TH_{i,1},D_{i,n-2 }]\\
			&=[B_{\tau i,1},B_{i,n-2 }]_{v^{-1}}+[2][B_{\tau i,-1},B_{i,n-2}]_{v^{-1}}C-[2][B_{\tau i},B_{i,n-1}]_{v^{-1}}-[B_{\tau i},B_{i,n-3}]_{v^{-1}}C
			\\
			& \quad+[B_{ i,n-1},B_{\tau i }]_{v^{-1}}+[2][B_{i,n-1},B_{\tau i, -2}]_{v^{-1}}C-[2][B_{ i,n},B_{\tau i, -1}]_{v^{-1}}-[B_{i,n-2},B_{\tau i,-1}]_{v^{-1}}C
			\\
			&=[2](1-C \TT_{\bome})D_{i,n-1}-\TT_{{\bome}}^{-1}(1-C \TT_{\bome}) D_{i,n-3},
		\end{align*}
		for any $n\geq0$. By the assumption $\TT_{\bome}(\TH_{i,m})=\TH_{i,m}$ for $m\leq n$, we have
		\begin{align*}
			(1-C \TT_{\bome})D_{i,n-1}&=(1-C \TT_{\bome})\big(-v^{-1}\TH_{i,n+1}\bK_{\tau i}C^{-1}+\TH_{i,n-1}\bK_{\tau i} \big)
			\\
			&=-v^{-1}(1-\TT_{\bome})\TH_{i,n+1} \bK_{\tau i}C^{-1},
			\\
			(1-C \TT_{\bome}) D_{i,n-3}&=(1-C \TT_{\bome}) \big(-v^{-1} \TH_{i,n-1}\bK_{\tau i}C^{-1}+\TH_{i,n-3}\bK_{\tau i} \big)=0.
		\end{align*}
		So
		\begin{align*}
			[\TH_{i,1},D_{i,n-2 }]=-v^{-1}(1-\TT_{\bome})\TH_{i,n+1} \bK_{\tau i}C^{-1}.
		\end{align*}
		Replace $D_{i,n-2}$ by $-v^{-1}\TH_{i,n}\bK_{\tau i}C^{-1} +\TH_{i,n-2}\bK_{\tau i}$. By the assumption $[\TH_{i,1},\TH_{i,n-2}]=0$, we have
		\begin{align*}
			[\TH_{i,1},\TH_{i,n }]\bK_{\tau i}C^{-1} =[2](1-\TT_{\bome})\TH_{i,n+1} \bK_{\tau i}C^{-1},
		\end{align*}
		and then the desired formula follows.
		
		For $j=\tau i$, the proof is entirely similar by using now the identity
		\begin{align*}
			[\TH_{\tau i,1},D_{i,n-2}]=-(1-C \TT_{\bome})D_{i,n-1}+[2]\TT_{{\bome}}^{-1}(1-C \TT_{\bome}) D_{i,n-3}.
		\end{align*}
		We omit the details.
	\end{proof}
	
	The following corollary will be used repeatedly in the subsequent inductive arguments.
	\begin{corollary}
		\label{cor:fix}
		Let $n\geq 1$ and $i, j\in \{1,2\}$. Assume that $[\TH_{j,1},\TH_{i,m}] =0$, for all $m<n$. Then $ \TH_{i,m}$  are fixed by $\TT_{\bome}$, for $1\leq m \leq n$.
	\end{corollary}
	
	\begin{proof}
		The proof follows from Lemmas \ref{lem:THi1fixed}, \ref{lem:TH1THkcom1} and an induction on $m$.
	\end{proof}
	\subsection{The commutator $[\TH_{i,n}, B_j]$}
	
	For $i=1,2$ and $n\geq 0$, we denote by $\ch_{i,n}$ the $\Q(v)$-subalgebra of $\tUi$ generated by $\{\TH_{i,m},\K_1,\K_2,\K_0\mid 1\leq m < n\}$.
	
	\begin{proposition}
		\label{lem:ind}
		Let $n\geq1$, and assume that $[\TH_{j,1},\TH_{i,m}] =0$ for all $i, j \in \{1,2\}$ and all $m<n$. Then there exist $X_{k,n}^{(i)} \in \ch_{i,n}$, for $-n\le k \le n$, such that
		\begin{align}
			[\TH_{i,n},B_1]=\sum_{k=-n}^n B_{1,k} X_{k,n}^{(i)}.
			\label{eq:Th1B1}
		\end{align}
	\end{proposition}
	
	In order to prove Proposition \ref{lem:ind}, we need to prepare some notations and lemmas. For any $k_1,k_2\in\Z$ and $i=1,2$, we denote
	\begin{align}
		\bR(k_1,k_2|l;i)& :=\Sym_{k_1,k_2}\Big[ -v^{-1} \TH_{\tau i,l-k_2+1}C^{k_2}\bK_{i}+\TH_{\tau i,l-k_2-1}C^{k_2+1}\bK_{i}  , B_{i,k_1}\Big]_v
		\notag \\
		&\quad + \Sym_{k_1,k_2}\Big[-v^{-1}\TH_{i,k_2-l +1}C^l\bK_{\tau i} +\TH_{i,k_2 -l-1}C^{l+1}\bK_{\tau i}, B_{i,k_1}\Big]_v,
		\label{eq:RR} \\
		\bP(k_1,k_2|l;i)& :=\Sym_{k_1,k_2}\Big[ B_{i,k_1}, -v^{-1}\TH_{\tau i,l-k_2+1}C^{k_2-1}\bK_{i}+\TH_{\tau i,l-k_2-1}C^{k_2}\bK_{i} \Big]_v
		\notag \\
		&\quad +\Sym_{k_1,k_2}\Big[B_{i,k_1}, -v^{-1}\TH_{i,k_2-l +1}C^{l-1}\bK_{\tau i} +\TH_{i,k_2 -l-1}C^{l}\bK_{\tau i} \Big]_v.
		\label{eq:PP}
	\end{align}
	
	Recall $\bS(k_1,k_2|l;i)$ from \eqref{eq:SS}.
	\begin{lemma}
		Let $n\geq 3$, and $i\in \{1,2\}$. Assume that $[\TH_{i,1},\TH_{i,m}]=[\TH_{\tau i,1},\TH_{i,m}]=0$ for all $m<n$. Then, for $|k_1-l|\leq n-1,|k_2-l|\leq n-1$,
		\begin{align}\label{SR}
			&\bS(k_1+1,k_2|l;i)+ \bS(k_1,k_2+1|l;i)-[2] \bS(k_1,k_2|l+1;i)=[2]\bR(k_1,k_2|l;i),\\
			&\bS(k_1-1,k_2|l;i)+ \bS(k_1,k_2-1|l;i)-[2] \bS(k_1,k_2|l-1;i)=[2]\bP(k_1,k_2|l;i)\label{SP}.
		\end{align}
	\end{lemma}
	
	\begin{proof}
		If $[\TH_{i,1},\TH_{i,m}]=[\TH_{\tau i,1},\TH_{i,m}]=0$ for all $m<n$, then  $\TH_{1,m},\TH_{2,m}, 1\leq m\leq n$ are fixed by $\TT_{\bome}$ by Corollary~ \ref{cor:fix}. Thus, by applying $\TT_{\bome} $ to the identity \eqref{THn} we obtain the relation \eqref{qsiA1DR5} for $|l-k|\leq n-1$.
		
		Below we shall work only with $i=1$. For $|k_1-l|\leq n-1,|k_2-l|\leq n-1$, we compute
		\begin{align*}
			&\bS(k_1+1,k_2|l;1)+\bS(k_1,k_2+1|l;1)\\\notag
			&=\Sym_{k_1,k_2}\Big(B_{1,k_1+1}B_{1,k_2}B_{2,l}-[2]B_{1,k_1+1} B_{2,l} B_{1,k_2} + B_{2,l} B_{1,k_1+1}B_{1,k_2}\Big)\\\notag
			& \quad +\Sym_{k_1,k_2}\Big( B_{1,k_1}B_{1,k_2+1}B_{2,l}-[2]B_{1,k_1} B_{2,l} B_{1,k_2+1} + B_{2,l} B_{1,k_1}B_{1,k_2+1}\Big)\\\notag
			&=\Sym_{k_1,k_2}\Big(v^2 B_{1,k_2}B_{1,k_1+1}B_{2,l}-[2]B_{1,k_1+1} B_{2,l} B_{1,k_2} + B_{2,l} B_{1,k_1+1}B_{1,k_2}\Big)\\\notag
			& \quad +\Sym_{k_1,k_2}\Big( B_{1,k_1}B_{1,k_2+1}B_{2,l}-[2]B_{1,k_1} B_{2,l} B_{1,k_2+1} +v^{-2} B_{2,l}B_{1,k_2+1}B_{1,k_1}\Big)\\\notag
			&=[2]\Sym_{k_1,k_2}\big(v B_{1,k_1}B_{1,k_2+1}B_{2,l}- B_{1,k_1} B_{2,l} B_{1,k_2+1}-B_{1,k_1+1} B_{2,l} B_{1,k_2} + v^{-1}B_{2,l} B_{1,k_1+1}B_{1,k_2}\big)\\
			&=[2]\Sym_{k_1,k_2}\big(v B_{1,k_1} [B_{1,k_2+1},B_{2,l}]_{v^{-1}} -[B_{1,k_1+1},B_{2,l}]_{v^{-1}}B_{1,k_2}\big),
		\end{align*}
		where the second equality follows from
		$\Sym_{k_1,k_2} [B_{1,k_1},B_{1,k_2+1}]_{v^{-2}}=0$; see Proposition \ref{prop:BBcom}.

		Hence, we have
		\begin{align}
			\label{eq:SSSmi}
			&\bS(k_1+1,k_2|l;1)+\bS(k_1,k_2+1|l;1)-[2]\bS(k_1,k_2|l+1;1)
			\\\notag
			&=[2]\Sym_{k_1,k_2}\Big(v B_{1,k_1} [B_{1,k_2+1},B_{2,l}]_{v^{-1}} -[B_{1,k_1+1},B_{2,l}]_{v^{-1}}B_{1,k_2}\Big)
			\\\notag
			& \; +[2]\Sym_{k_1,k_2}\Big(v B_{1,k_1} [B_{2,l+1},B_{1,k_2}]_{v^{-1}} -[B_{2,l+1},B_{1,k_1}]_{v^{-1}}B_{1,k_2}\Big).
		\end{align}
		Applying the relation \eqref{qsiA1DR5} for $|l-k|\leq n-1$ (which holds as shown in the first paragraph of this proof) to sum up each of the 2 columns of the right hand side of \eqref{eq:SSSmi}, we have
		\begin{align}
			\notag
		&\bS(k_1+1,k_2|l;1)+\bS(k_1,k_2+1|l;1)-[2]\bS(k_1,k_2|l+1;1)
			\\\notag
			&=[2]\Sym_{k_1,k_2}\Big[  -v^{-1}\TH_{2,l-k_2+1}C^{k_2}\bK_{1}+\TH_{2,l-k_2-1}C^{k_2+1}\bK_{1}  , B_{1,k_1}\Big]_v
			\\
			& \quad +[2]\Sym_{k_1,k_2}\Big[-v^{-1}\TH_{1,k_2-l +1}C^l\bK_{2} +\TH_{1,k_2 -l-1}C^{l+1}\bK_{2}, B_{1,k_1}\Big]_v\label{Serre2}
			\\\notag
			&=[2]\bR(k_1,k_2|l;1).
		\end{align}
		This proves the identity \eqref{SR}.
		
		The proof of the remaining identity \eqref{SP} is similar and will be omitted here.
	\end{proof}
	
	\begin{lemma}
		Let $n\geq 3$ and assume that $[\TH_{i,1},\TH_{i,m}]=[\TH_{\tau i,1},\TH_{i,m}]=0$ for any $i=1,2$, and  all $m<n$. We have, for $|k_1-l|\leq n-2,|k_2-l|\leq n-2$,
		\begin{align}
			\label{RRRPPP}
			\bR(k_1-1,k_2|l;i) & +\bR(k_1,k_2-1 |l;i)-[2]\bR(k_1,k_2|l-1;i)\\\notag
			&=\bP(k_1+1,k_2|l;i)+\bP(k_1,k_2+1 |l;i)-[2]\bP(k_1,k_2|l+1;i).
		\end{align}
		In particular, for $|k-l|\leq n-2$, we have
		\begin{align}
			\label{RRPP}
			\bR(k -1,k |l) & +\bR(k ,k-1 |l)-[2]\bR(k ,k |l-1)\\\notag
			&=\bP(k +1,k |l)+\bP(k ,k +1 |l)-[2]\bP(k ,k |l+1).
		\end{align}
	\end{lemma}
	
	\begin{proof}
		The proof follows by replacing both $\bR$-terms and $\bP$-terms in \eqref{RRRPPP} by the $\bS$-terms via \eqref{SR}--\eqref{SP}.
	\end{proof}
	
	For $m\geq 0,k\in\Z$, we define
	\begin{align}
		\label{def:Y1ps}
		Y_{1,m}^k & :=
		[\TH_{1,m},B_{1,k}]+v^3[\TH_{1,m-2},B_{1,k}]_{v^{-6}}C-v^2[\TH_{1,m-1},B_{1,k+1}]_{v^{-4}}-v [\TH_{1,m-1},B_{1,k-1}]_{v^{-2}}C,
		\\
		\label{def:Y2ps}
		Y_{2,m}^k & :=
		[\TH_{2,m},B_{1,k}]+v^{-3}[\TH_{2,m-2},B_{1,k}]_{v^{6}}C-v^{-1}[\TH_{2,m-1},B_{1,k+1}]_{v^{2}}- v^{-2}[\TH_{2,m-1},B_{1,k-1}]_{v^{4}}C.
	\end{align}
	(Note that the desired relation \eqref{qsiA1DR3} can be formulated as $Y_{1,m}^k =0 =Y_{2,m}^k$, by Lemma~ \ref{lem:relationsreform}.)
	For convenience, we set $Y_{1,-m}^l=0=Y_{2,-m}^l$ for $m>0,l\in\Z$.
	We rewrite \eqref{RRPP} as
	\begin{align}
		\label{YYY}
		\begin{split}
			(Y_{2,l-k}^k C-v^{-1}Y^k_{2,l-k+2})C^{k-1}\bK_1 +(v^3 Y^k_{1,k-l}C - v^2 Y^k_{1,k-l+2})C^{l-1}\bK_2 &=0,
			\\
			\text{ for }|k-l| & \leq n-2.
		\end{split}
	\end{align}
	
	\begin{lemma}
		\label{lem:THB}
		Let $n\geq1$, and assume that $[\TH_{j,1},\TH_{i,m}] =0$, for $i,j =1,2$ and all $m<n$. Then we have
		\begin{equation}
			\label{Ys=0}
			Y_{1,n}^k=0,\qquad Y_{2,n}^k=0,\quad \forall k\in\Z.
		\end{equation}
	\end{lemma}
	
	\begin{proof}
		Since $\TH_{i,m},\TH_{\tau i,m}, 1\leq m \leq n$ are fixed by $\TT_{\bome}$ by Corollary \ref{cor:fix}, $\TT_{\bome}$ acts on $Y_{i,m}^k$ by shifting the superscript $k$. Hence, for fixed $1\leq m\leq n$, if $Y_{i,m}^k=0$ for some $k$, then $Y_{i,m}^k=0$ for all $k \in \Z.$
		
		Denote $Y_{i,m}=Y_{i,m}^0$. It remains to show by induction on $n$ that $Y_{i,n}=0$ for $i=1,2$. This holds for $n=1,2$ by Lemma~ \ref{lem:THi1Bj} and Lemma~ \ref{propTH2}. Let $n\geq3$. Applying a suitable power of $\TT_{\bome}$ to \eqref{YYY} gives us the identity
		\begin{equation}\label{YY}
			(Y_{2,l-k} C-v^{-1}Y_{2,l-k+2})C^{k-1}\bK_1 +(v^3 Y_{1,k-l} C - v^2 Y_{1,k-l+2})C^{l-1}\bK_2 =0,
		\end{equation}
		for $|k-l|\leq n-2$. For $n= 3$, set $k-l=1$ in \eqref{YY}. Note $Y_{i,p}=0$ for $p\leq 0$. Then
		\begin{align*}
			Y_{1,3}\bK_2=-v^{-3}Y_{2,1} C \bK_1 +v  Y_{1,1} C \bK_2=0.
		\end{align*}
		Similarly, we have $Y_{2,3}\bK_1=vY_{2,1} C \bK_1 -v^3  Y_{1,1} C \bK_2=0.$
		Let $n\geq 4$. Setting $k -l=n-2$ in \eqref{YY} gives us
		$Y_{1,n}= v Y_{1,n-2} C$, while setting $l -k=n-2$ in \eqref{YY} gives us $Y_{2,n}= v Y_{2,n-2} C.$
		Therefore, by induction on $n$, we have proved $Y_{i,n}=0$ for $i=1,2$.
	\end{proof}
	
	\begin{proof}[Proof of Proposition \ref{lem:ind}]
		For the sake of notational simplicity, we shall work with $i=1$ in this proof (the case for $i=2$ is entirely similar).

		We proceed by an induction on $n$.
		The statement for $n=1,2$ follows by Lemma~\ref{lem:THi1Bj} and Lemma~\ref{propTH2}. Suppose that the statement holds for $[\TH_{1,m},B_1]$, where $1\le m \le n-1.$ By Lemma \ref{lem:THB} (see \eqref{def:Y1ps} for notation $Y_{i,n}^k$), we have
		\begin{align}
			[\TH_{1,n},B_1]&=-v^3[\TH_{1,n-2},B_1]_{v^{-6}}C+v^2[\TH_{1,n-1},B_{1,1}]_{v^{-4}} +v[\TH_{1,n-1},B_{1,-1}]_{v^{-2}}C
			\notag \\
			&=-v^3\TH_{1,n-2}B_1C+ v^{-3}B_1\TH_{1,n-2}C +v^2\TH_{1,n-1}B_{1,1}-v^{-2}B_{1,1} \TH_{1,n-1}
			\notag \\
			& \quad +v\TH_{1,n-1} B_{1,-1}C -v^{-1} B_{1,-1}\TH_{1,n-1}C
			\notag \\
			&=-v^3\TH_{1,n-2}B_1C +v^2\TH_{1,n-1}B_{1,1} +v\TH_{1,n-1} B_{1,-1}C +\sum_{k=-n}^n B_{1,k}Z_{k,n}
			\label{eq:XTB}
		\end{align}
		for some $Z_{k,n}\in \ch_{1,n}$. Furthermore, by the inductive assumption, we have
		\begin{align}
			\TH_{1,n-2}B_1C &= B_1\TH_{1,n-2}C + \sum_{k=-n+2}^{n-2} B_{1,k}X'_{k,n-2},
			\notag \\
			\label{eq:TH1n-1B1}
			\TH_{1,n-1}B_{1} &= B_{1}\TH_{1,n-1} + \sum_{k=-n+1}^{n-1} B_{1,k} X''_{k,n-1},
		\end{align}
		for some $X'_{k,n-2}\in \ch_{1,n-2}$ and $X''_{k,n-1}\in \ch_{1,n-1}$. By applying $\TT_{\bome}^{\pm1}$ to \eqref{eq:TH1n-1B1}, we have
		\begin{align*}
			\TH_{1,n-1}B_{1,1}= &B_{1,1}\TH_{1,n-1} + \sum_{k=-n+1}^{n-1} B_{1,k+1} \Ddot{X}'_{k,n-1},
			\\
			\TH_{1,n-1}B_{1,-1}C&= B_{1,-1}\TH_{1,n-1}C + \sum_{k=-n+1}^{n-1} B_{1,k-1} \Ddot{X}''_{k,n-1},
		\end{align*}
		for some $\Ddot{X}'_{k,n-1}, \Ddot{X}''_{k,n-1} \in \ch_{1,n-1}$,
		since $\TT_{\bome}(\TH_{1,n-1})= \TH_{1,n-1}$ and $\TT_{\bome}(\ch_{1,n-1})\subseteq \ch_{1,n-1}$ by Corollary \ref{cor:fix}. Thus the desired formula \eqref{eq:Th1B1} follows from \eqref{eq:XTB}.
	\end{proof}

	\subsection{Relation \eqref{qsiA1DR2}}
	
	With all the technical preparation in the prior subsections, we are ready to prove the crucial commutativity among the imaginary root vectors $\TH_{i,m}$.

	%
	
	
	
	\begin{lemma}
		\label{prop:commTHTH}
		Let $n\geq1$, and assume that $[\TH_{j,1},\TH_{i,k}] =0$ for all $i,j \in \{1,2\}$, and all $k<n$. Then for all $1\leq m_1,m_2 \leq n$ and $i,j \in \{1,2\}$, we have
		\begin{enumerate}
			\item $\big[\TH_{i,m_1},[\TH_{j,m_2},B_1]\big]=\big[\TH_{j,m_2},[\TH_{i,m_1},B_1]\big]$;
			\item
			$[\TH_{i,m_1},\TH_{j,m_2}] =0.$
		\end{enumerate}
	\end{lemma}
	
	\begin{proof}
		Without loss of generality, we assume $m_1 \ge m_2$.
		Let ``$<$'' denote the lexicographic ordering on $\mathbb{N}^2$. We use induction on $(m_1,m_2)$ in lexicographic ordering. Assume that
		\begin{equation}
			\label{indhyp}
			[\TH_{i,k},\TH_{j,l}] =0
		\end{equation}
		for all $i,j$ and all $(k,l)<(m_1,m_2)$. By Corollary \ref{cor:fix}, $ \TH_{i,k},\TH_{j,l} $ are fixed by $\TT_{{\bome}}$. By Proposition \ref{lem:ind},
		there exist $X_{k,m}^{(i)} \in \ch_{i,m}$ such that
		\begin{align}
			[\TH_{i,m},B_{1,r}]=\sum_{k=-m}^m B_{1,r+k} X_{k,m}^{(i)},
		\end{align}
		for $1\leq m \leq n,r\in \Z$. Hence,
		\begin{align*}
			\big[\TH_{i,m_1},[\TH_{j,m_2},B_1]\big]
			&=\sum_{l=-m_2}^{m_2}\big[\TH_{i,m_1}, B_{1,l} X_{l,m_2}^{(j)} \big]
			\\
			&=\sum_{l=-m_2}^{m_2}\big[\TH_{i,m_1}, B_{1,l} \big] X_{l,m_2}^{(j)}
			\\
			& =\sum_{l=-m_2}^{m_2}\sum_{k=-m_1}^{m_1}  B_{1,k+l} X_{k,m_1}^{(i)}X_{l,m_2}^{(j)}
			\\
			&=\sum_{k=-m_1}^{m_1}\sum_{l=-m_2}^{m_2}  B_{1,k+l} X_{l,m_2}^{(j)}X_{k,m_1}^{(i)}
			=\big[\TH_{j,m_2},[\TH_{i,m_1},B_1]\big],
		\end{align*}
		where the second equality follows from $[\TH_{i,m_1}, X_{l,m_2}]=0$ by the inductive assumption \eqref{indhyp}, and the commutativity of $X_{k,m_1},X_{l,m_2}$ used in the fourth equality also follows from \eqref{indhyp} and $X_{k,m_1}\in \ch_{1,m_1},X_{l,m_2}\in\ch_{1,m_2}$. This proves (1).
		
		It follows by the Jacobi identity that $\big[[\TH_{i,m_1},\TH_{j,m_2}\big], B_1] =0$. By symmetry, we also have $\big[[\TH_{i,m_1},\TH_{j,m_2}\big], B_2] =0$. By Lemma~\ref{lem:KiX} and Lemma~ \ref{lem:degree root}, we have $\big[[\TH_{i,m_1},\TH_{j,m_2}],\bK_a \big]=0$, for all $a\in\I$. Therefore, Lemma~ \ref{lem:X0} is applicable and implies $[\TH_{i,m_1},\TH_{j,m_2}] =0$, which proves (2).
	\end{proof}
	
	\begin{proposition}
		\label{prop:verfDR2}
		Relation \eqref{qsiA1DR2} holds in $\tUi$.
	\end{proposition}
	
	\begin{proof}
		Since $\TH_{i,0}= \frac{1}{v-v^{-1}}$, we have $[\TH_{j,1},\TH_{i,0}]=0$, for all $i,j$. Then it follows from Lemma~ \ref{prop:commTHTH} that $[\TH_{i,m_1},\TH_{j,m_2}] =0$ by induction. So \eqref{qsiA1DR2} holds in $\tUi$ by Lemma~ \ref{lem:relationsreform}(1).
	\end{proof}
	
	Now we can remove the assumption in Corollary \ref{cor:fix}.
	
	\begin{theorem}
		\label{thm:fix1}
		We have $\TT_{\bome}(\TH_{i,n})=\TH_{i,n}$, for all $n\geq1$ and $i\in \{1,2\}$.
	\end{theorem}
	
	\begin{proof}
		The proof follows by Corollary \ref{cor:fix} and Proposition \ref{prop:verfDR2}.
	\end{proof}

	\subsection{Relation \eqref{qsiA1DR3}}
	\label{subsec:5.3}
	
	\begin{proposition}
		\label{prop:qsiA1DR3}
		Relation \eqref{qsiA1DR3}  holds in $\tUi$.
	\end{proposition}
	
	\begin{proof}
		We shall prove the identity \eqref{qsiA1DR3reform}, which is equivalent to \eqref{qsiA1DR3} by Lemma~ \ref{lem:relationsreform}(2).
		
		By Proposition \ref{prop:verfDR2}, the assumption of Lemma \ref{lem:THB} holds for any $n\geq1$, and hence
		\begin{equation*}
			Y_{1,n}^k=0,\qquad Y_{2,n}^k=0,\quad \forall n\geq 1, \,k\in\Z.
		\end{equation*}
		That is, the identity \eqref{qsiA1DR3reform} holds for $j=1$; it holds also for $j=2$ by the symmetry $\btau$ in \eqref{Phi}.
	\end{proof}

	\subsection{Relation \eqref{qsiA1DR5} }
	
	The proof of Relation \eqref{qsiA1DR5} is rather straightforward as we now have Theorem~\ref{thm:fix1} at our disposal.
	
	\begin{proposition}
		\label{prop:verfDR5}
		Relation \eqref{qsiA1DR5} holds in $\tUi$.
	\end{proposition}
	
	\begin{proof}
		In the proof below, we refer to the relation \eqref{qsiA1DR5} as $R(k,l)$.
		
		Relation $R(k,k)$ follows by applying a suitable power of $\TT_{\bome}$ to the identity \eqref{root0}, thanks to the $\TT_{\bome}$-invariance of $\Theta_{i,m}$ by Theorem~\ref{thm:fix1}.
		
		Relation $R(k,-1)$ for $k\geq0$ follows by definition of $\TH_{i,m}$ in \eqref{TH}--\eqref{THn}. Then $R(k,l)$ for $k>l$ follows by applying $\TT_{\bome}$ to the known case $R(k-l-1, -1)$ and noting that $\TH_{i,m}$ is $\TT_{\bome}$-invariant by Theorem~\ref{thm:fix1}.
		
		For $k<l$, we apply the involution $\btau$  in \eqref{Phi} and reduce to the case $R(l, k)$ already established above	(the second equation below follows by a simple rewriting of the commutators):
		\begin{align*}
			\btau\big( \text{LHS}\eqref{qsiA1DR5} \big)
			&= [B_{\tau i,k},B_{i,l+1}]_v-v[B_{\tau i,k+1},B_{i,l}]_{v^{-1}}
			\\
			&= [B_{ i,l},B_{\tau i,k+1}]_v-v[B_{ i,l+1},B_{\tau i,k}]_{v^{-1}}
			\\
			&=-\Theta_{{\tau i},k-l+1}C^l \bK_{i} +v \Theta_{ {\tau i},k-l-1}C^{l+1}\bK_{i}
			-\Theta_{i,l-k+1}C^k\bK_{{\tau i}} +v \Theta_{i,l-k-1}C^{k+1}\bK_{\tau i}.
		\end{align*}
		Another application of the involution $\btau$ to the right hand side above turns it into the right hand side of \eqref{qsiA1DR5}. This completes the proof of the proposition.
	\end{proof}

	\subsection{Relation \eqref{qsiA1DR6}}
	
	Finally we shall prove the Drinfeld type Serre relations.
	
	\begin{proposition}
		\label{prop:verfqsiA1DR6}
		Relation \eqref{qsiA1DR6} holds in $\tUi$.
	\end{proposition}
	\begin{proof}
		Since \eqref{qsiA1DR6} is equivalent to the relation \eqref{qsiA1DRG6} in generating function form, it suffices to prove \eqref{qsiA1DRG6}.
		
		We can rewrite \eqref{SR} as
		\begin{align}
			\label{Serre3'}
			(w_1^{-1}+&w_2^{-1}-[2]z^{-1})\mathbb{S}(w_1,w_2\mid z;i)=-\frac{v^{-1}[2]}{{v-v^{-1}}}\Sym_{w_1,w_2}\bDel
			(w_2 z)\\
			\notag
			&\times \Big((z^{-1}-vw_2^{-1})[\bTH_{\tau i}(z)\bK_i,\bB_i(w_1)]_v + (w_2^{-1}-vz^{-1})[\bTH_i(w_2)\bK_{\tau i},\bB_i(w_1)]_v\Big).
		\end{align}
		Similarly, we rewrite \eqref{SP} as
		\begin{align}
			\label{Serre4'}
			(w_1+&w_2-[2]z)\mathbb{S}(w_1,w_2\mid z;i)=-\frac{v^{-1}[2]}{{v-v^{-1}}}\Sym_{w_1,w_2}\bDel(w_2 z)\\
			\notag
			&\times \Big(( w_2-vz)[\bB_i(w_1),\bTH_{\tau i}(z)\bK_i]_v + ( z-v w_2)[\bB_i(w_1),\bTH_i(w_2)\bK_{\tau i}]_v\Big).
		\end{align}
	   
    We calculate \eqref{Serre3'}$\times [2]z$ $+$\eqref{Serre4'}$\times (w_1^{-1}+w_2^{-1})$ as follows:
   \begin{align}\notag
   &(w_1-v^2w_2)(w_2^{-1}-v^{-2}w_1^{-1})\mathbb{S}(w_1,w_2\mid z;i)\\\notag
   =&-\frac{v^{-1}[2]}{{v-v^{-1}}}\Sym_{w_1,w_2}\bDel(w_2 z)(w_1^{-1}-v^2 w_2^{-1})(w_2-v z) \bB_i(w_1)\bTH_{\tau i}(z)\bK_i\\  \label{Serre7}
    &-\frac{v^{-1}[2]}{{v-v^{-1}}}\Sym_{w_1,w_2}\bDel(w_2 z)v^{-1}(w_1-v^2 w_2)(w_1^{-1}-v w_1^{-1} w_2^{-1}z)\bTH_{\tau i}(z)\bK_i\bB_i(w_1)
    \\\notag
    &+\frac{v[2]}{{v-v^{-1}}}\Sym_{w_1,w_2}\bDel(w_2 z)(w_2^{-1}-v^{-2}w_1^{-1})(z-v w_2)\bB_i(w_1)\bTH_{i}(w_2)\bK_{\tau i}
    \\\notag
    &+\frac{v^{-1}[2]}{{v-v^{-1}}}\Sym_{w_1,w_2}\bDel(w_2 z)v^{-1}(w_1-v^{2}w_2)(v w_1^{-1}- w_1^{-1} w_2^{-1} z)\bTH_{i}(w_2)\bK_{\tau i}\bB_i(w_1).
   \end{align}
   
   Note that the constant component of \eqref{Serre7} is exactly the relations \eqref{B211}--\eqref{B122}. Then, dividing both sides of \eqref{Serre7} by $(w_1-v^2w_2)(w_2^{-1}-v^{-2}w_1^{-1})$, we obtain the desired relation \eqref{qsiA1DRG6}.
	\end{proof}
	
	
	Therefore, we have verified that all the relations \eqref{qsiA1DR1}--\eqref{qsiA1DR6} are preserved under $\Phi: \tUiD \rightarrow \tUi$; that is,  $\Phi: \tUiD \rightarrow \tUi$ defined in \eqref{eq:isom} is a homomorphism.
	
	\appendix

	\section{Proof of Identity \eqref{eq:tqcI2'}} 
	
	\label{App:AB}

	\subsection{Proof of Identity \eqref{eq:tqcI2'}}
	\label{App:A}
	
	In this subsection, we prove the identity \eqref{eq:tqcI2'}, and hence complete the proof of the identity \eqref{eq:interB0}. When combining with discussions in \S\ref{subsec:steps}--\S\ref{subsec:S3a}, we have completed the proof of Theorem~\ref{thm:T-1}.
	
	%
	
	The identity $\Ddot{\mathfrak{F}} =\tT_{\bs_1}^{-1}(F_0 )$ in \eqref{eq:tqcI2'} is not homogeneous in $\tU$, and it consists of homogeneous summands of 4 different weights: $-(\alpha_0 + 2\alpha_1 + 2\alpha_2)$, $-(\alpha_0+\alpha_1 +\alpha_2)$, $-\alpha_0$, and $-\alpha_0 + \alpha_1 +\alpha_2$, respectively. Accordingly, the proof of the identity \eqref{eq:tqcI2'} is reduced to the proofs of the 4 identities \eqref{eq:tqcI4}--\eqref{eq:tqcI7} below:
	
	\begin{align}\label{eq:tqcI4}
		&v\Big[ \big[F_1,F_2\big]_v, \Big[F_2,[F_1,F_0]_v\Big]\Big]=\tT_{\bs_1}^{-1}(F_0 ),
	\end{align}
	\begin{align}\notag
		&\Big[ \big[F_1,E_1 K_2\big]_v, \big[F_2,[F_1,F_0]_v\big]\Big]+\Big[ \big[F_1,F_2\big]_v, \big[E_1 K_2,[F_1,F_0]_v\big]\Big]\\\label{eq:tqcI5}
		&\quad + \textstyle v \Big[ \frac{K_1 K_2 -K_1 K_2'}{v-v^{-1}}, \big[F_2,[F_1,F_0]_v\big]\Big]=\big[ [F_1 , F_2 ]_{v^3},F_0 \big]K_1'K_2,
	\end{align}
	\begin{align}\notag
		&\Big[ \big[E_2 K_1,E_1 K_2\big]_v, \big[F_2,[F_1,F_0]_v\big]\Big]+\Big[ [F_1,E_1 K_2 ]_v , \big[E_1 K_2,[F_1,F_0]_v\big]\Big]\\\notag
		&\quad +\textstyle v\Big[\frac{K_1 K_2 -K_1 K_2'}{v-v^{-1}},\big[E_1 K_2,[F_1,F_0]_v\big]\Big]\\
		\label{eq:tqcI6}
		&= \big[ [  E_2 K_1, F_2 ]_{v^3} +[F_1, E_1 K_2]_{v^3},F_0 \big]K_2 K_1' - v^2 F_0 K_1 K_2 K_1' K_2',
	\end{align}
	\begin{align}
		\label{eq:tqcI7}
		&\Big[ \big[E_2 K_1,E_1 K_2\big]_v , \big[E_1 K_2,[F_1,F_0]_v\big]\Big]=\big[ [ E_2 K_1,  E_1 K_2]_{v^3},F_0 \big]K_2 K_1'.
	\end{align}
	The identity \eqref{eq:tqcI4} holds by Lemma \ref{lem:tqcI0}. The proofs for \eqref{eq:tqcI5}-\eqref{eq:tqcI7} require only elementary computations and are given in the following 3 subsections.

	\subsubsection{Proof of \eqref{eq:tqcI5}}
	
	Denote $X:=[F_1,F_0]_v$.
	\begin{lemma}\label{lem:tqcI7}
		We have
		\begin{itemize}
			\item[(1)] $X F_1 = v F_1 X$;
			\item[(2)] $ \big[E_1, X\big]=F_0 K_1'$;
			\item[(3)] $ X E_1 K_2= v^{-2} E_1 K_2 X-  F_0 K_1' K_2$.
		\end{itemize}
	\end{lemma}
	
	Note the third term on the left hand side of \eqref{eq:tqcI5} is $0$, i.e.,
	$\textstyle  \Big[ \frac{K_1 K_2 -K_1 K_2'}{v-v^{-1}}, \big[F_2,[F_1,F_0]_v\big]\Big]=0.$
	We now compute using Lemma \ref{lem:tqcI7} as follows:
	\begin{align*}
		&\big[ [F_1,E_1 K_2]_v,[F_2,X]\big]+\big[ [F_1,F_2]_v, [E_1 K_2,X]\big]
		\\
		&=\big[F_1,E_1 K_2\big]_v F_2 X - \big[F_1,E_1 K_2\big]_v X F_2 -F_2 X \big[F_1,E_1 K_2\big]_v + X F_2 \big[F_1,E_1 K_2\big]_v
		\\
		&\quad +[F_1,F_2]_v E_1 K_2 X - [F_1,F_2]_v X E_1 K_2- E_1 K_2 X[F_1,F_2]_v + X E_1 K_2[F_1,F_2]_v\\
		&=\Big( [F_1,E_1 K_2 ]_v F_2-v^{-1} F_2  [F_1,E_1 K_2 ]_v   + [F_1,F_2]_v E_1 K_2 - v^{-2}[F_1,F_2]_v E_1 K_2\Big)X\\
		&\quad +v F_2 F_1 F_0 K_1' K_2 -v^4 F_2 F_0 F_1 K_1'K_2+[F_1,F_2]_v  F_0 K_1' K_2\\
		&\quad +X\Big( F_2[F_1,E_1 K_2]_v-v\big[F_1,E_1 K_2\big]_v  F_2 + E_1 K_2 [F_1,F_2]_v - v^2 E_1 K_2 [F_1,F_2]_v \Big)\\
		&\quad -v^{-1}F_1 F_0 F_2 K_1' K_2 + v^2 F_0 F_1 F_2 K_1' K_2 -v^2 F_0 [F_1,F_2]_v K_1' K_2 \\
		&= - F_2 K_1' K_2 X+  F_1 F_2 F_0 K_1' K_2 -v^4 F_2 F_0 F_1 K_1'K_2 \\
		&\quad  + v^{-1} X F_2 K_1' K_2+v^3 F_0 F_2 F_1 K_1' K_2 -v^{-1}F_1 F_0 F_2 K_1' K_2 \\
		&= \big[ [F_1 , F_2 ]_{v^3},F_0 \big]K_1'K_2.
	\end{align*}
	This proves the identity \eqref{eq:tqcI5}.

	\subsubsection{Proof of \eqref{eq:tqcI6}}
	
	Denote again $X:=[F_1,F_0]_v$. Denote $\E_1:=E_1 K_2,\E_2:=E_2 K_1.$
	\begin{lemma}\label{lem:tqcI9}
		We have
		\begin{itemize}
			\item[(1)] $X \E_2 =v \E_2 X$;
			\item[(2)] $\big[\E_2,\E_1]_v,F_2\big]_{v^{-1}}=-v  \E_1 K_1 K_2$;
			\item[(3)] $\big[[\E_2,\E_1]_v ,X\big]_v=v^2[\E_2,F_0]_{v^3}K_1'K_2$;
			\item[(4)] $\textstyle[F_1,\E_1]_{v^{-1}}= \frac{K_1'-K_1}{v-v^{-1}}K_2$;
			\item[(5)] $\big[[F_1,\E_1]_v,X\big]_v=v^2 [F_1,F_0]_{v^3} K_1' K_2$.
		\end{itemize}
	\end{lemma}
	
	\begin{proof}
		Parts (1)-(2) and (4) follow by a direct computation. Parts (3) and (5) follow from (1), Lemma \ref{lem:tqcI7}, and the following two $v$-Jacobi identities, respectively:
		\begin{align*}
			\big[[\E_2,\E_1]_v, X \big]_v & = \big[\E_2,[\E_1,X]_{v^2}\big]-v\big[\E_1,[\E_2,X]_{v^{-1}}\big]_v,
			\\
			\big[[F_1,\E_1]_v, X\big]_v & = \big[F_1,[\E_1,X]_{v^2}\big]-v\big[ \E_1, [F_1,X]_{v^{-1}}\big]_v.
		\end{align*}
		The lemma is proved.
	\end{proof}
	
	\begin{proposition}\label{prop:tqcI4}
		We have
		\begin{align}\notag
			\big[[\E_2, & \E_1]_v, [F_2,X]\big]+\textstyle v\Big[\frac{K_1 K_2 -K_1 K_2'}{v-v^{-1}}, [\E_1,X ]\Big]
			\\
			&=\big[[\E_2,F_2]_{v^3},F_0\big] K_1' K_2 - v^2 F_0 K_1 K_2 K_1' K_2'.
			\label{eq:tqcI9}
		\end{align}
	\end{proposition}
	
	\begin{proof}
		By Lemma \ref{lem:tqcI9}(2)-(3), the first term on the left hand side of \eqref{eq:tqcI9} is equal to
		\begin{align*}
			\big[[\E_2,\E_1]_v, [F_2,X]\big]&=\Big[\big[[\E_2,\E_1]_v ,F_2 \big]_{v^{-1}},X\Big]_v+v^{-1}\Big[F_2, \big[[\E_2,\E_1]_v ,X\big]_{v }\Big]_v
			\\
			&=-v^2\Big[ \E_1 ,X\Big]  K_1 K_2 +v \Big[F_2, [\E_2,F_0]_{v^3}\Big]_{v^{-2}}K_1'K_2.
		\end{align*}
		Hence, the left hand side of \eqref{eq:tqcI9} is simplified as
		\begin{align*}
			&\big[[\E_2,\E_1]_v, [F_2,X]\big]+\textstyle v\Big[\frac{K_1 K_2 -K_1 K_2'}{v-v^{-1}}, [\E_1,X ]\Big]
			\\
			&=-v^2 [ \E_1 ,X ]  K_1 K_2+v \Big[F_2, [\E_2,F_0]_{v^3}\Big]_{v^{-2}}K_1'K_2+ v^2 [\E_1,X ]  K_1 K_2
			\\
			&=v \Big[F_2, [\E_2,F_0]_{v^3}\Big]_{v^{-2}}K_1'K_2
			\\
			&=\big[[\E_2,F_2]_{v^3},F_0\big] K_1' K_2 - v^2 F_0 K_1 K_2 K_1' K_2',
		\end{align*}
		as desired. Here the last equality follows by
		\begin{align*}
			& v \Big[F_2, [\E_2,F_0]_{v^3}\Big]_{v^{-2}}-\big[[\E_2,F_2]_{v^3},F_0\big]
			\\
			=&(v-v^3)F_2\E_2F_0+(v^2-1)F_0\E_2F_2-\E_2F_2F_0+v^3F_2\E_2F_0+F_0\E_2F_2-v^3F_0F_2\E_2
			\\
			=&v[F_2,E_2]K_1F_0+v^3F_0[E_2,F_2]K_1
			\\
			=&-v\frac{K_2-K_2'}{v-v^{-1}}K_1F_0+v_3F_0 \frac{K_2-K_2'}{v-v^{-1}} K_1
			\\
			=&-v^2F_0K_1K_2'.
		\end{align*}
		Hence, Proposition \ref{prop:tqcI4} is proved.
	\end{proof}
	
	\begin{proposition}\label{prop:tqcI5}
		We have
		\begin{align*}
			\Big[ [F_1,\E_1]_v , [\E_1,X ]\Big]=\big[ [F_1, E_1 K_2]_{v^3},F_0 \big] K_1' K_2.
		\end{align*}
	\end{proposition}
	
	\begin{proof}
		By Lemma \ref{lem:tqcI9}(4)-(5), we have
		\begin{align*}\notag
			\Big[ [F_1,\E_1]_v  ,  [\E_1,X ]\Big]
			&= \Big[\big[ [F_1,\E_1]_v ,\E_1\big]_{v^{-1}}, X \Big]_v +v^{-1} \Big[ \E_1, \big[[F_1,\E_1]_v, X\big]_v\Big]_v \\\notag
			&= \Big[\big[ [F_1,\E_1]_{v^{-1}},\E_1\big]_v , X \Big]_v +v^{-1} \Big[ \E_1, \big[F_1,[\E_1, X]\big]_{v^2}\Big]_v\\
			&= \Big[\big[ \frac{K_1'-K_1}{v-v^{-1}}K_2,\E_1\big]_v , X \Big]_v +v \big[ \E_1,[F_1,F_0]_{v^3} K_1' K_2\big]_v\\
			&= (v+v^{-1}) \big[ [F_1,F_0]_v,\E_1\big]_{v^2} K_1' K_2+v \big[ \E_1,[F_1,F_0]_{v^3}\big]_{v^{-2}} K_1' K_2
			\\
			&\stackrel{(*)}{=} -v(v+v^{-1}) \big[F_0, [F_1,\E_1]_{v^3}\big]_{v^{-2}} K_1' K_2- v^{-2} \big[ [F_1,\E_1]_{v^{3}},F_0\big]_{v^{4}} K_1' K_2\\
			&= \big[ [F_1, E_1 K_2]_{v^3},F_0 \big] K_1' K_2.
		\end{align*}
		where for the equation $(*)$ above we have used the following identities
		\begin{align*}
			\big[ [F_1,F_0]_v,\E_1\big]_{v^2}= -v\big[F_0, [F_1,\E_1]_{v^3}\big]_{v^{-2}},
			\quad
			\big[ \E_1,[F_1,F_0]_{v^3}\big]_{v^{-2}}=\big[ [\E_1,F_1]_{v^{-3}},F_0\big]_{v^{4}}.
		\end{align*}
		The proposition is proved.
	\end{proof}
	The identity \eqref{eq:tqcI6} follows by Proposition \ref{prop:tqcI4} and Proposition \ref{prop:tqcI5}.

	\subsubsection{Proof of \eqref{eq:tqcI7}}
	
	We continue to denote $\E_1:=E_1 K_2,\E_2:=E_2 K_1,$ and $X:=[F_1,F_0]_v.$
	
	\begin{lemma}\label{lem:tqcI8}
		We have
		\begin{itemize}
			\item[(1)] $[\E_1,F_1]_v=v\frac{K_1-K_1'}{v-v^{-1}}K_2$;
			\item[(2)] $[\E_1,F_0]_v=[\E_2,F_0]_v=0$;
			\item[(3)] $[\E_2,F_1]_{v^{-2}}=0$;
			\item[(4)] $\big[[\E_2,\E_1]_v,\E_1\big]_{v^{-1}}=0$;
			\item[(5)] $\big[\E_2, X\big]_{v^{-1}}=0$;
			\item[(6)] $\big[\E_1,X\big]_{v^2}=v^2 F_0 K_1' K_2$;
			\item[(7)] $ \big[[\E_2,\E_1]_v, X\big]_v= v^2\big[\E_2,F_0 K_1' K_2\big]$.
		\end{itemize}
	\end{lemma}
	
	\begin{proof}
		Parts (1)-(3) are clear, (4) is a Serre relation, and (5) follows from (1)-(3). Part~(6) follows from (1)-(3) and a $v$-Jacobi identity
		$\big[\E_1,[F_1,F_0]_v\big]_{v^2}=\big[[\E_1, F_1]_v , F_0\big]_{v^2}+v\big[ F_1,[\E_1,F_0]_v\big].$
		Part~(7) follows from (5)-(6) and
		$\big[[\E_2,\E_1]_v, X\big]_v=\big[\E_2, [\E_1, X ]_{v^2}\big]-v\big[\E_1, [\E_2, X ]_{v^{-1}}\big]_v.$
	\end{proof}
	
	We now compute the left hand side of \eqref{eq:tqcI7} by applying Lemma \ref{lem:tqcI8}:
	\begin{align*}
		\Big[ \big[\E_2,\E_1\big]_v , [\E_1 ,X]\Big]
		&=\Big[ \big[[ \E_2,\E_1]_v ,\E_1\big]_{v^{-1}} ,X\Big]_v + v^{-1}\Big[ \E_1 ,\big[[\E_2,\E_1]_v, X\big]_v\Big]_v\\
		&=v \Big[ \E_1 , \big[\E_2,F_0 K_1' K_2\big]\Big]_v\\
		&=v \Big[ \E_1 , \big[\E_2,F_0\big]_{v^3}\Big]_{v^{-2}} K_1' K_2\\
		&=v \Big(\E_1 \E_2 F_0-v^3\E_1 F_0\E_2- v^{-2}\E_2 F_0\E_1+v F_0 \E_2\E_1\Big)K_1' K_2\\
		&=v(v- v^{-1}) F_0\Big( - v^3\E_1\E_2+\E_2\E_1\Big)K_1' K_2\\
		&= (v^2- 1) F_0[\E_2,\E_1]_{v^3} K_1' K_2.
	\end{align*}
	On the other side, by Lemma \ref{lem:tqcI8}, we have  $[\E_2,\E_1]_{v^3} F_0 =v^2 F_0 [\E_2,\E_1]_{v^3}$, and thus
	\begin{align*}
		\big[[\E_2,\E_1]_{v^3},F_0 \big]=(v^2- 1) F_0[\E_2,\E_1]_{v^3}.
	\end{align*}
	Hence, \eqref{eq:tqcI7} follows from the above computations.

	\subsection{Proof of Lemma \ref{lem:Bi1}}
	\label{App:B}
	
	In this subsection, we shall prove Lemma \ref{lem:Bi1}.
	\subsubsection{A preparatory lemma}
	
	We denote
	$B_{i_1i_2\cdots i_r}:=B_{i_1}B_{i_2}\cdots B_{i_r}$.
	First, we prepare the following lemma.
	\begin{lemma}
		\label{lem:grade6}
		We have
		\begin{align*}
			&\Big[\big[ [B_1,B_2]_v,  [B_1,B_0]_v\big]_v,[B_1,B_2]_v\Big]_v
			\\
			&=v^3B_{0112}\K_1-(v^4-2)B_{0112}\K_2+v^2B_{0121}\K_1+v^5B_{0121}\K_2
			\\
			&+vB_{0211}\K_1+v^6B_{0211}\K_2-(2v^2+v^4)B_{1012}\K_1-(v+2v^{-1})B_{1012}\K_2
			\\
			&-(2v+v^{-1})B_{1021}\K_1-(2v^4+v^2)B_{1021}\K_2+v[3]B_{1102}\K_1+(v^2+2)B_{1102}\K_2
			\\
			&-B_{1120}\K_1-vB_{1120}\K_2+(v^2+2) B_{1201}\K_1+(v^5+2v^3)B_{1201}\K_2
			\\
			&-v^{-1}B_{1210}\K_1-v^2B_{1210}\K_2-v[2]B_{2011}\K_1-v^6[2]B_{2011}\K_2
			\\
			&+v^{-1}B_{2101}\K_1+v^8B_{2101}\K_2-(v^7-v^3)B_{2110}\K_2.
		\end{align*}
	\end{lemma}
	
	Denote
	\begin{align*}
	& X_1=B_{011122}, X_2=B_{012211}, X_3=B_{022111}, X_4=B_{102211}, X_5=B_{110221}, X_6=B_{111022},
	\\
	& X_7=B_{111202}, X_8=B_{112012}, X_9=B_{112021}, X_{10}=B_{120211}, X_{11}=B_{122110},
	\\
	& X_{12}=B_{201112},
	X_{13}=B_{202111}, X_{14}=B_{210211}, X_{15}=B_{211102}, X_{16}=B_{221110}.
	\end{align*}
	Then we have
	\begin{align*}
		B_{101122}&=\frac{[2]}{[3]}X_1+\frac{1}{[3]}X_{6},
		\\
		B_{110122}&=\frac{1}{[3]}X_1+\frac{[2]}{[3]}X_6,
		\\
		B_{110212}&=\frac{1}{[3]!}X_1+\frac{1}{[2]}X_5+\frac{1}{[3]}X_6+vB_{1102}\K_1+v^{-2}B_{1102}\K_2,
		\\
		B_{021211}&=\frac{1}{[2]}X_2+\frac{1}{[2]}X_3+v^4B_{0211}\K_2+v^{-5}B_{0211}\K_1,
		\\
		B_{220111}&=[2]X_{13}-X_{3},
		\\
		B_{221011}&=\frac{[2]^2}{[3]}X_{13}-\frac{[2]}{[3]}X_3+\frac{1}{[3]}X_{16},
		\\
		B_{211210}&=X_{11},
		\\
		B_{211120}&=\frac{[3]}{[2]}X_{11}-\frac{1}{[2]}X_{16}-v^{-3}[3]B_{2110}\bK_1-v^2[3]B_{2110}\bK_1,
		\\
		B_{212110}&=\frac{1}{[2]}X_{11}+\frac{1}{[2]}X_{16}+v^{-5}B_{2110}\bK_1+v^4B_{2110}\bK_2
		\\
		B_{021112}&=\frac{[3]}{[2]}X_{2}-\frac{1}{[2]}X_3-v^{-3}[3]B_{0211}\bK_1-v^2[3]B_{0211}\bK_2,
		\\
		B_{021121}&=X_2,
		\\
		B_{011212}&=\frac{[2]}{[3]}X_1+\frac{1}{[2]}X_2-\frac{1}{[3]!}X_3+[2]B_{0112}\K_1+v^{-1}[2]B_{0112}\K_2-v^{-3}B_{0211}\K_1-v^2B_{0211}\K_2,
		\\
		B_{012112}&=B_{011221}=\frac{1}{[3]}X_1+X_2-\frac{1}{[3]}X_3-v^{-3}[2]B_{0211}\K_1-v^2[2]B_{0211}\K_2\\
		&\qquad\qquad\qquad+v^{-1}[2]B_{0112}\K_1+[2]B_{0112}\K_2,
		\\
		B_{012121}&=\frac{1}{[3]!}X_1+\frac{2}{[2]}X_2-\frac{1}{[3]!}X_3+v^{-1}B_{0112}\K_1+B_{0112}\K_2+v^{-2}B_{0121}\K_1+vB_{0121}\K_2\\
		&\quad -v^{-3}B_{0211}\K_1-v^2B_{0211}\K_2,
		\\
		B_{101212}&=\frac{[2]^2+1}{[3]![2]}X_1+\frac{1}{[2]^2}X_2-\frac{1}{[3]![2]}X_3+\frac{1}{[2]^2}X_{5}+\frac{1}{[3]!}X_6+B_{0112}\K_1+v^{-1}B_{0112}\K_2\\
		& \quad +\frac{v}{[2]}B_{1102}\K_1+\frac{v^{-2}}{[2]}B_{1102}\K_2-\frac{v^{-3}}{[2]}B_{0211}\K_1-\frac{v^2}{[2]}B_{0211}\K_2,
		\\
		B_{101221}&=B_{102112}=\frac{1}{[3]!}X_1+\frac{1}{[2]}X_2-\frac{1}{[3]!}X_3+\frac{1}{[2]}X_{5}+[2]B_{0112}\K_1+v^{-1}[2]B_{0112}\K_2\\
		& \qquad\qquad\qquad -v^{-3}B_{0211}\K_1-v^2B_{0211}\K_2-v[2]B_{1012}\K_1-v^{-2}[2]B_{1012}\K_2
		\\
		& \qquad\qquad\qquad	+vB_{1102}\K_1+v^{-2}B_{1102}\K_2,
		\\
		B_{102121}&=\frac{1}{[3]![2]}X_1+\frac{1}{[2]^2}X_2-\frac{1}{[3]![2]}X_3+\frac{1}{[2]}X_4+\frac{1}{[2]^2}X_5+B_{0112}\K_1+v^{-1}B_{0112}\K_2\\
		&  \quad -\frac{v^{-3}}{[2]}B_{0211}\K_1-\frac{v^2}{[2]}B_{0211}\bK_2-vB_{1012}\K_1-v^{-2}B_{1012}\K_2+v^{-2}B_{1021}\K_1
		\\	
		& \quad
		+vB_{1021}\K_2
		+\frac{v}{[2]}B_{1102}\K_1
		+\frac{v^{-2}}{[2]}B_{1102}\K_2,
		\\
		B_{121102}&=\frac{[2]}{[3]}X_{15}+\frac{1}{[3]}X_{7}+v^{-1}[2]B_{1102}\K_1+[2]B_{1102}\K_2,
		\\
		B_{112102}&=\frac{1}{[3]}X_{15}+\frac{[2]}{[3]}X_7+[2]B_{1102}\K_1+v^{-1}[2]B_{1102}\K_2,
		\\
		B_{111220}&=[2]X_7-X_6,
		\\
		B_{112120}&=\frac{[2]^2}{[3]}X_7-\frac{[2]}{[3]}X_6+\frac{1}{[2]}X_{11}-\frac{1}{[3]!}X_{16}+[2]B_{1120}\K_1+v^{-1}[2]B_{1120}\K_2
		\\
		&\quad-v^{-3}B_{2110}\K_1-v^2B_{2110}\K_2,
		\\
		B_{121120}&=B_{112210}=\frac{[2]}{[3]}X_7-\frac{1}{[3]}X_6+X_{11}-\frac{1}{[3]}X_{16}-v^{-3}[2]B_{2110}\K_1-v^2[2]B_{2110}\K_2
		\\
		& \quad +v^{-1}[2]B_{1120}\K_1+[2]B_{1120}\K_2,
		\\
		B_{121210}&=\frac{1}{[3]}X_{7}-\frac{1}{[3]!}X_6+\frac{2}{[2]}X_{11}-\frac{1}{[3]!}X_{16}+v^{-1}B_{1120}\K_1+B_{1120}\K_2+v^{-2}B_{1210}\K_1
		\\
		& \quad+vB_{1210}\K_2 -v^{-3}B_{2110}\K_1-v^2B_{2110}\K_2,
		\\
		B_{122011}&=[2]X_{10}-X_4,
		\\
		B_{122101}&=B_{211201}=X_{10}-\frac{1}{[2]}X_4+\frac{1}{[2]}X_{11},
		\\
		B_{212011}&=\frac{[2]}{[3]}X_{13}-\frac{1}{[3]}X_3+\frac{1}{[3]!}X_{16}+X_{10}-\frac{1}{[2]}X_{4}+v^{-5}B_{2011}\K_1+v^4B_{2011}\K_2,
		\\
		B_{212101}&= \frac{1}{[3]}X_{13}-\frac{1}{[3]!}X_3+\frac{1}{[3]}X_{16}+\frac{1}{[2]}X_{10}-\frac{1}{[2]^2}X_{4}+\frac{1}{[2]^2}X_{11}+v^{-5}B_{2101}\K_1+v^4B_{2101}\K_2,
		\\
		B_{201121}&=\frac{1}{[3]}X_{13}+\frac{[2]}{[3]}X_{12}+v^{-3}[2]B_{2011}\K_1+v^2[2]B_{2011}\K_2,
		\\
		B_{201211}&=\frac{[2]}{[3]}X_{13}+\frac{1}{[3]}X_{12}+v^{-4}[2]B_{2011}\K_1+v^3[2]B_{2011}\K_2,
		\\
		B_{210112}&=\frac{[2]}{[3]}X_{12}+\frac{1}{[3]}X_{15},
		\\
		B_{211012}&=\frac{1}{[3]}X_{12}+\frac{[2]}{[3]}X_{15},
		\\
		B_{210121}&=\frac{1}{[3]}X_{12}+\frac{1}{[3]!}X_{15}+\frac{1}{[2]}X_{14}+v^{-2}B_{2101}\K_1+vB_{2101}\K_2,
		\\
		B_{211021}&=\frac{1}{[3]}X_{15}+X_{14}-\frac{1}{[3]}X_{13}-v^{-3}[2]B_{2011}\K_1-v^2[2]B_{2011}\K_2\\
		&\quad+v^{-2}[2]B_{2101}\K_1+v[2]B_{2101}\K_2,
		\\
		B_{112201}&=[2]X_{9}-X_{5},
		\\
		B_{121021}&=\frac{1}{[2]}X_9+\frac{1}{[3]!}X_{15}+\frac{1}{[2]}X_{14}-\frac{1}{[3]!}X_{13}+v^{-2}B_{1021}\K_1+vB_{1021}\K_2
		\\
		&\quad -v^{-3}B_{2011}\K_1-v^2B_{2011}\K_2+v^{-2}B_{2101}\K_1+vB_{2101}\K_2,
		\\
		B_{121201}&=X_{9}-\frac{1}{[2]}X_5+\frac{1}{[2]}X_{10}-\frac{1}{[2]^2}X_4+\frac{1}{[2]^2}X_{11}+v^{-2}B_{1201}\K_1+vB_{1201}\K_2,
		\\
		B_{121012}&=\frac{1}{[2]}X_8+\frac{1}{[3]!}X_{12}+\frac{1}{[3]}X_{15}+v^{-2}B_{1012}\K_1+vB_{1012}\K_2
		\\
		B_{120112}&=X_8+\frac{1}{[3]}X_{12}-\frac{1}{[3]}X_{7}+v^{-2}[2]B_{1012}\K_1+v[2]B_{1012}\K_2-v^{-1}[2]B_{1102}\K_1-[2]B_{1102}\K_2.
	\end{align*}
	Lemma~\ref{lem:grade6} follows now by a routine computation using the above formulas.

	\subsubsection{Completing the proof of Lemma \ref{lem:Bi1}}	
	
	By definition, we have
	\begin{align*}
		o(1)B_{1,1}\K_1
		&= \TT_1^{-1} \TT_0 ^{-1}(B_1)\K_1
		=-v[\TT_1^{-1}(B_0),B_1\bK_1^{-1}]_v\K_1
		\\
		&= -v\bigg[ \Big[B_2,\big[ [B_1,B_2]_v,[B_1,B_0]_v\big]_v\Big]_v   ,B_1\bigg]_v
		+v^{2} \Big[ \big[ B_2, [B_1,B_0]_{v^{3}}\big]_{v^{-2}} \bK_2 ,B_1\Big]_v
		\\
		&+v \Big[ \big[ B_1,[ B_2 ,B_0 ]_v\big]_{v^{2}} \bK_2,B_1\Big]_v
		+  [2] \Big[  \big[B_2,[B_1,B_0]_v\big]_{v^4}\bK_1  ,B_1\Big]_v
		\\
		&
		-v^{2} [ B_0,B_1]_v\bK_1\bK_2
		\\
		&= v^2 \bigg[ B_1,\Big[B_2,\big[ [B_1,B_2]_v,[B_1,B_0]_v\big]_v\Big]_v  \bigg]_{v^{-1}}
		+v^{5} \Big[ \big[ B_2, [B_1,B_0]_{v^{3}}\big]_{v^{-2}} ,B_1\Big]_{v^{-2}}\bK_2
		\\
		&+v^{4} \Big[ \big[ B_1,[ B_2 ,B_0 ]_v\big]_{v^{2}},B_1\Big]_{v^{-2}}\bK_2
		+   v^{-3}[2] \Big[  \big[B_2,[B_1,B_0]_v\big]_{v^4}  ,B_1\Big]_{v^4}\bK_1
		\\
		&
		-v^{2} [ B_0,B_1]_v\bK_1\bK_2.
	\end{align*}
	Moreover, we have
	\begin{align*}
		& v^2 \bigg[ B_1,\Big[B_2,\big[ [B_1,B_2]_v,[B_1,B_0]_v\big]_v\Big]_v  \bigg]_{v^{-1}}
		\\
		&=v^3\Big[B_2,\big[ B_1,[ [B_1,B_2]_v,[B_1,B_0]_v]_v\big]_{v^{-2}}\Big]
	-v\Big[\big[ [B_1,B_2]_v,  [B_1,B_0]_v\big]_v,[B_1,B_2]_v\Big]_v
		\\
		&= -v^{2} \Big[B_2,-v^{-1}[2] \big[ [B_1,B_0]_v, B_1\big]_{v^{-3}}\bK_1- v^2[2] \big[ [B_1,B_0]_v, B_1\big]_{v^{3}}\bK_2\Big]
		\\
		&-v \Big[\big[ [B_1,B_2]_v,  [B_1,B_0]_v\big]_v,[B_1,B_2]_v\Big]_v
		\\
		&=v [2]\Big[B_2,\big[ [B_1,B_0]_v, B_1\big]_{v^{-3}}\Big]_{v^3} \bK_1+v^{4}[2] \Big[B_2,\big[ [B_1,B_0]_v, B_1\big]_{v^{3}}\Big]_{v^{-3}}\bK_2
		\\
		&-v\Big[\big[ [B_1,B_2]_v,  [B_1,B_0]_v\big]_v,[B_1,B_2]_v\Big]_v.
	\end{align*}
	So it suffices to check that
	\begin{align*}
		&-v\Big[\big[ [B_1,B_2]_v,  [B_1,B_0]_v\big]_v,[B_1,B_2]_v\Big]_v
		\\\notag
		&=- v[2] \Big[ B_2,\big[[B_1,B_0]_v,B_1\big]_{v^{-3}} \Big]_{v^3} \bK_1
		-v^{-3}[2] \Big[  \big[B_2,[B_1,B_0]_v\big]_{v^4}  ,B_1\Big]_{v^4} \bK_1
		\\\notag
		&\big[[B_1,B_2]_v,[B_1,B_0]_v \big]_v\bK_1-v^4[2] \Big[B_2,\big[[B_1,B_0]_v,B_1\big]_{v^3} \Big]_{v^{-3}}\K_2-v^4\Big[\big[B_1,[B_2,B_0]_v\big]_{v^{2}},B_1\Big]_{v^{-2}}\K_2
		\\\notag
		&-v^5\Big[\big[B_2,[B_1,B_0]_{v^3}\big]_{v^{-2}},B_1\Big]_{v^{-2}}\K_2
		-v^2[2][B_1,B_0]_v\K_2\K_1-v^{-1}[2][B_1,B_0]_v\K_1^2.
	\end{align*}
	By using Lemma \ref{lem:grade6}, this last identity follows by comparing the coefficients with the help of
	\begin{align*}
		B_{0121}&=\frac{1}{[2]}B_{0112}+\frac{1}{[2]}B_{0211}+v^{-2}B_{01}\K_1+vB_{01}\K_2,
		\\
		B_{1210}&=\frac{1}{[2]}B_{1120}+\frac{1}{[2]}B_{2110}+v^{-2}B_{10}\K_1+vB_{10}\K_2,
		\\
		B_{1012}&=\frac{1}{[2]}B_{0112}+\frac{1}{[2]}B_{1102},\qquad\qquad
		B_{2101}=\frac{1}{[2]}B_{2110}+\frac{1}{[2]}B_{2011}.
	\end{align*}
	This completes the proof of Lemma \ref{lem:Bi1}.


\end{document}